\documentclass{article}
\pdfoutput=1

\usepackage[utf8]{inputenc}
\usepackage[american]{babel}
\usepackage{amsmath, amsfonts, amssymb, amsthm, dsfont}
\usepackage{array, multirow, hhline, tabularx, graphicx, wrapfig}
\usepackage{comment}
\usepackage{stmaryrd}
\usepackage{color}
\usepackage{geometry}
\usepackage{hyperref}
\usepackage{authblk}

\hypersetup{colorlinks=true, citecolor = blue}

\newcommand\Sym[1]{\mathfrak{S}_{#1}}
\newcommand\Rank[1]{\mathfrak{S}^{\prime}_{#1}}
\newcommand\Sn{\mathfrak{S}_{n}}

\newcommand\1[1]{\mathds{1}_{#1}}
\newcommand\n{\llbracket n \rrbracket}
\newcommand\nk{\llbracket k \rrbracket}

\newcommand{\supp}{\operatorname{supp}}

\newcommand{\Cycle}{\operatorname{Cycle}}

\newcommand{\cyc}{\operatorname{cyc}}

\newcommand{\Part}{\operatorname{Part}}

\theoremstyle{plain}
\newtheorem{theorem}{Theorem}
\newtheorem{corollary}{Corollary}
\newtheorem{lemma}{Lemma}
\newtheorem{proposition}{Proposition}
\theoremstyle{remark}
\newtheorem{example}{Example}
\newtheorem{remark}{Remark}
\theoremstyle{definition}
\newtheorem{definition}{Definition}
\newtheorem{algorithm}{Algorithm}

\begin{document}

\title{Multiresolution Analysis of Incomplete Rankings
\thanks{This work was supported by Agence Nationale de la Recherche (France) grant ANR-11-IDEX-0003-02.} }

\date{}

\author[1]{St\'ephan Cl\'emen\c{c}on}
\author[2]{J\'er\'emie Jakubowicz}
\author[1]{Eric Sibony\thanks{Corresponding author - email: eric.sibony@telecom-paristech.fr - postal address: Telecom ParisTech 37-39 rue Dareau, 75014 Paris, France.}}
\affil[1]{LTCI UMR No. 5141 Telecom ParisTech/CNRS\\ Institut Mines-Telecom, Paris, 75013, France}
\affil[2]{SAMOVAR UMR No. 5157 Telecom SudParis/CNRS\\ Institut Mines-Telecom, Paris, 75013, France}

\maketitle

\begin{abstract}
Incomplete rankings on a set of items $\{1,\; \ldots,\; n\}$ are orderings of the form $a_{1}\prec\dots\prec a_{k}$, with $\{a_{1},\dots a_{k}\}\subset\{1,\dots,n\}$ and $k < n$. Though they arise in many modern applications, only a few methods have been introduced to manipulate them, most of them consisting in representing any incomplete ranking by the set of all its possible linear extensions on $\{1,\; \ldots,\; n\}$. It is the major purpose of this paper to introduce a completely novel approach, which allows to treat incomplete rankings directly, representing them as injective words over $\{1,\; \ldots,\; n\}$. Unexpectedly, operations on incomplete rankings have very simple equivalents in this setting and the topological structure of the complex of injective words can be interpretated in a simple fashion from the perspective of ranking. We exploit this connection here and use recent results from algebraic topology to construct a multiresolution analysis and develop a wavelet framework for incomplete rankings. Though purely combinatorial, this construction relies on the same ideas underlying multiresolution analysis on a Euclidean space, and permits to localize the information related to rankings on each subset of items. It can be viewed as a crucial step toward nonlinear approximation of distributions of incomplete rankings and paves the way for many statistical applications, including preference data analysis and the design of recommender systems.\\

\end{abstract}

\section{Introduction}
\label{intro}

Data expressing rankings or preferences have become ubiquitous in the Big Data era. Operating continuously on still more content, modern applications such as recommendation systems and search engines generate and/or exploit massive data of this nature. The design of statistical machine-learning algorithms, tailored to this type of data, is crucial to optimize the performance of such systems (\textit{e.g.} rank documents by degree of relevance for a specific query in information retrieval, propose a sorted list of items/products to a prospect she/he is most liable to buy in e-commerce). A well studied situation is when raw data are of the form of ``full rankings'' on a given set of items indexed by $\n = \{1,\; \ldots,\; n\}$ and are then described by permutations $\sigma$ on $\n$ that map an item to its rank, $a_1 \prec \ldots \prec a_n$ with $a_i=\sigma^{-1}(i)$ for $i\in \n$. The variability of observations is represented by a probability distribution $p$ on the set $\Sn$ of all the permutations on $\n$, which can be seen as an element of the space
\[
L(\Sn) = \{ f : \Sn \rightarrow \mathbb{R}\},
\]
such that $p(\sigma)\geq 0$ for all $\sigma\in\Sn$ and $\sum_{\sigma\in \Sn}p(\sigma)= 1$.
Though empirical estimation of $p$ may appear as a problem of disarming simplicity at first glance,  it is actually a great statistical challenge because the number of possible rankings (\textit{i.e.} $\Sn$'s cardinality) explodes as $n!$ with the number of instances to be ranked. Traditional methods in machine-learning and statistics quickly become either intractable or inaccurate in practice and many approaches have been proposed these last few years to deal with preference data and overcome these challenges in different situations (\textit{e.g.} \cite{FISS03}, \cite{HFCB08}, \cite{MM09}, \cite{HW09}, \cite{LM08}, \cite{HG09}, \cite{Sun2012}). Whatever the type of task considered (supervised, unsupervised), machine-learning algorithms generally rest upon the computation of statistical quantities such as averages or medians, summarizing/representing efficiently the data or the performance of a decision rule candidate applied to the data. However, summarizing ranking variability is far from straightforward and extending simple concepts such as an average or a median in the context of preference data raises a certain number of deep mathematical and computational problems, see \cite{Bar81}, and call for new constructions.

One approach, much documented in the literature, consists in exploiting the algebraic structure of the (noncommutative) group $\Sn$ and perform a harmonic analysis on $L(\Sn$), see for example \cite{Diaconis1988}, \cite{Rockmore2002}, \cite{Kondor2010}, \cite{Jagabathula2011}, \cite{Kakarala2011}. This corresponds to a decomposition of the form
\[
L(\Sn) \cong \bigoplus_{\lambda}d_{\lambda}S^{\lambda},
\]
where the $S^{\lambda}$'s are irreducible spaces invariant under the action of the translations $f\mapsto f(\sigma_{0}^{-1}.)$ for all $\sigma_{0}\in\Sn$, the $\lambda$'s correspond to ``frequencies'', and the $d_{\lambda}$'s are positive integers. The sign $\cong$ above means that the two spaces are isomorphic, the spaces $S^{\lambda}$ being not necessarily subspaces of $L(\Sn)$. This decomposition allows to localize the different spectral components of any function $f\in L(\Sn)$. Furthermore, it is possible to define a (partial) order on the $\lambda$'s that indicates the different level of ``smoothness'' of the elements of the corresponding $S^{\lambda}$'s (for instance, the smoothest component is the space of constant functions), thus providing a natural framework for linear approximation in $L(\Sn)$, see \cite{HGG09} or \cite{Irurozki2011}. This framework also extends to the analysis of full rankings with ties, referred to as \textit{bucket orders} (or \textit{partial rankings} sometimes): for $1\leq r \leq n$ and $\mu=(\mu_1,\; \ldots,\; \mu_r)\in \mathbb{N}^{*r}$ such that $\mu_{1} +\dots + \mu_{r} = n$, orderings of the type $a_{1,1},\; \ldots,\; a_{1,\mu_1}\prec \ldots \prec a_{r,1},\; \ldots,\; a_{r,\mu_r}$ described by mappings $\sigma: \n \rightarrow \{1,\; \ldots,\; r\}$ such that $\sigma^{-1}(\{i\})=\{a_{i,1},\; \ldots,\; a_{i,\mu_i} \}$ for any $i\in\{1,\; \ldots, \; r\}$. Among bucket orders, top-$k$ rankings received special attention. They correspond to orderings of the form $a_{1} \prec \dots  \prec a_{k} \prec \textit{the rest}$. The same notion of translation can be defined on the space $M^{\mu}$ of real-valued functions on partial rankings with fixed form $\mu = (\mu_{1}, \dots, \mu_{r})$, leading to a similar decomposition, called Young's rule,
\[
M^{\mu} \cong \bigoplus_{\lambda}K_{\lambda,\mu}S^{\lambda},
\]
where the $S^{\lambda}$'s are the same as before and the $K_{\lambda,\mu}$'s are integers $\geq 0$ called the Kotska numbers.

This ``$\Sn$-based'' harmonic analysis is however not suited for the analysis of ranked data of the form  $a_{1}\prec \dots \prec a_{k}$ with $k < n$, \textit{i.e.} when the rankings do not involve all the items. Such data shall be referred to as \textit{incomplete rankings} throughout the article. Indeed, though \cite{Kondor2010} provides a remarkable application of $\Sn$-based harmonic analysis to incomplete rankings, the decomposition into $\Sn$-based translation-invariant components is by essence inadequate to localize the information relative to incomplete rankings on specific subsets of items. Yet incomplete rankings arise in many modern applications (such as recommending systems), where the number of objects to be ranked is very high whereas preferences are generally observed for a small number of objects only. In statistical signal and image processing, novel harmonic analysis tools such as wavelet bases and their extensions have recently revitalized structured data analysis and lead to sparse representations and efficient algorithms for a wide variety of statistical tasks: estimation, prediction, denoising, compression, clustering, etc. Inspired by advances in computational harmonic analysis and its applications to high-dimensional data analysis, our goal is to develop new concepts and algorithms to handle preference data taking the form of incomplete rankings, in order to solve statistical learning problems, motivated by the applications aforementioned, such as efficient/sparse representation of rankings, ranking aggregation, prediction of rankings. More precisely, it is the purpose of this paper to extend the principles of wavelet theory and construct a multiresolution analysis tailored for the description of incomplete rankings.

Let us introduce some preliminary notations to be more specific. For a finite set $E$ of cardinality $\vert E\vert$ and $k\in\{0,\; \ldots,\; \vert E\vert\}$, we denote by $\binom{E}{k}$ the set of all subsets of $E$ with $k$ elements and we set $\mathcal{P}(E) = \bigcup_{j=2}^{\vert E\vert}\binom{E}{j}$. By definition, a ranking over a subset $A\in\mathcal{P}(\n)$ is described by a bijective mapping $\pi : A \rightarrow \{1, \dots, \vert A\vert\}$ that assigns to each item $a\in A$ its rank (with respect to $A$). The ensemble $\Rank{A}$ of such mappings can thus be viewed as the set of the incomplete rankings on $\n$ involving the items of $A$ solely. Notice that unless $A = \{1, \dots, k\}$ with $k\in\{2, \dots, n\}$, this set is different from  -- yet in one-to-one correspondence with -- the set $\Sym{A}$ of permutations on $A$, \textit{i.e.} bijective mappings $\tau : A\rightarrow A$. The variability of incomplete rankings is then represented by a family $(P_{A})_{A\in\mathcal{P}(\n)}$, where $P_{A}$ is a probability distribution on $\Rank{A}$. In order to guarantee that this family describes the distribution of the preferences of a statistical population, it is unavoidable to assume that the following ``projectivity'' property holds: for any $A = \{a_{1}, \dots, a_{k}\}\in\mathcal{P}(\n)$ with $k < n$ and $b\in \n\setminus A$,
\begin{multline}
\label{rationality}
P_{A}(a_{i_{1}}\prec \ldots\prec a_{i_{k}}) = P_{A\cup\{b\}}(a_{i_{1}}\prec \ldots\prec a_{i_{k}}\prec b) + P_{A\cup\{b\}}(a_{i_{1}}\prec \ldots\prec b\prec a_{i_{k}}) +\\ \ldots + P_{A\cup\{b\}}(a_{i_{1}}\prec b\prec \ldots\prec a_{i_{k}}) + P_{A\cup\{b\}}(b\prec a_{i_{1}}\prec \ldots\prec a_{i_{k}}).
\tag{$\ast$}
\end{multline}
It simply means that the probability of a ranking should be conserved when a new item is added. It is straightforward to see that this assumption is equivalent to that stipulating the existence of a probability distribution $p$ on $\Sn$ such that for all $A\in\mathcal{P}(\n)$,
\[
P_{A}(\pi) = \sum_{\sigma\in\Sn(\pi)}p(\sigma),
\]
where $\Sn(\pi)$ is the set of all the permutations $\sigma\in\Sn$ that extend $\pi$, \textit{i.e.} such that for all $(a,b)\in A^2$, $\pi(a) < \pi(b) \Rightarrow \sigma(a) < \sigma(b)$.  For a function $f\in L(\Sn)$, we define its ``marginal'' on the subset $A\in\mathcal{P}(\n)$ by $f_{A}(\pi) = \sum_{\sigma\in\Sn(\pi)}f(\sigma)$. Assumption \eqref{rationality} then states that the $P_{A}$'s are the marginals of a global probability distribution $p$ on $\Sn$. Now, in practical applications, incomplete rankings are not observed on all the subsets of $\mathcal{P}(\n)$ but only on a collection $\mathcal{A}\subset\mathcal{P}(\n)$, called the observation design, and the variability of the observed incomplete rankings is represented by the sub-family $(P_{A})_{A\in\mathcal{A}}$ of $(P_{A})_{A\in\mathcal{P}(\n)}$. Defining the linear operators

\vspace{-0.2cm}
\begin{tabular}{ccc}
\parbox{.3\linewidth}{
\begin{align*}
M_{A} : L(\Sn) &\rightarrow L(\Rank{A})\\
f &\mapsto f_{A}
\end{align*}
} & and &
\parbox{.5\linewidth}{
\begin{align*}
M_{\mathcal{A}} = \bigoplus_{A\in\mathcal{A}}M_{A} : L(\Sn) &\rightarrow \bigoplus_{A\in\mathcal{A}} L(\Rank{A})\\
f &\mapsto (f_{A})_{A\in\mathcal{A}},
\end{align*}
}
\end{tabular}
\vspace{-0.2cm}

\noindent
the analysis of preference data must then be performed in the space
\[
\mathbb{M}_{\mathcal{A}} = M_{\mathcal{A}}(L(\Sn)).
\]
Whereas the space $L(\Sn)$ has been thoroughly studied, $\mathbb{M}_{\mathcal{A}}$ has never been investigated in contrast. Defining an explicit basis for this space or even simply calculating its dimension is indeed far from obvious. Furthermore, unless $\mathcal{A}$ is of the form $\bigcup_{j\in J}\binom{\n}{j}$ with $J\subset\{2, \dots, n\}$, $\Sn$-based translations cannot be defined and $\Sn$-based harmonic analysis cannot previously cannot be applied. Instead, one needs a decomposition that localizes the information related to each subset of items (which is by nature not invariant under $\Sn$-based translations).

\subsection{Main contributions}

In this article, we construct for any $A\in\mathcal{P}(\n)$, the subspace $W_{A}$ of $L(\Sn)$ that localizes the information that is specific to marginals on $A$ and not to marginals on other subsets. Denoting by $\psi_{0}$ the constant function in $L(\Sn)$ equal to $1$ and by $V^{0} = \mathbb{R}\psi_0$ the subspace of constant functions, the major contribution of the present paper is to establish the linear decomposition
\begin{equation}
\label{general-decomposition}
L(\Sn) = V^{0}\oplus\bigoplus_{B\in\mathcal{P}(\n)}W_{B}.
\end{equation}
Notice that this decomposition is an equality and not an isomorphism, because the $W_{B}$'s are subspaces of $L(\Sn)$. Denoting by $\ker M$ the null space of any linear operator $M$, our construction of the spaces $W_{B}$ then allows to localize the information of the marginal on any subset $A\in\mathcal{P}(\n)$ via
\begin{equation*}
L(\Sn) = \ker M_{A} \oplus \left[V^{0}\oplus\bigoplus_{B\in\mathcal{P}(A)}W_{B}\right],
\end{equation*}
and more generally the information of the marginals on the subsets of any collection $\mathcal{A}$ via
\begin{equation*}
L(\Sn) = \ker M_{\mathcal{A}} \oplus \left[V^{0}\oplus\bigoplus_{B\in\bigcup_{A\in\mathcal{A}}\mathcal{P}(A)}W_{B}\right].
\end{equation*}
This last decomposition gives the multiresolution decomposition of the space $\mathbb{M}_{\mathcal{A}}$
\begin{equation}
\label{observation-decomposition}
\mathbb{M}_{\mathcal{A}} = M_{\mathcal{A}}\left(V^{0}\right)\oplus\bigoplus_{B\in\bigcup_{A\in\mathcal{A}}\mathcal{P}(A)}M_{\mathcal{A}}\left(W_{B}\right),
\end{equation}
where $M_{\mathcal{A}}\left(V^{0}\right)$ is the component related to constant functions and for each $B\in\bigcup_{A\in\mathcal{A}}\mathcal{P}(A)$, $M_{\mathcal{A}}\left(W_{B}\right)$ is the component that localizes the information specific to the marginal on $B$. Our result relies on recent advances in algebraic topology about the homological structure of the complex of injective words established in \cite{Reiner13}. We call the decomposition \eqref{general-decomposition} a ``multiresolution decomposition'' because the subspaces localize meaningful parts of the global information of incomplete rankings at different ``scales''. We nonetheless draw attention on the fact that this decomposition is not orthogonal (as we shall see in section \ref{sec:wavelet}) and it is not a ``multiresolution analysis'' in the strict sense. Indeed, the discrete nature of $\Sn$ does not allow to define any dilation operator. However, as shall be seen later in the paper, translation and "dezooming" operators can still be defined to reinforce the analogy between our construction and standard multiresolution analysis, see subsection \ref{subsec:multiresolution-analysis}.

In order to use this decomposition to perform approximation in $\mathbb{M}_{\mathcal{A}}$ in practice, one needs an explicit basis for each space $M_{\mathcal{A}}\left(W_{B}\right)$. The effective construction of such an explicit basis is far from being obvious, because each space $W_{B}$ is defined by many linear constraints based on the complex combinatorial structure of $\Sn$. However, this problem can be related to that of constructing a basis for the homology of certain types of simplicial complexes (namely boolean complexes of Coxeter systems), for which a solution was recently established in \cite{RT11}. Here we adapt the results from \cite{RT11} to exhibit an explicit basis $\Psi_{B}$ for each space $W_{B}$. The concatenated family $\Psi= \{\psi_{0}\}\cup\bigcup_{B\in\mathcal{P}(\n)}\Psi_{B}$ is then a basis of $L(\Sn)$ adapted to the multiresolution decomposition \eqref{general-decomposition}, which shall be referred to as a \textit{wavelet basis} here. From the basis $\Psi$, one obtains, for any collection $\mathcal{A}$ of subsets of $\n$, the wavelet basis
\[
\{ M_{\mathcal{A}}(\psi_{0}) \}\cup \bigcup_{B\in\bigcup_{A\in\mathcal{A}}\mathcal{P}(A)} \{ M_{\mathcal{A}}(\psi) \}_{\psi\in\Psi_{B}},
\]
adapted to the multiresolution decomposition \eqref{observation-decomposition} of the space $\mathbb{M}_{\mathcal{A}}$. Again we draw attention on the fact that $\Psi$ is not a wavelet basis in the strict sense, obtained from the dilations and translations of a ``mother wavelet'', because of the nature of decomposition \eqref{general-decomposition}. It happens however that the choice of the algorithm adapted from \cite{RT11} to generate each $\Psi_{B}$ for $B\in\mathcal{P}(\n)$ leads to a global structure for $\Psi$ encoded in two general relations, strengthening the analogy with classic wavelet bases, see subsection \ref{subsec:structure-wavelet-basis}.

\subsection{Related work}

To the best of our knowledge, only three approaches are documented in the literature to analyze incomplete rankings. The first method is based on the Luce-Plackett model (see \cite{Luce59}, \cite{Plackett75}), the sole parametric statistical model on the group of permutations that can be straightforwardly extended to incomplete rankings. It relies on a strong assumption, referred to as \textit{Luce's choice axiom}, which reduces the complexity of the model, encapsulated by $n$ parameters only (contrasting with the cardinality of $\Sn$). It has been used in a wide variety of applications and several algorithms have been proposed to infer its parameters, see \cite{Hunter04} or \cite{AS13} for instance. Several numerical experiments on real datasets have shown however that its capacity to fit real data is limited, the model being too rigid to handle singularities observed in practice, refer to \cite{Marden96} and \cite{Sun2012}. The two other approaches are non-parametric kernel methods. The one proposed in \cite{Kondor2010} is a diffusion kernel in the Fourier domain, and the one proposed in \cite{Sun2012} is a triangular kernel with respect to the Kendall's tau distance. Though leading to efficient algorithms, both approaches deal with sets $\Sn(\pi)$'s and not directly with incomplete rankings $\pi$'s. This tends to blend the estimated probabilities of the incomplete rankings and thus induces a statistical bias. In contrast, our framework relies on the natural multiresolution structure of incomplete rankings and is the first to allow the definition of approximation procedures directly on this type of ranked data.

We point out that an alternative construction of a multiresolution analysis on $L(\Sn)$ has already been proposed in \cite{Kondor2012}. It is a first breakthrough to deal with singularities of probability distributions on rankings, however it entirely relies on the algebraic structure of $\Sn$. It may be thus viewed as a refinement of harmonic analysis for full or bucket rankings, but does not apply efficiently to the analysis of incomplete rankings. Several approaches have been proposed to generalize the construction of multiresolution analysis and wavelet bases on discrete spaces, mostly on trees and graphs, see for instance \cite{Coifman06}, \cite{Gavish2010}, \cite{Hammond2011}, \cite{REC11} and \cite{Rustamov11}. None of them leads however to the construction for incomplete rankings we promote in this paper, which crucially relies on the topological properties of the complex of injective words.

The use of topological tools to analyze ranked data has been introduced in \cite{JLYY11} and then pursued in several contributions such as in \cite{DSS2012} or \cite{OBO13}. Their approach consists in modeling a collection of pairwise comparisons as an oriented flow on the graph with vertices $\n$ where two items are linked if the pair appears at least once in the comparisons. They show that this flow admits a ``Hodge decomposition'' in the sense that it can be decomposed as the sum of three components, a ``gradient flow'' that corresponds to globally consistent rankings, a ``curl flow'' that corresponds to locally inconsistent rankings, and a ``harmonic flow'', that corresponds to globally inconsistent but locally consistent rankings. Our construction also relies on results from topology but it decomposes the information in a quite different manner, and is tailored to the situation where incomplete rankings can be of any size.

\subsection{Outline of the paper}

The article is structured as follows. In section \ref{sec:localization}, the mathematical formalism that gives a rigorous definition for the concept of information localization is introduced. It is explained how group-based harmonic analysis fits in this framework and why it is not adapted to localize information related to incomplete rankings, and the analysis of the latter is formulated in the setting of injective words. Section \ref{sec:multiresolution-decomposition} contains our major contribution: the spaces $W_{A}$ are constructed and the multiresolution decomposition of $L(\Sn)$ in function of these spaces is exhibited. These results are interpreted in terms of multiresolution analysis and the connection with group-based harmonic analysis is thoroughly discussed. In section \ref{sec:wavelet}, we construct an explicit wavelet basis adapted to the multiresolution decomposition thus built. We establish its main properties and investigate its mathematical structure. Some concluding remarks are collected in section \ref{sec:perspectives}, where several lines of further research are also sketched. Technical proofs are deferred to the Appendix section.

\section{Information localization} \label{sec:localization}

It is the purpose of this section to define concepts which the subsequent analysis fully rests on, while giving insights into the relevance of our construction.

\subsection{Notations}
\label{subsec:notations}

Here an throughout the article, the inclusion between two sets is denoted by $\subset$, the strict inclusion by $\subsetneq$ and the disjoint union by $\sqcup$. Given a finite set $E$, denote by $L(E) = \{ f : E \rightarrow \mathbb{R} \}$ the $\vert E\vert$-dimensional Euclidean space of real valued functions on $E$ equipped with the canonical inner product defined by $\left\langle f, g\right\rangle = \sum_{x\in E}f(x)g(x)$ for any $(f,g)\in L(E)^{2}$. We denote by $\delta_{x}$ the Dirac function at any point $x\in E$ and by $\1{S}$ the indicator function of any $S\subset E$. A partition of $E$ is a collection of nonempty pairwise disjoint subsets $\{ S_{1}, \dots, S_{r}\}$ such that $\bigsqcup_{i=1}^{r}S_{i} = E$.

\subsection{Localizing information through marginals}
\label{subsec:marginal-localization}

Let $\mathcal{X}$ and $\mathcal{Y}$ be two finite sets with $\vert\mathcal{Y}\vert \leq \vert\mathcal{X}\vert$ and $\Pi$ be a mapping $\Pi : \mathcal{X} \rightarrow \mathcal{Y}$. The image of a probability distribution $p$ on $\mathcal{X}$ by $\Pi$ is denoted by $p_{\Pi}$. It is the probability distribution on $\mathcal{Y}$ defined by $p_{\Pi}(y) = \sum_{x\in\Pi^{-1}(\{y\})}p(x)$. If $p$ is the probability distribution of a random variable $X$ on $\mathcal{X}$, then $p_{\Pi}$ is the probability distribution of the random variable $\Pi(X)$ on $\mathcal{Y}$, and for $y\in\mathcal{Y}$, $p_{\Pi}(y)$ is the probability that $\Pi(X) = y$. It is straightforward to see that the two following conditions on $\Pi$ are equivalent:
\begin{enumerate}
	\item $\Pi$ is surjective on $\mathcal{Y}$ and the value of $\vert\Pi^{-1}(\{y\})\vert$ is the same for all $y\in\mathcal{Y}$.
	\item The image by $\Pi$ of the uniform distribution on $\mathcal{X}$ is the uniform distribution on $\mathcal{Y}$, \textit{i.e.} if $p(x) = 1/\vert\mathcal{X}\vert$ for all $x\in\mathcal{X}$ then $p_{\Pi}(y) = 1/\vert\mathcal{Y}\vert$ for all $y\in\mathcal{Y}$.
\end{enumerate}
We assume that these conditions are satisfied in the sequel and call the mapping $\Pi$ a ``marginal transformation''. For $p$ a probability distribution on $\mathcal{X}$, we call $p_{\Pi}$ the ``marginal of p associated to $\Pi$''. More generally, for any function $f\in L(\mathcal{X})$, its marginal associated to $\Pi$, denoted by $f_{\Pi}\in L(\mathcal{Y})$, is defined by
\[
f_{\Pi}(y) = \sum_{x\in\Pi^{-1}(\{y\})}f(x).
\]
If the function $f$ represents a signal over the space $\mathcal{X}$ such as a probability distribution or a discrete image for example, the idea is to interpret $f_{\Pi}$ as the degraded version of $f$ obtained when we observe it through the transformation $\Pi$. It is "degraded" in the sense that $\Pi$ being not injective in general (it can be injective only if $\vert\mathcal{Y}\vert = \vert\mathcal{X}\vert$, and in this case it is isomorphic to the identity transform), $f_{\Pi}$ is an averaged version of $f$ and therefore ``contains less information''. Without being specific about any information measure, the uniform probability distribution can be naturally interpreted as that containing no information, \textit{i.e.} the ``less localized''. The assumption made on $\Pi$ implies that if the original signal $f$ on $\mathcal{X}$ contains no information, then its degraded version on $\mathcal{Y}$ still contains no information.

With a marginal transformation $\Pi$ is naturally associated the marginal operator
\begin{align*}
M_{\Pi} : L(\mathcal{X}) &\rightarrow L(\mathcal{Y})\\
f &\mapsto f_{\Pi}.
\end{align*}
Notice that $M_{\Pi}$ is not a projection because $L(\mathcal{Y})$ is not a subspace of $L(\mathcal{X})$ (unless $\mathcal{Y} = \mathcal{X}$).

If $V$ is a supplementary subspace of $\ker M_{\Pi}$ in $L(\mathcal{X})$, and $f = f_{\ker M_{\Pi}} + f_{V}$ is the corresponding decomposition of a function $f\in L(\mathcal{X})$, one has immediately $f_{\Pi} = M_{\Pi}(f) = M_{\Pi}(f_{V})$. We can thus claim that $f_{V}$ provides the same amount of information as $f_{\Pi}$. This means that data analysis on the space $L(\mathcal{Y})$ can be done equivalently on any supplementary space of $\ker M_{\Pi}$ in $L(\mathcal{X})$. The most natural choice is then surely the orthogonal supplementary $(\ker M_{\Pi})^{\perp}$, because the latter is exactly the space of functions in $L(\mathcal{X})$ that are constant on each $\Pi^{-1}(\{y\})$ for $y\in\mathcal{Y}$ (the proof is straightforward and left to the reader).

In practice however, the signal $f$ is observed through a finite family of marginal transformations $(\Pi_{i})_{i\in I}$, and we would like to ``localize'' as much as possible the information related to a specific transformation $\Pi$ and not to the others. For two marginal transformations $\Pi_{1}$ and $\Pi_{2}$, we say that the subspace $W_{1}$ of $L(\mathcal{X})$ ``fully localizes'' the information related to $\Pi_{1}$ with respect to $\Pi_{2}$ if it satisfies the two conditions listed below:
\begin{itemize}
	\item $W_{1}\cap\ker M_{\Pi_{1}} = \{0\}$ (it localizes information related to $\Pi_{1}$),
	\item $W_{1}\subset\ker M_{\Pi_{2}}$ (it localizes information that is not contained in $\Pi_{2}$).
\end{itemize}
Notice that there is no reason that $(\ker M_{\Pi_{1}})^{\perp}$ satisfies the latter condition for any marginal transformation $\Pi_{2}$.

In the general case, the definition of a space that localizes the information related to a marginal transformation $\Pi$ with respect to the others depends on the relations between all the transformations of the considered family. One particularly important relation is the \textit{refinement}. We say that the transformation $\Pi_{2}$ is a refinement of $\Pi_{1}$ if $\ker M_{\Pi_{2}} \subset \ker M_{\Pi_{1}}$. In that case, there exists a surjective linear mapping from the image of $\Pi_{2}$ to the image of $\Pi_{1}$, and we can say that $\Pi_{1}$ degrades the information more than $\Pi_{2}$ in the sense that the information related to $\Pi_{1}$ can be recovered from the marginal associated to $\Pi_{2}$ (through this surjective linear mapping) whereas the opposite is not true.

\subsection{Group-based harmonic analysis on $\Sn$}
\label{subsec:group-based-harmonic-analysis}

When the original signal space $\mathcal{X}$ is a finite group $G$, we can consider marginal transformations defined through its actions (see the Appendix section for some background in group theory). Let $\mathcal{Y}$ be a finite set on which $G$ acts transitively, by $(g,y)\mapsto g\cdot y$. To each $y_{0}\in\mathcal{Y}$, we associate the marginal transformation
\begin{align*}
\Pi_{y_{0}} : G &\rightarrow \mathcal{Y}\\
g &\mapsto g\cdot y_{0}.
\end{align*}
It satisfies the two conditions of a marginal transformation. First, $\Pi_{y_{0}}$ is surjective on $\mathcal{Y}$ because the action of $G$ on $\mathcal{Y}$ is transitive. Second, for $y\in\mathcal{Y}$, the set $\Pi_{y_{0}}^{-1}(\{y\}) = \{g\in G \;\vert\; g\cdot y_{0} = y \}$ is a left coset of the stabilizer of $y_{0}$, $\{g\in G \;\vert\; g\cdot y_{0} = y_{0} \}$, and thus have same cardinality. The interpretation behind this marginal transformation is as follows. If $p$ is a probability distribution on $G$, then $p(g)$ is the probability of drawing the element $g$ in $G$ and $p_{\Pi_{y_{0}}}(y)$ is the probability of drawing an element $g\in G$ such that $g\cdot y_{0} = y$. For a function $f\in L(G)$, the collection of all its marginals associated to the transformations $\Pi_{y_{0}}$ for $y_{0}\in\mathcal{Y}$ can be gathered in the $\vert\mathcal{Y}\vert$-squared matrix $T_{\mathcal{Y}}(f)$ defined by
\[
\left[T_{\mathcal{Y}}(f)\right]_{y,y_{0}} = f_{\Pi_{y_{0}}}(y),
\]
each column representing a marginal. Now, by linearity, $T_{\mathcal{Y}}(f) = \sum_{g\in G}f(g)T_{\mathcal{Y}}(\delta_{g})$ for all $f\in L(G)$, and it is easy to see that for $g\in G$, $T_{\mathcal{Y}}(\delta_{g})$ is actually the translation operator on $L(\mathcal{Y})$, \textit{i.e.} for all $F\in L(\mathcal{Y})$ and $y\in\mathcal{Y}$,
\[
\left(T_{\mathcal{Y}}(\delta_{g})F\right)(y) = F(g^{-1}\cdot y).
\]
In other words, $g \mapsto T_{\mathcal{Y}}(\delta_{g})$ is the permutation representation of $G$ on $L(\mathcal{Y})$ associated to the action $(g,y)\mapsto g\cdot y$. This means that the information contained in the collection of marginals $(\Pi_{y_{0}})_{y_{0}\in\mathcal{Y}}$ can be decomposed using group representation theory, which is the general principle of harmonic analysis. Harmonic analysis on a finite group $G$ is defined as the decomposition of $L(G)$ into irreducible representations of $G$, see \cite{Diaconis1988}. These components are invariant under translation and each localize the information related to a specific ``frequency''. The symmetric group has an additional particularity: each of its irreducible representations is ``associated to'' one specific meaningful permutation representation and thus allows to localize the information of the associated marginals. Let us develop this interpretation.

The symmetric group $\Sn$ is the set of all the bijective mappings $\sigma : \n \rightarrow \n$ equipped with the composition law $(\sigma,\tau)\mapsto \sigma\tau$ defined by $\sigma\tau (i) = \sigma(\tau(i))$ for $i\in \n$. Irreducible representations of $\Sn$, are indexed by partitions of $n$, \textit{i.e.} tuples $\lambda = (\lambda_{1}, \dots, \lambda_{r})\in\n^{r}$ such that $\lambda_{1} \geq \dots \geq \lambda_{r}$ and $\lambda_{1} + \dots + \lambda_{r} = n$, with $r\in\{1, \dots n\}$. The irreducible representation indexed by $\lambda$ is denoted by $S^{\lambda}$, and its dimension by $d_{\lambda}$. The harmonic decomposition of $L(\Sn)$ thus writes
\[
L(\Sn) \cong \bigoplus_{\lambda\,\vdash\, n} d_{\lambda}S^{\lambda},
\]
where $\lambda\vdash n$ means that the sum is taken on all the partitions $\lambda$ of $n$ (see the Appendix section). The spaces $S^{\lambda}$ are called the Specht modules. For $\mathcal{B} = (B_{1}, \dots, B_{r})$ an ordered partition of $\n$ and $\sigma\in \Sn$, we set
\[
\sigma\cdot\mathcal{B} = (\sigma(B_{1}), \dots, \sigma(B_{r})),
\]
where $\sigma(B) = \{\sigma(b) \;\vert\; b\in B\}$, for $B\subset\n$. This defines an action of $\Sn$ on the set of all ordered partitions of $\n$. The shape of an ordered partition of $\n$ $\mathcal{B} = (B_{1}, \dots, B_{r})$ is the tuple $(\vert B_{1}\vert, \dots, \vert B_{r}\vert)$. It is easy to see that the orbits of $\Sn$ are the sets of ordered partitions of $n$ of a given shape. For $\lambda \vdash n$, we denote by $\Part_{\lambda}(\n)$ be the set of ordered partitions of $\n$ of shape $\lambda$, and define $M^{\lambda} = L(\Part_{\lambda}(\n))$, called a Young module. In this context, the marginal transformation associated to a $\mathcal{B}_{0}\in\Part_{\lambda}(\n)$ is defined by
\begin{align*}
\Pi_{\mathcal{B}_{0}} : \Sn &\rightarrow \Part_{\lambda}(\n)\\
\sigma &\mapsto \sigma\cdot\mathcal{B}_{0}.
\end{align*}
The marginal of a function $f\in L(\Sn)$ associated to this transformation is denoted by $f_{\mathcal{B}_{0}}$ and is called a $\lambda$-marginal. The $\vert\Part_{\lambda}(\n)\vert$-square matrix $T_{\lambda}(f)$ that gathers all the $\lambda$-marginals of $f$ is equal to the sum $\sum_{\sigma\in\Sn}f(\sigma)T_{\lambda}(\delta_{\sigma})$, where $T_{\lambda}(\delta_{\sigma})$ is the matrix of the permutation representation of $\Sn$ on $M^{\lambda}$ taken in $\sigma$. Now it happens that the Specht module $S^{\lambda}$ can be defined as a subspace of $M^{\lambda}$, see \cite{Diaconis1988}. This means that the information localized by $S^{\lambda}$ in the harmonic decomposition of $L(\Sn)$ is contained in the $\lambda$-marginals. This information is also specific to $\lambda$-marginals in a certain way.

The dominance order on partitions of $n$ is the partial order defined, for $\lambda = (\lambda_{1}, \dots, \lambda_{r})$ and $\mu = (\mu_{1}, \dots, \mu_{s})$ by $\lambda \unrhd \mu$ if for all $j\in \{1, \dots, r\}$, $\sum_{i=1}^{j}\lambda_{i} \geq \sum_{i=1}^{j}\mu_{i}$. When $\lambda \unrhd \mu$ and $\lambda\neq\mu$ we write $\lambda\rhd\mu$. The decomposition of the Young module $M^{\lambda}$ for $\lambda\vdash n$ is given by Young's rule (see \cite{Diaconis1988}):
\[
M^{\lambda} \cong \bigoplus_{\mu\,\vdash\, n}K_{\mu,\lambda}S^{\mu}, \qquad\text{where}\quad  \left\{
\begin{aligned}
K_{\mu,\lambda} &= 0\quad \text{if } \mu\lhd\lambda\\
K_{\lambda,\lambda} &= 1\\
K_{\mu,\lambda} &\geq 1\quad \text{if } \mu \rhd \lambda\\
\end{aligned}\right. ,
\]
\textit{i.e.}
\[
M^{\lambda} \cong S^{\lambda}\oplus\bigoplus_{\mu \rhd \lambda}K_{\mu,\lambda}S^{\mu}.
\]
This means that for a given $\lambda \vdash n$, the Specht module $S^{\lambda}$ localizes the information of $M^{\lambda}$ that is not contained in the $M^{\mu}$'s for $\mu \rhd \lambda$. In this sense, $S^{\lambda}$ contains the information that is specific to $\lambda$-marginals and not the others.

\subsection{``Absolute'' and ``relative'' rank information}
\label{subsec:absolute-relative-rank-information}

The precedent subsection shows that harmonic analysis on $\Sn$ consists in localizing information specific to collections of marginal transformations $(\Pi_{\mathcal{B}})_{\mathcal{B}\in\Part_{\lambda}(\n)}$, for $\lambda \vdash n$. Let $p$ be a probability distribution on $\Sn$ and $\Sigma$ a random permutation of law $p$. If $S\subset\n$ represents an event $\mathcal{E}$ on the random variable $X$, we denote by $\mathbb{P}[\mathcal{E}]$ the probability of this event, \textit{i.e.} $\mathbb{P}[\mathcal{E}] = \sum_{\sigma\in S}p(\sigma)$. For $\lambda = (\lambda_{1}, \dots, \lambda_{r})\vdash n$ and $\mathcal{B} = (B_{1}, \dots, B_{r})\in\Part_{\lambda}(\n)$, we have for any $\mathcal{B}^{\prime} = (B^{\prime}_{1}, \dots, B^{\prime}_{r})\in\Part_{\lambda}(\n)$,
\[
p_{\mathcal{B}}(\mathcal{B}^{\prime}) = \mathbb{P}\left[\Sigma(B_{1}) = B^{\prime}_{1},\dots, \Sigma(B_{r}) = B_{r}^{\prime}\right].
\]
To gain more insight into the interpretation of these marginals, let us consider first the simple case $\lambda = (n-1, 1)$. Ordered partitions of $\n$ of shape $(n-1,1)$ are necessarily of the form $(\n\setminus\{i\}, \{i\})$, with $i\in\n$. Then for $(i,j)\in\n^{2}$, we have the simplification
\[
\mathbb{P}\left[\Sigma(\n\setminus\{i\}) = \n\setminus\{j\},\ \Sigma(\{i\}) = \{j\}\right] = \mathbb{P}\left[\Sigma(i) = j\right].
\]
The marginal of $p$ associated to $(\n\setminus\{i\}, \{i\})$ is thus the probability distribution $(\mathbb{P}[\Sigma(i) = j])_{j\in\n}$ on $\n$. From a ranking point of view, this is the law of the rank of item $i$. The matrix $T_{(n-1,1)}(p)$ that gathers all the $(n-1,1)$-marginals of $p$ is given by
\[
T_{(n-1,1)}(p) = \left(
\begin{array}{ccc}
\mathbb{P}[\Sigma(1) = 1] & \cdots & \mathbb{P}[\Sigma(n) = 1]\\
\vdots & \ddots & \vdots \\
\mathbb{P}[\Sigma(1) = n] & \cdots & \mathbb{P}[\Sigma(n) = n]\\
\end{array}\right),
\]
where the marginal of $p$ associated to $(\n\setminus\{i\}, \{i\})$ is represented by column $i$. It is easy to see that this matrix is bistochastic, and that the row $i$ represents the probability distribution $(\mathbb{P}[\Sigma^{-1}(i) = j])_{j\in\n}$ on $\n$. This is the probability distribution of the index of the element ranked at the $i^{th}$ position. In both cases, the distribution captures the information about an ``absolute rank'', in the sense that it is the rank of an item inside a ranking implying all the $n$ items. Either we consider the distribution of the absolute ranks of a fixed item $i$, or else we consider the distribution of the item having a fixed absolute rank $i$.

More generally for $k\in\{1, \dots, n-1\}$, elements of $\Part_{(n-k,k)}(\n)$ are of the form $(\n\setminus A, A)$ with $A\in\binom{\n}{k}$, and the marginal law of $p$ associated to $(\n\setminus A, A)$ is the probability distribution $(\mathbb{P}[\Sigma(A) = B])_{\vert B\vert = k}$ on $\binom{\n}{k}$. From a ranking point of view, $\mathbb{P}[\Sigma(A) = B]$ is the probability that the items of $A$ are ranked at the absolute positions of $B$, regardless of their order inside these positions. In the general case, for $\lambda = (\lambda_{1}, \dots, \lambda_{r})\vdash n$, $\mathcal{B} = (B_{1}, \dots, B_{r})\in\Part_{\lambda}(\n)$ and $\mathcal{B}^{\prime} = (B^{\prime}_{1}, \dots, B^{\prime}_{r})\in\Part_{\lambda}(\n)$, $\mathbb{P}[\Sigma\cdot\mathcal{B} = \mathcal{B}^{\prime}]$ is the probability that the items of $B_{i}$ are ranked at the absolute positions of $B^{\prime}_{i}$, for $i\in\{1, \dots, r\}$.

\begin{example}
We give an illustration of this type of marginals on a real dataset with $n = 4$, studied in \cite{Diaconis1998}. It is composed of $2262$ answers of German citizens who were asked to rank the desirability of four political goals, that we consider as items $1$, $2$, $3$ and $4$. Each ranking of these four items received a certain number of votes. Normalizing by $2262$, the total number of votes,we obtain a probability distribution $p$ on $\Sym{4}$. It is represented in figure \ref{fig:German-full}, where the $x$-axis represents the $24$ elements of $\Sym{4}$, denoted by $a_{1}a_{2}a_{3}a_{4}$ instead of $a_{1}\prec a_{2}\prec a_{3}\prec a_{4}$ and ordered by the lexicographic order, \textit{i.e.} $1234, 1243, ..., 4321$.

\begin{figure}[h!]
\centering
\includegraphics[scale=0.5]{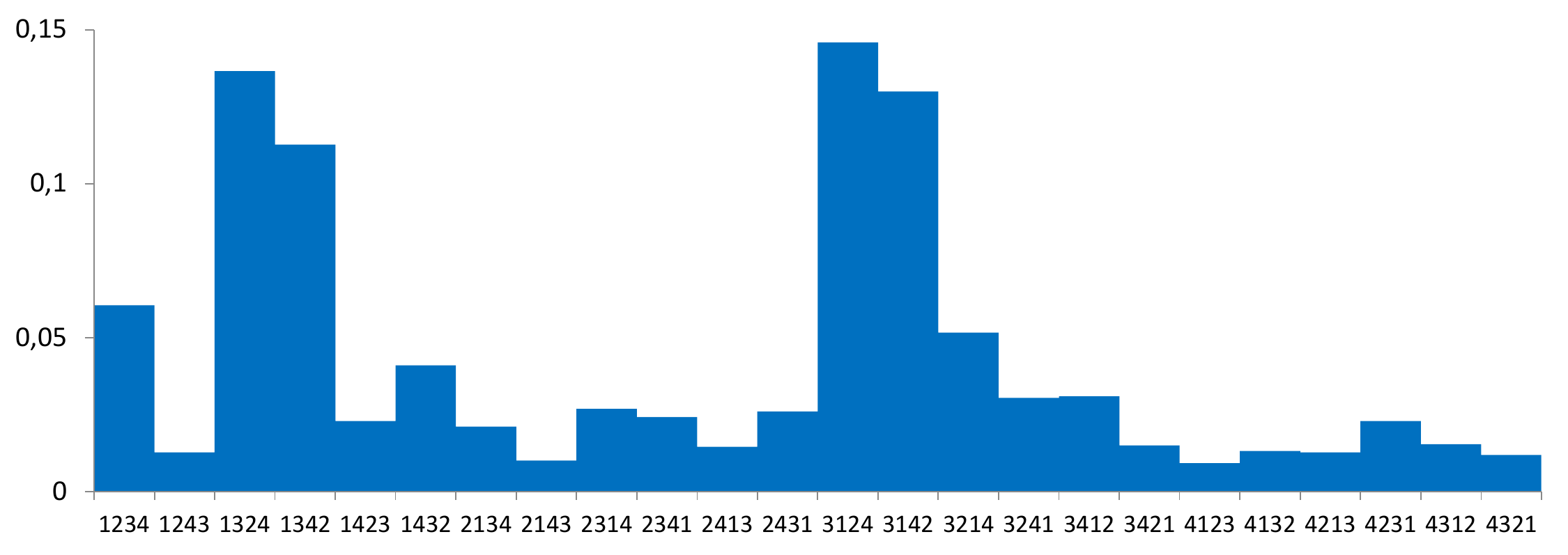}
\caption{Probability distribution $p$ of the rankings of the four political goals}
\label{fig:German-full}
\end{figure}

There are $5$ partitions of $4$: $(4), (3,1), (2,2), (2,1,1), (1,1,1,1)$. There is only one $(4)$-marginal, it is the constant $\sum_{\sigma\in\Sym{4}}p(\sigma) = 1$, and there are $24$ $(1,1,1,1)$-marginals, the translations of $p$. We consider the marginals of the three other types. Let $\Sigma$ be a random permutation of law $p$. The $(3,1)$ marginals are the laws of the random variables $\Sigma(i)$, representing the rank of item $i$, for $i\in\llbracket 4\rrbracket$. The $(2,2)$-marginals are the laws of the random variables $\{\Sigma(i), \Sigma(j)\}$ and the $(2,1,1)$ are the laws of the random variables $(\Sigma(i), \Sigma(j))$, for $(i,j)\in\n^{2}$ with $i\neq j$. All these marginals are represented in figure \ref{fig:German-absolute-marginals} (with different scales).

\begin{figure}[h!]
\centering
\includegraphics[scale=0.45]{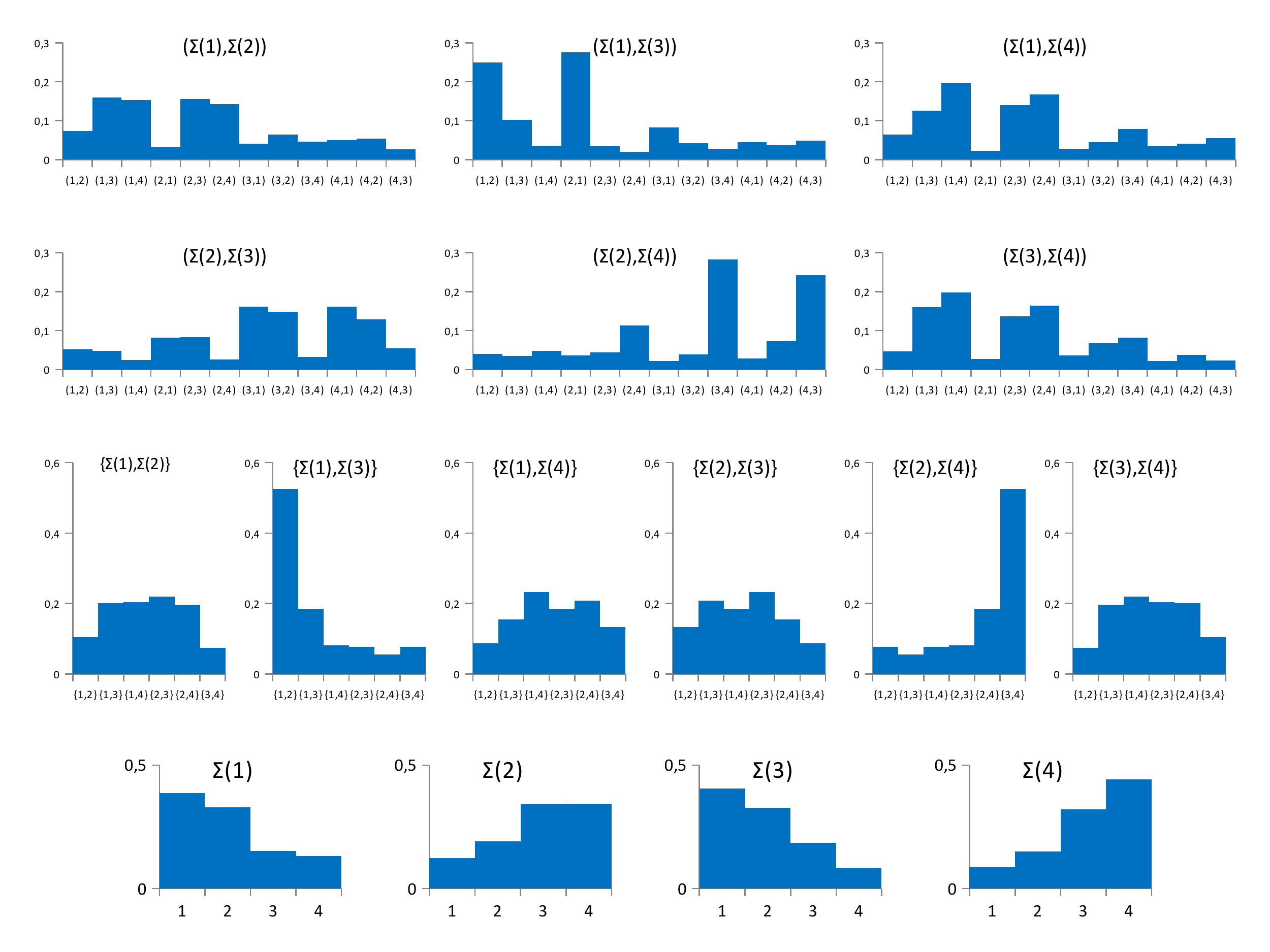}
\caption{Absolute $\lambda$-marginals of $p$, for $\lambda = (3,1), (2,2), (2,1,1)$}
\label{fig:German-absolute-marginals}
\end{figure}

\end{example}

In the analysis of incomplete rankings, we are not interested in absolute rank information, but in ``relative'' rank information. When incomplete rankings are observed on a subset of items $A\in\mathcal{P}(\n)$, the information we have access to is about the ranks of the items of $A$ relatively to $A$. In the same way, the prediction of a ranking on $A$ only involves the information about the ranks of the items of $A$ relatively to $A$. This is the fundamental difference between the analysis of full rankings or bucket orders (also called partial rankings) and the analysis of incomplete rankings. This implies that the marginal transformations and the information localization involved are completely different. Let $\Sigma$ be a random permutation of law $p$ on $\Sn$. In the analysis of incomplete rankings, we are interested in probabilities of the form
\[
\mathbb{P}[\Sigma(a_{1}) < \dots < \Sigma(a_{k})],
\]
for $k\in\{2, \dots, n\}$ and $\{a_{1}, \dots, a_{k}\}\subset\n$. So we are not interested in the values $\Sigma(a_{1}), \dots, \Sigma(a_{k})$, but only in their ordering, which induce a ranking of the items $a_{1}, \dots, a_{k}$. We are thus interested in the marginals $p_{A}$ of $p$ defined in the Introduction section, for $A\in\mathcal{P}(\n)$.

\begin{example}
Considering the same example as before, we represent all the marginals of $p$ involved in the analysis of incomplete rankings. For each $A\in\mathcal{P}(\llbracket 4\rrbracket)$, the marginal $p_{A}$ is represented in figure \ref{fig:German-relative-marginals} by a graph with the $x$-axis constituted of the elements of $\Rank{A}$ ordered by the lexicographic order.

\begin{figure}[h!]
\centering
\includegraphics[scale=0.38]{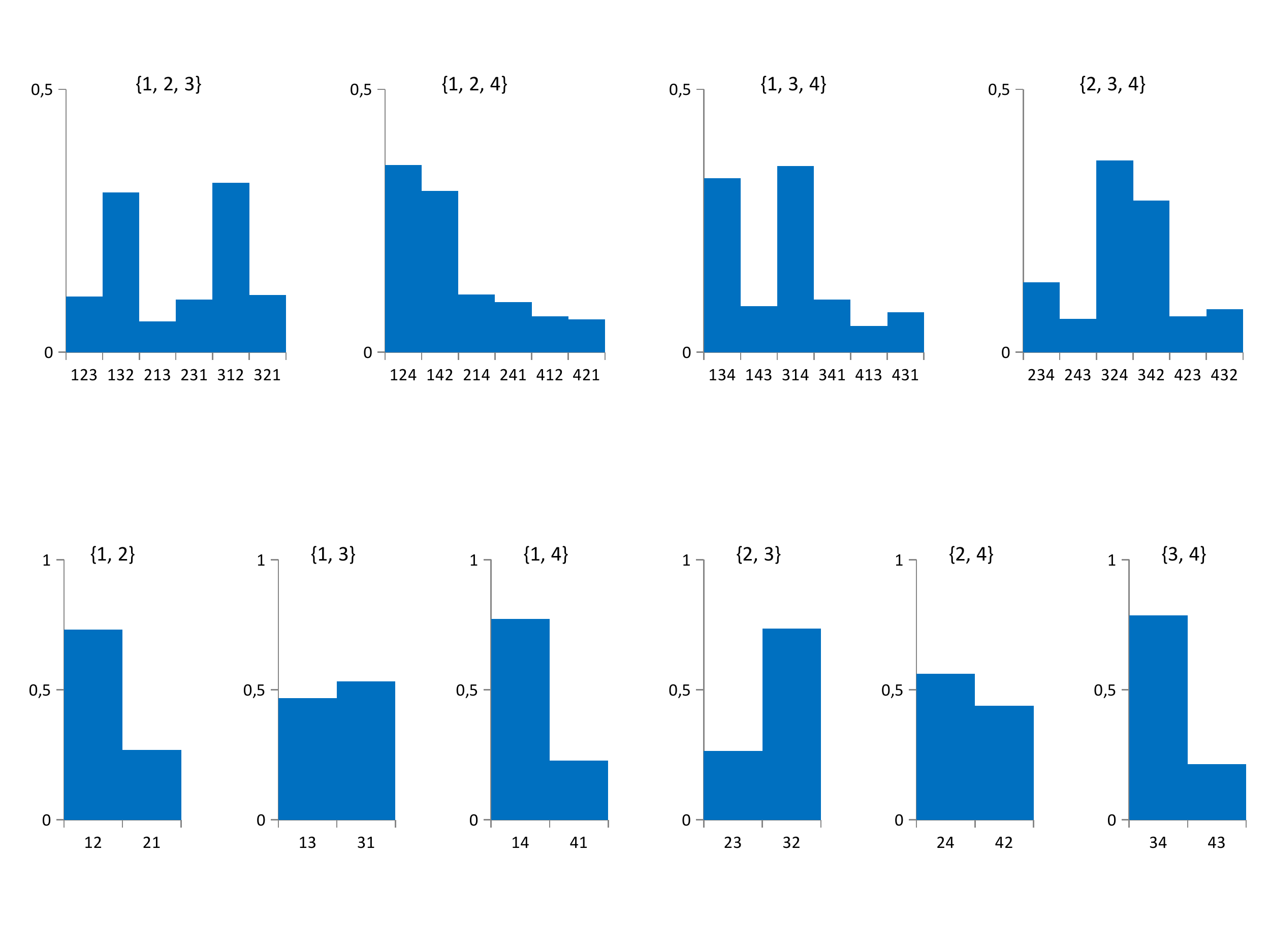}
\caption{Relative marginals of $p$ on subsets $A\subset\llbracket 4\rrbracket$ with $\vert A\vert = 2$ or $3$}
\label{fig:German-relative-marginals}
\end{figure}

\end{example}

It is obvious that each of the two families of marginal transformations leads to the analysis of completely different objects: full rankings or bucket orders involving absolute rank information, or incomplete rankings involving relative rank information. There is a way to handle incomplete rankings with $\Sn$-based harmonic analysis, as it is done in \cite{Kondor2010}, but it is not really adapted and it does not provide a powerful general framework. This can be achieved only by fully exploiting the structure of incomplete rankings and considering the right marginal transformations.

In this case, the marginal transformations suited for the analysis of incomplete rankings map a permutation $\sigma$ to the ranking it induces on a subset of items $A = \{a_{1}, \dots, a_{k}\}\in\mathcal{P}(\n)$, through the order of the values $\sigma(a_{1}), \dots, \sigma(a_{k})$. This definition is however not easy to use and thus not  convenient to characterize the structure of incomplete rankings. It happens that they can be defined from another point of view that fits with the mathematical structure of incomplete rankings. It comes from the observation that the ranking induced by a full ranking on a subset of items $A$ is obtained by keeping only the items of $A$, in the same order. More specifically, if $\sigma$ corresponds to the full ranking $a_{1} \prec \dots \prec a_{n}$ on $\n$, the ranking it induces on $A$ is given by $a_{i_{1}} \prec \dots \prec a_{i_{\vert A\vert}}$ where $i_{1} < \dots < i_{\vert A\vert}$ and $A = \{a_{i_{1}}, \dots, a_{i_{\vert A\vert}}\}$. This perspective is best expressed in the language of injective words.

\subsection{Analysis of incomplete rankings through injective words}
\label{subsec:injective-words}

An injective word over $\n$ is an expression $\omega = \omega_{1}\dots\omega_{k}$ where $1\leq k\leq n$ and $\omega_{1}, \dots, \omega_{k}$ are distinct elements of $\n$. The content of the word $\omega = \omega_{1}\dots\omega_{k}$ is $c(\omega) = \lbrace \omega_{1}, \dots, \omega_{k}\rbrace$, and its size is $\vert\omega\vert := \vert c(\omega)\vert$. The empty word $\overline{0}$ is by convention the unique word of size $0$ and content $\emptyset$. A subword of a word $\omega = \omega_{1}\dots\omega_{k}\in\Gamma_{n}$ is an expression $\omega_{i_{1}}\dots\omega_{i_{r}}$ with $1 \leq i_{1} < \dots < i_{r} \leq k$. We denote by $\Gamma_{n}$ the set of injective words over $\n$ and for $A\subset\n$ and $k\in\{0, \dots, n\}$, we set $\Gamma(A) = \{\omega\in\Gamma_{n} \;\vert\; c(\omega) = A\}$ and $\Gamma^{k} = \{ \omega\in\Gamma_{n} \;\vert\; \vert\omega\vert = k \}$. We thus have
\begin{equation}
\label{decomposition Gamma_n}
\Gamma_{n} = \bigsqcup_{k=0}^{n}\Gamma^{k} = \bigsqcup_{k=0}^{n}\bigsqcup_{\vert A\vert = k}\Gamma(A).
\end{equation}

To each incomplete ranking $\pi = a_{1} \prec\dots\prec a_{k}$, we associate the corresponding injective word $a_{1}\dots a_{k}$, and we still denote it by $\pi$. The sets $\Rank{A}$ and $\Gamma(A)$ are thus identified for $A\in\mathcal{P}(\n)$, in particular $\Sn$ is identified to $\Gamma(\n)$, and both interpretations will be used indifferently in the sequel.

The language of injective words has two major advantages for the analysis of incomplete rankings. The first is that it is well suited to express the marginal transformations that we want to consider and their properties. Let $(A,B)\in\mathcal{P}(\n)^{2}$ with $A\subset B$ and $\sigma = a_{1}\dots a_{\vert B\vert}\in\Rank{B}$ representing the ranking $a_{1} \prec\dots\prec a_{\vert B\vert}$. Then the ranking induced by $\sigma$ on $A$ is represented by the unique subword of $\sigma$ with content equal to $A$. The latter is obtained by deleting from $a_{1}\dots a_{\vert B\vert}$ all the $a_{i}$'s that do not belong to $A$. We denote by $\sigma_{\vert A}$ the ranking induced by $\sigma$ on $A$ as well as the injective word representing it. The marginal transformations of interest in the analysis of incomplete rankings are thus defined by
\begin{align*}
\Pi_{A} : \Sn &\rightarrow \Rank{A}\\
	   	\sigma &\mapsto \sigma_{\vert A},
\end{align*}
for $A\in\mathcal{P}(\n)$. We denote respectively by $f_{A}$ and $M_{A}$ the marginal of a function $f\in (\Sn)$ and the marginal operator associated to $\Pi_{A}$ (these notations are the same as in the introduction). Recall that for $\pi\in\Rank{A}$ viewed as a mapping $A\rightarrow\{1,\dots,\vert A\vert\}$, the set $\Sn(\pi)$ is defined as $\Sn(\pi) = \{\sigma\in\Sn \;\vert\; \text{for all }(a,b)\in A^2 \text{ such that }\pi(a) < \pi(b),\ \sigma(a) < \sigma(b)\}$. Viewing now $\pi\in\Rank{A}$ as an injective word, it is clear that $\Sn(\pi) = \{\sigma\in\Sn \;\vert\; \sigma_{\vert A} = \pi\} = \Pi^{-1}(\{\pi\})$. More generally we define, for $(A,B)\in\mathcal{P}(\n)^{2}$ with $A\subset B$ and $\pi\in\Rank{A}$, $\Rank{B}(\pi) = \{\sigma\in\Rank{B} \;\vert\; \sigma_{\vert A} = \pi\}$, with $\Sn(\pi) := \Rank{\n}(\pi)$. The fact that $\Pi_{A}$ is a marginal transformation is thus a direct consequence of the following lemma (for $B = \n$), its technical proof is postponed to the Appendix section.

\begin{lemma}
\label{combinatorics}
Let $(A,B)\in\mathcal{P}(\n)^{2}$ with $A\subset B$. Then $\{\Rank{B}(\pi)\}_{\pi\in\Rank{A}}$ is a partition of $\Rank{B}$ and for all $\pi\in\Rank{A}$, $\vert \Rank{B}(\pi)\vert = \vert B\vert !/\vert A\vert !$.
\end{lemma}

The refinement relations inside the family of marginal transformations $(\Pi_{A})_{A\in\mathcal{P}(\n)}$ rely on the structure of injective words. For $\pi\in\Gamma_{n}$ with $\vert\pi\vert < n$, $b\in\n\setminus c(\pi)$ and $i\in\{ 1, \dots, \vert\pi\vert + 1\}$, we denote by $\pi\triangleleft_{i} b$ the word obtained by inserting $b$ in $i^{\text{th}}$ position in $\pi$. The following lemma, the proof of which is straightforward, is the base of the refinement relations.

\begin{lemma}
\label{ranking-extension}
Let $A\subsetneq B\subset\n$ and $\pi\in\Rank{A}$. For all $b\in B\setminus A$, $\Rank{B}(\pi) = \bigsqcup_{i = 1}^{\vert A\vert +1}\Rank{B}(\pi\triangleleft_{i}b)$. In particular, $\Rank{A\cup\{b\}}(\pi) = \{\pi\triangleleft_{1}b, \dots, \pi\triangleleft_{\vert A\vert + 1}b \}$.
\end{lemma}

\begin{proposition}
\label{refinement}
For $A, B\in\mathcal{P}(\n)$ with $A\subset B$, $\Pi_{B}$ is a refinement of $\Pi_{A}$, \textit{i.e.}
\[
\ker M_{B} \subset \ker M_{A}.
\]
\end{proposition}

\begin{proof}
Let $A\in\mathcal{P}(\n)$ and $b\in\n\setminus A$. For $f\in L(\Sn)$, lemma \ref{ranking-extension} gives, for all $\pi\in\Rank{A}$,
\[
M_{A}f(\pi) = \sum_{\sigma\in\Sn(\pi)}f(\sigma) = \sum_{i=1}^{\vert A\vert + 1}\sum_{\sigma\in\Sn(\pi\triangleleft_{i}b)}f(\sigma) = \sum_{i=1}^{\vert A\vert + 1}M_{A\cup\{b\}}f(\pi\triangleleft_{i}b).
\]
This implies that $\ker M_{A\cup\{b\}} \subset \ker M_{A}$ and the proof is concluded by induction.
\end{proof}

The second major advantage of the language of injective words is that it allows to define a global framework for all incomplete rankings. To this purpose, we see the elements of $L(\Gamma_{n})$ as free linear combinations of injective words, also called chains, \textit{i.e.} expressions of the form $x = \sum_{\omega\in\Gamma_{n}}x(\omega)\omega$, where $\omega$ refers at the same time to a word in $\Gamma_{n}$ and to the Dirac function of this word in $L(\Gamma_{n})$. Notice then that $\overline{0}$ denotes the Dirac function in the empty word, whereas $0$ denotes the function equal to $0$ for all $\omega\in\Gamma_{n}$, and that the indicator function of a set $S\subset\Gamma_{n}$ is equal to the sum of the Dirac functions in its elements
\[
\1{S} = \sum_{\sigma\in S}\sigma.
\]
By definition, the marginal operator $M_{A}$ applied to the Dirac function of $\sigma\in\Sn$ in $L(\Sn)$ is equal to the Dirac function of $\sigma_{\vert A}$ in $L(\Rank{A})$. Using the chain notation, this gives:
\begin{equation}
\label{marginal-injective-word}
M_{A}\sigma = \sigma_{\vert A}.
\end{equation}
A function in $L(\Gamma(A))$ for $A\subset\n$ is thus directly seen as a chain in $L(\Gamma_{n})$, and by equation \eqref{decomposition Gamma_n}, we have $L(\Gamma_{n}) = \bigoplus_{k=0}^{n}\bigoplus_{\vert A\vert = k}L(\Gamma(A))$. This decomposition allows to embed $L(\Sn)$, all the spaces of marginals $L(\Rank{A})$ for $A\in\mathcal{P}(\n)$ and the spaces $L(\Gamma(A))$ for $\vert A\vert \leq 1$ into one general space, that is $L(\Gamma_{n})$. For $n = 4$, $L(\Gamma_{4})$ decomposes as follows.
\begin{center}
$L\left(\Sym{4}\right)$\\
\vspace{0.25cm}
$L\left(\Rank{\{1,2,3\}}\right) \ \oplus\ L\left(\Rank{\{1,2,4\}}\right) \ \oplus\ L\left(\Rank{\{1,3,4\}}\right) \ \oplus\ L\left(\Rank{\{2,3,4\}}\right)$\\
\vspace{0.25cm}
$L\left(\Rank{\{1,2\}}\right) \ \oplus\ L\left(\Rank{\{1,3\}}\right) \ \oplus\ L\left(\Rank{\{1,4\}}\right) \ \oplus\ L\left(\Rank{\{2,3\}}\right) \ \oplus\ L\left(\Rank{\{2,4\}}\right) \ \oplus\ L\left(\Rank{\{3,4\}}\right)$\\
\vspace{0.25cm}
$L(\Gamma(\{1\})) \ \oplus\ L(\Gamma(\{2\})) \ \oplus\ L(\Gamma(\{3\})) \ \oplus\ L(\Gamma(\{4\}))$\\
\vspace{0.25cm}
$L(\Gamma(\overline{0}))$
\end{center}
This embedding allows to model all possible observations of incomplete rankings. Indeed, let $\mathcal{A}\subset\mathcal{P}(\n)$ be an observation design. Then for each $A\in\mathcal{A}$, the variability of the observed rankings on $A$ is represented by a probability distribution $P_{A}\in L(\Rank{A})$. The total variability of the observed rankings is thus represented by the collection $(P_{A})_{A\in\mathcal{A}}\in\bigoplus_{A\in\mathcal{A}}L(\Rank{A})\subset L(\Sn)$.

\begin{example}
Let us assume that we observe incomplete rankings on $\llbracket 4\rrbracket$ through the observation design $\mathcal{A} = \{ \{1,3\}, \{2,4\}, \{3,4\}, \{1,2,3\}, \{1,3,4\} \}$. Then the collection of probability distributions is an element of the sum of the spaces in bold, in the following representation.
\begin{center}
$L\left(\Sym{4}\right)$\\
\vspace{0.25cm}
$\mathbf{L\left(\Rank{\{1,2,3\}}\right)} \ \oplus\ L\left(\Rank{\{1,2,4\}}\right) \ \oplus\ \mathbf{L\left(\Rank{\{1,3,4\}}\right)} \ \oplus\ L\left(\Rank{\{2,3,4\}}\right)$\\
\vspace{0.25cm}
$L\left(\Rank{\{1,2\}}\right) \ \oplus\ \mathbf{L\left(\Rank{\{1,3\}}\right)} \ \oplus\ L\left(\Rank{\{1,4\}}\right) \ \oplus\ L\left(\Rank{\{2,3\}}\right) \ \oplus\ \mathbf{L\left(\Rank{\{2,4\}}\right)} \ \oplus\ \mathbf{L\left(\Rank{\{3,4\}}\right)}$\\
\vspace{0.25cm}
$L(\Gamma(\{1\})) \ \oplus\ L(\Gamma(\{2\})) \ \oplus\ L(\Gamma(\{3\})) \ \oplus\ L(\Gamma(\{4\}))$\\
\vspace{0.25cm}
$L(\Gamma(\overline{0}))$
\end{center}
\end{example}

Notice however that we are not interested in performing data analysis in the space $\bigoplus_{A\in\mathcal{A}}L(\Rank{A})$ but in its subspace $\mathbb{M}_{\mathcal{A}}$ of the collections $(f_{A})_{A\in\mathcal{A}}\in \bigoplus_{A\in\mathcal{A}}L(\Rank{A})$ that satisfy condition \eqref{rationality}, $\mathbb{M}_{\mathcal{A}} = M_{\mathcal{A}}(L(\Sn))$. This embedding remains nonetheless very convenient to define global operators that exploit the structure of injective words.

\begin{definition}[Deletion operator]
\label{def-deletion}
Let $a\in\n$. For $\pi\in\Gamma_{n}$ such that $a\in c(\pi)$, we denote by $\pi\setminus\{a\}$ the word obtained by deleting the letter $a$ in the word $\pi$. We extend this operation into the operator $\varrho_{a} : L(\Gamma_{n}) \rightarrow L(\Gamma_{n})$, defined on a Dirac function $\pi$ by
\begin{equation*}
\varrho_{a}\pi = \left\{
\begin{aligned}
\pi\setminus\{a\} \qquad &\text{if } a\in c(\pi)\\
\pi \qquad &\text{otherwise.}
\end{aligned}\right.
\end{equation*}
For $a_{1}, a_{2}\in\n$, it is obvious that $\varrho_{a_{1}}\varrho_{a_{2}} = \varrho_{a_{2}}\varrho_{a_{1}}$. This allows to define, for $A = \{a_{1}, \dots, a_{k}\}\subset\n$, $\varrho_{A} = \varrho_{a_{1}}\dots\varrho_{a_{k}}$. We set by convention $\varrho_{\emptyset}x = x$ for all $x\in L(\Gamma_{n})$.
\end{definition}

\begin{remark}
Notice that for any $\pi\in\Gamma_{n}$, $\varrho_{c(\pi)}\pi = \overline{0}$. This implies that for $A\subset\n$ and $x\in L(\Gamma(A))$, $\varrho_{A}x = \left[\sum_{\pi\in\Gamma(A)}x(\pi)\right]\overline{0}$.
\end{remark}

The family of spaces $(L(\Gamma(A)))_{A\subset\n}$ equipped with the family of operators $(\varrho_{B\setminus A})_{A\subset B\subset\n}$ is a projective system, \textit{i.e.} for all $A\subset B\subset C\subset \n$,
\begin{itemize}
\item $\varrho_{B\setminus A} : L(\Gamma(B)) \rightarrow L(\Gamma(A))$,
\item $\varrho_{A\setminus A}x = x\qquad$ for all $x\in L(\Gamma(A))$,
\item $\varrho_{B\setminus A}\varrho_{C\setminus B} = \varrho_{C\setminus A}$.
\end{itemize}
It is represented for $n = 4$ in figure \ref{fig:projective-system-4}.

\begin{figure}[h!]
\centering
\includegraphics[scale=0.5]{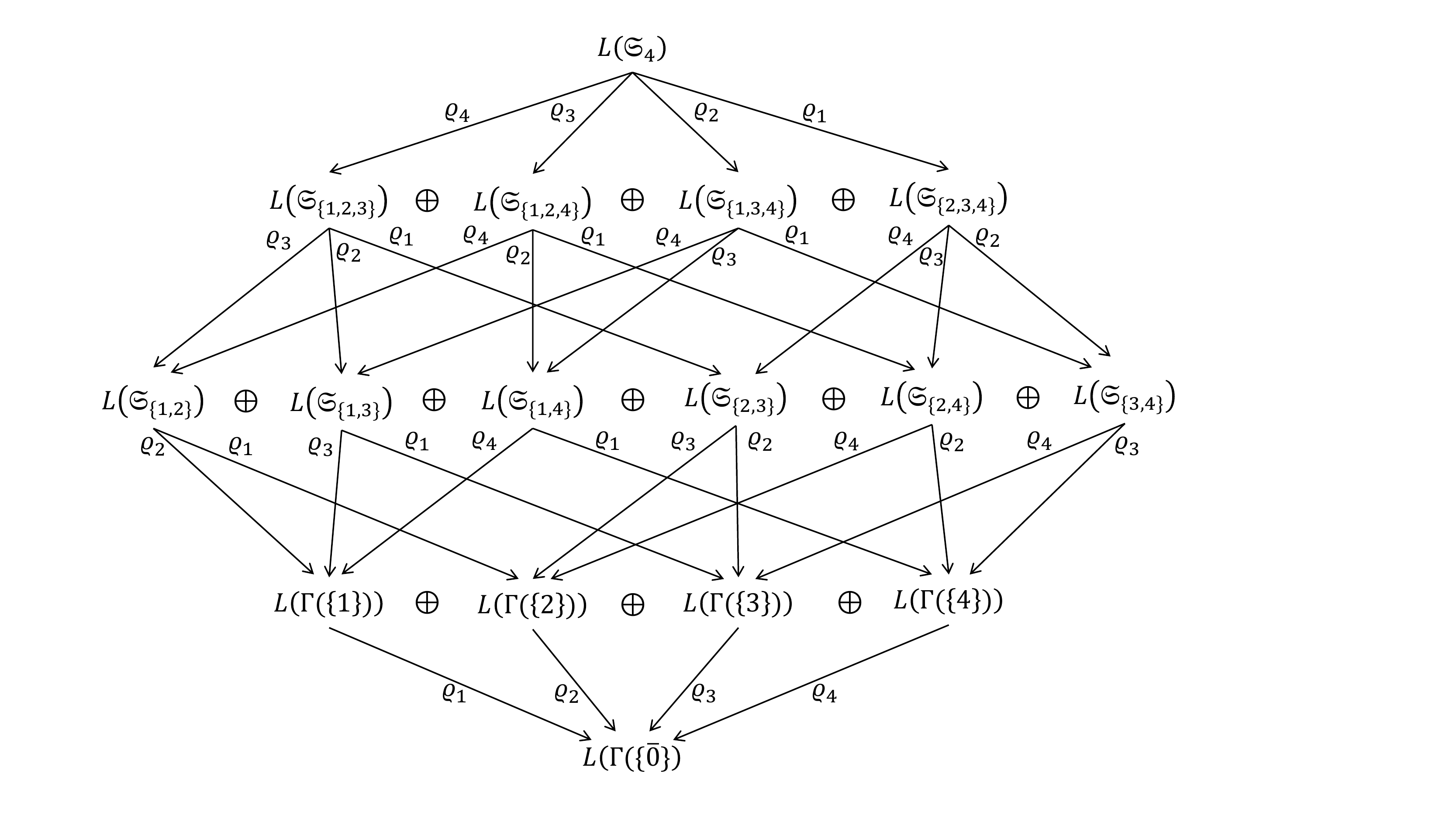}
\caption{Projective system $((L(\Gamma(A)))_{A\subset\n}, (\varrho_{B\setminus A})_{A\subset B\subset\n})$ for $n = 4$}
\label{fig:projective-system-4}
\end{figure}

With these notations, assumption \eqref{rationality} for a family $(f_{A})_{A\in\mathcal{P}(\n)}$ becomes: for $A\in\mathcal{P}(\n)$ with $\vert A\vert < n$ and $b\in\n\setminus A$,
\[
f_{A}(\pi) = \sum_{i=1}^{\vert A\vert + 1}f_{A\cup\{b\}}(\pi\triangleleft_{i}b) = \varrho_{b}f_{A\cup\{b\}}(\pi),
\]
for all $\pi\in\Rank{A}$. The projective system properties then imply more generally that for any $(A,B)\in\mathcal{P}(\n)$ with $A\subset B$,
\[
\varrho_{B\setminus A}f_{B} = f_{A}.
\]

\begin{example}
We keep the same example as before: the number of items is $n=4$ and the observation design is $\mathcal{A} = \{ \{1,3\}, \{2,4\}, \{3,4\}, \{1,2,3\}, \{1,3,4\} \}$. The relations imposed on an element $(f_{A})_{A\in\mathcal{A}}\in\mathbb{M}_{\mathcal{A}}$ are represented in figure \ref{fig:space-M-projectivity}.

\begin{figure}[h!]
\centering
\includegraphics[scale=0.3]{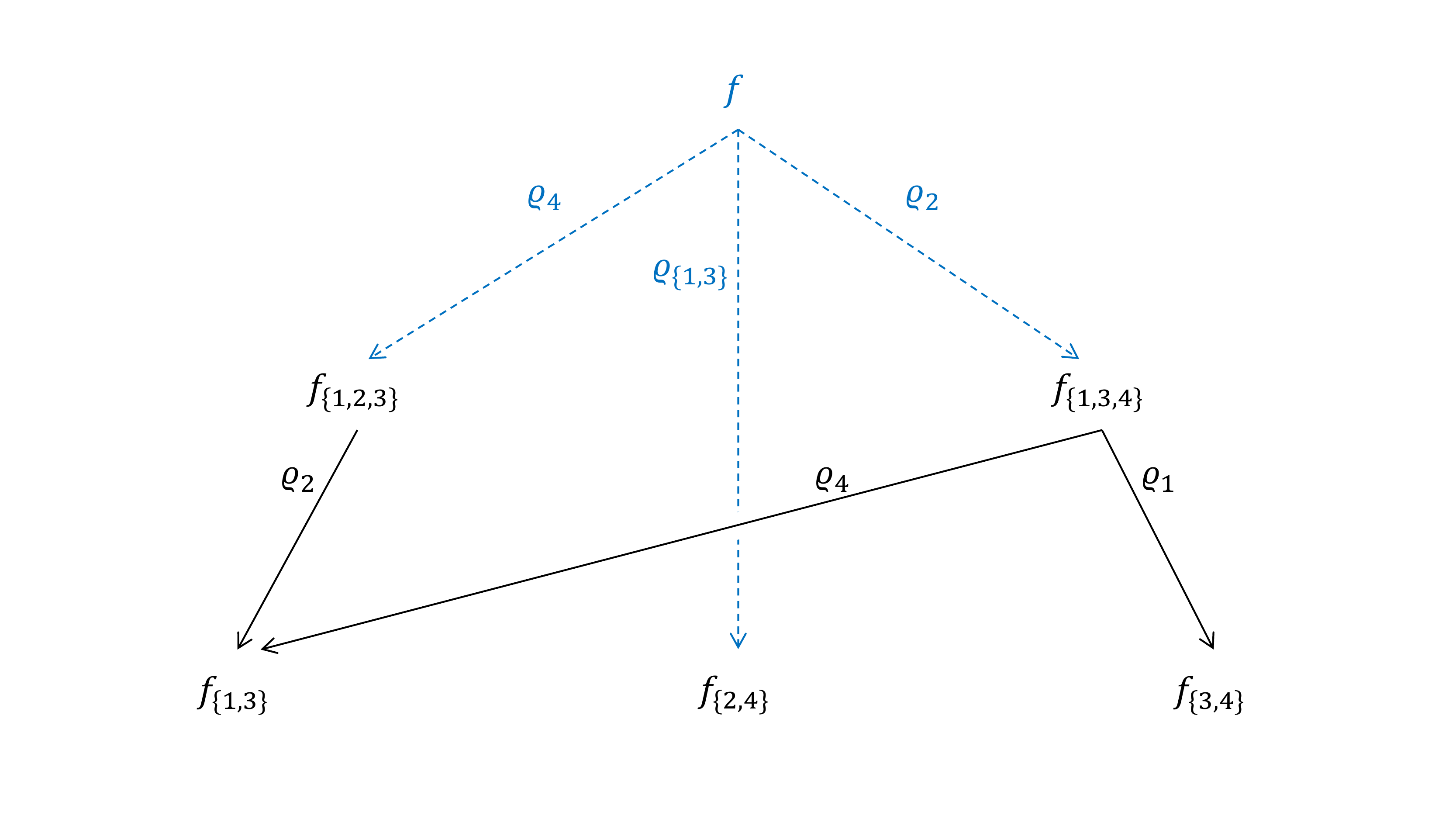}
\caption{Projectivity conditions on $\mathbb{M}_{\mathcal{A}}$ for $\mathcal{A} =\{ \{1,3\}, \{2,4\}, \{3,4\}, \{1,2,3\}, \{1,3,4\} \}$}
\label{fig:space-M-projectivity}
\end{figure}

\end{example}

Now, let $(A,B)\in\mathcal{P}(\n)^{2}$ with $A\subset B$ and $\sigma\in\Rank{B}$. By definition, $\sigma_{\vert A}$ is the word obtained by deleting in $\sigma$ all the elements that are not in $A$, so $\sigma_{\vert A} = \sigma\setminus (B\setminus A) = \varrho_{B\setminus A}\sigma$. In particular for $B = \n$, we have from equation \eqref{marginal-injective-word}
\begin{equation}
\label{characterization marginal}
M_{A} = \varrho_{\n\setminus A}.
\end{equation}
All the marginals operators can thus be expressed in terms of deletion operators. For an element $(f_{A})_{A\in\mathcal{A}}$ of a observation design $\mathcal{A}\subset\mathcal{P}(\n)$, we express each $f_{A}$ as the marginal on $A$ of a function $f$. Our goal is to obtain a decomposition of $f$ into components that have a localized effect on the $f_{A}$'s. More precisely, we want a decomposition of $f$ of the form
\[
f = \tilde{f}_{0} + \sum_{B\in\bigcup_{A\in\mathcal{A}}\mathcal{P}(A)}\tilde{f}_{B}
\]
such that for any $A\in\mathcal{A}$,
\begin{equation}
\label{objective}
f_{A} = M_{A}\left[\tilde{f}_{0} + \sum_{B\in\mathcal{P}(A)}\tilde{f}_{B}\right].
\end{equation}

\begin{example}
Using the same example as before, we represent the principle of the decomposition in figure \ref{fig:decomposition-space-M}.

\begin{figure}[h!]
\centering
\includegraphics[scale=0.35]{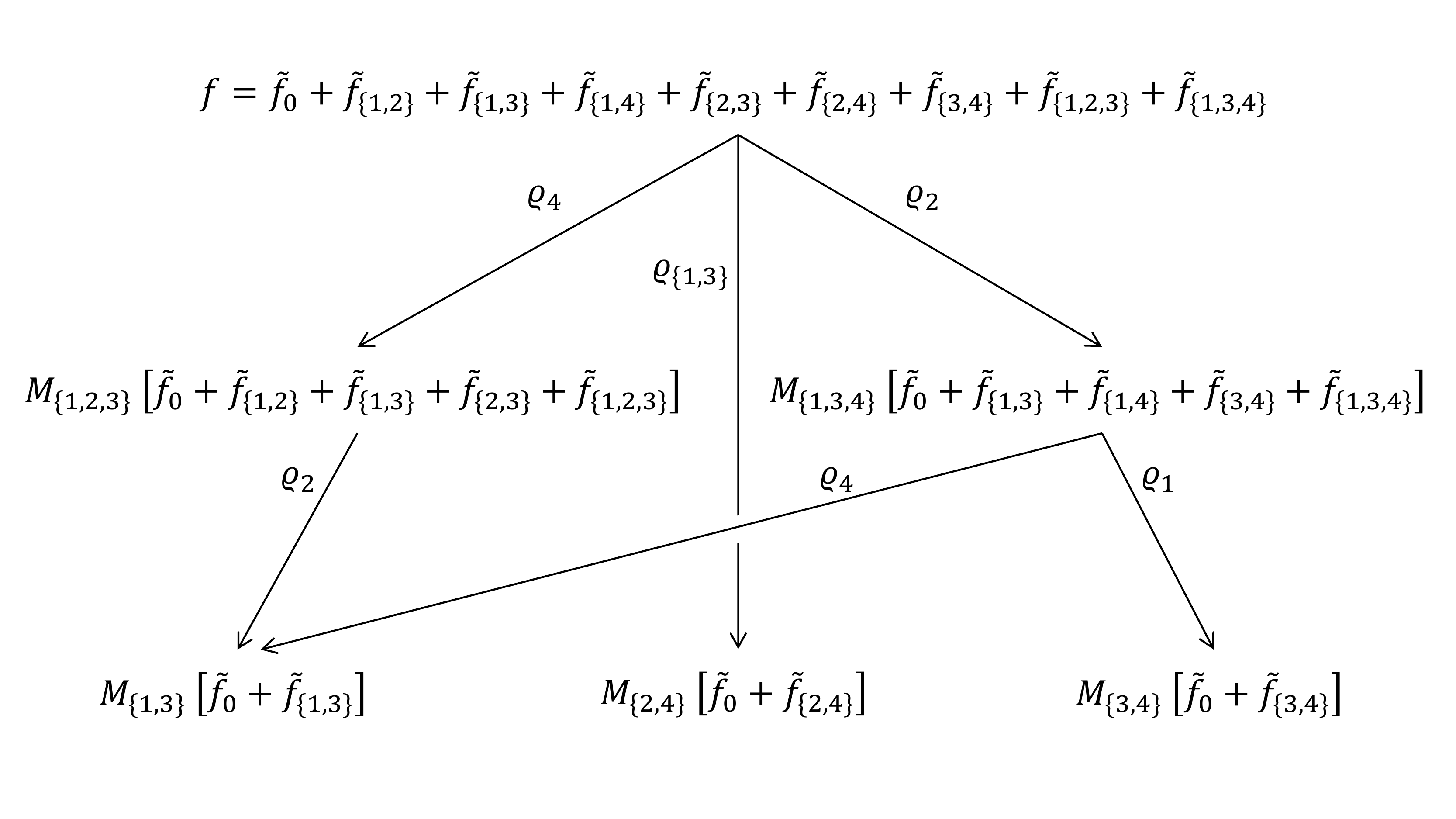}
\caption{Decomposition of a function $f\in L(\Sn)$ adapted to the observation design $\mathcal{A} = \{ \{1,3\}, \{2,4\}, \{3,4\}, \{1,2,3\}, \{1,3,4\} \}$}
\label{fig:decomposition-space-M}
\end{figure}

\end{example}

\section{The multiresolution decomposition}
\label{sec:multiresolution-decomposition}

We now enter the details of the construction of our multiresolution decomposition of $L(\Sn)$ that provides a multiresolution decomposition of $\mathbb{M}_{\mathcal{A}}$ for any observation design $\mathcal{A}\subset\mathcal{P}(\n)$.

\subsection{Requirements for $W_{A}$}
\label{subsec:requirements-space-W}

For $A\in\mathcal{P}(\n)$, we want to construct a subspace $W_{A}$ of $L(\Sn)$ that ``localizes the information that is specific to the marginal on $A$ and not to the others''. The precise definition of this statement relies on the refinement relation on the $\Pi_{B}$'s shown in proposition \ref{refinement}, namely for $B\subset B^{\prime}$, $\Pi_{B^{\prime}}$ is a refinement of $\Pi_{B}$. This implies first that $W_{A}$ cannot contain the entire information related to the marginal on $A$, otherwise it would contain also the entire information related to the marginal on $B$ for all $B\subset A$, which is not specific to $A$. So $W_{A}$ cannot be a supplementary space of $\ker M_{A}$ but we require that all the information it localizes be contained in the marginal on $A$, \textit{i.e.} that $W_{A}\cap\ker M_{A} = \{0\}$. Second, for $B\supset A$, the marginal on $B$ contains all the information related to the marginal on $A$ and \textit{a fortiori} the information localized by $W_{A}$, so we have necessarily $W_{A}\cap\ker M_{B} = \{0\}$. We can require however that $W_{A}\subset\ker M_{B}$ for all $B\in\mathcal{P}(\n)$ such that $B\not\supset A$. We thus want $W_{A}$ to satisfy two conditions:
\begin{enumerate}
	\item it localizes information related to the marginal on $A$, \textit{i.e.}
	\begin{equation}
	\label{condition 1}
	W_{A}\cap\ker M_{A} = \{0\}
	\end{equation}
	\item it localizes information that is not contained in the marginals on $B\in\mathcal{P}(\n)$ for $B\not\supset A$, \textit{i.e.}
	\begin{equation}
	\label{condition 2}
	W_{A}\subset\bigcap_{\substack{B\in\mathcal{P}(\n) \\ B\not\supset A}}\ker M_{B}.
	\end{equation}
\end{enumerate}
Let us first consider the case $A = \n$. The operator $M_{\n}$ is equal to the identity mapping on $L(\Sn)$, so $\ker M_{\n} = \{0\}$ and $W_{\n}$ only needs to satisfy condition \eqref{condition 2}. Since we want $W_{\n}$ to localize all the information that is specific to $M_{\n}$, we define
\[
W_{\n} = \bigcap_{\substack{B\in\mathcal{P}(\n) \\ B\subsetneq \n}}\ker M_{B}.
\]
Using proposition \ref{refinement}, one has $W_{\n} = \bigcap_{\vert B\vert = n-1}\ker M_{B}$. Now, if $\vert B\vert = n-1$, $\n\setminus B$ is necessarily of the form $\{a\}$ with $a\in\n$. Thus, using equation \eqref{characterization marginal}, we obtain
\[
W_{\n} = \{ x\in L(\Sn) \;\vert\; \varrho_{a}(x) = 0 \text{ for all } a\in \n\}.
\]
More generally for $A\in\mathcal{P}(\n)$, let $H_{A}$ be the space
\begin{equation}
\label{espace-H}
H_{A} = \{ x\in L(\Gamma(A)) \;\vert\; \varrho_{a}(x) = 0 \text{ for all } a\in A\}.
\end{equation}
Seeing $L(\Gamma(A))$ as the space of marginals on $A$, the space $H_{A}$ contains, among the information related to marginals on $A$, the information that is specific to $A$ and not to subsets $B\subsetneq A$. This is exactly the information that we want $W_{A}$ to localize. But the elements of $H_{A}$ are chains on words with content $A$, not $\n$, and $H_{A}$ is not a subspace of $L(\Sn)$. The space $W_{A}$ must then be constructed as an embedding of $H_{A}$ into $L(\Sn)$. The choice of the embedding can nonetheless not be arbitrary if we want $W_{A}$ to satisfy conditions \eqref{condition 1} and \eqref{condition 2}. We are looking for a linear operator $\phi_{n} : L(\Gamma_{n}) \rightarrow L(\Sn)$ such that for all $A\in\mathcal{P}(\n)$,
\[
\phi_{n}(H_{A})\cap\ker M_{A} = \{0\} \qquad\text{and}\qquad \phi_{n}(H_{A})\subset\bigcap_{\substack{B\in\mathcal{P}(\n) \\ B\not\supset A}}\ker M_{B}.
\]
The rationale behind this is that we want $\phi_{n}$ to ``pull up'' all the information contained in $H_{A}$ in $L(\Sn)$ in a way that it does not impact the marginals on the subsets $B\in\mathcal{P}(\n)$ such that $B\not\supset A$. Mapping the Dirac function of a $\pi\in\Rank{A}$ in $L(\Gamma_{n})$ to an element of $L(\Sn)$ involves necessarily the insertion of the missing items $\n\setminus A$. But this can be done in many different ways. In the case where $A = \n\setminus\{b\}$ with $b\in\n$, the insertion of $b$ in an element $\pi\in\Rank{A}$ can be done at any of the $n$ positions. More generally for $\vert A\vert = k$, the number of ways to insert the items of $\n\setminus A$ in an element $\pi\in\Rank{A}$ is equal to $n!/k!$. Perhaps the most natural embedding is to insert the items in all possible ways. The embedding operator would then be defined on the Dirac function of a word $\pi\in\Gamma^{n-1}$ with $\n\setminus c(\pi) = \{b\}$ by
\[
\phi_{n}^{\prime}\pi = \sum_{i=1}^{n}\pi\triangleleft_{i}b,
\]
and more generally on the Dirac function of any $\pi\in\Gamma_{n}$ by
\[
\phi_{n}^{\prime}\pi = \1{\Sn(\pi)} = \sum_{\sigma\in\Sn(\pi)}\sigma.
\]
For $A\in\mathcal{P}(\n)$ and $\pi\in\Rank{A}$, we have
\[
M_{A}\phi_{n}^{\prime} \pi = M_{A}\1{\Sn(\pi)} = \sum_{\sigma\in\Sn(\pi)}\sigma_{\vert A} = \frac{n!}{\vert A\vert!}\pi, \quad\text{thus}\quad \phi(H_{A})\cap\ker M_{A} = \{0\},
\]
but for $B\in\mathcal{P}(\n)$ such that $B\not\subset A$ and $x\in H_{A}\setminus\{0\}$, $M_{B}\phi_{n}^{\prime} x \neq 0$. This can be shown in the general case but it is not necessary here. We just consider a simple example to give some insights, we take $n=3$, $A = \{1,2\}$ and $B=\{1,3\}$. By definition, $H_{\{1,2\}}$ is the space of chains of the form $\alpha.12 + \beta.21$ such that $\alpha+\beta = 0$. It is thus spanned by the chain $12-21$ and we have
\begin{align*}
M_{\{1,3\}}\phi_{3}^{\prime}(12-21) 	&= \varrho_{2}[(312+132+123)-(321+231+213)] \\
															&= 31+2.13-2.31-13\\
															&=13-31 \\
															&\neq 0.
\end{align*}
This is due to the fact that when deleting $2$ in $132$ and $123$ (or in $321$ and $231$), we obtain twice the same result. More generally, it is easy to see that for  $A = \n\setminus\{a\}$ and $B = \n\setminus\{b\}$ with $(a,b)\in\n^{2}$, $a\neq b$, and $\pi\in\Rank{A}$,
\[
M_{B}\phi_{n}^{\prime} \pi  = \phi_{B}^{\prime}\varrho_{b}\pi + \pi_{b\rightarrow a},
\]
where $\phi_{B}^{\prime}$ is the linear operator $L(\Gamma_{n}) \rightarrow L(\Gamma(B))$ defined on the Dirac functions by $\pi\mapsto \1{\Rank{B}(\pi)}$ if $c(\pi)\subset B$ and $0$ otherwise, and $\pi_{b\rightarrow a}$ is the word obtained when replacing $b$ by $a$ in $\pi$. This implies that for $x\in H_{A}$,
\[
M_{B}\phi_{n}^{\prime} x = \phi_{B}^{\prime}\varrho_{b}x + \sum_{\pi\in\Rank{A}}x(\pi)\pi_{b\rightarrow a} = \sum_{\pi\in\Rank{A}}x(\pi)\pi_{b\rightarrow a},
\]
because, since $b\in A$, $\varrho_{b}x = 0$ by definition of $H_{A}$. Now, it is clear that the mapping defined on the Dirac functions by $\pi \mapsto \pi_{b\rightarrow a}$ induces a bijection from $L(\Gamma(A))$ to $L(\Gamma(B))$. So if $x\neq 0$, then $M_{B}\phi_{n}^{\prime} x \neq 0$. This extends to any couple of subsets $(A,B)\in\mathcal{P}(\n)^{2}$ such that $B\not\subset A$, and implies that we cannot take $\phi_{n}^{\prime}$ as embedding operator.

\subsection{Construction of $W_{A}$}
\label{subsec:construction-space-W}

The definition of our embedding operator $\phi_{n}$ requires a supplementary definition. A contiguous subword of a word $\omega = \omega_{1}\dots\omega_{k}\in\Gamma_{n}$ is an expression $\omega_{i}\omega_{i+1}\dots\omega_{i+j}$, with $1\leq i < i+j \leq k$. For $(A,B)\in\mathcal{P}(\n)^{2}$ with $A\subset B$ and $\pi\in\Rank{A}$, we denote by $\Rank{B}[\pi]$ the set of all the words $\sigma\in\Gamma(B)$ that contain $\pi$ as a contiguous subword. For $B = \n$, we denote it by $\Sn[\pi]$ instead of $\Rank{\n}[\pi]$. A contiguous subword being \textit{a fortiori} a subword, $\Rank{B}[\pi] \subset \Rank{B}(\pi)$.

\begin{definition}[Embedding operator $\phi_{n}$ and space $W_{A}$]
Let $\phi_{n}$ be the linear operator $L(\Gamma_{n}) \rightarrow L(\Sn)$ defined on Dirac functions by
\[
\phi_{n} : \pi \mapsto \1{\Sn[\pi]} = \sum_{\sigma\in\Sn[\pi]}\sigma,
\]
and for $A\in\mathcal{P}(\n)$, let $W_{A}$ be the image of $H_{A}$ by $\phi_{n}$, \textit{i.e.}
\[
W_{A} = \phi_{n}(H_{A}).
\]
\end{definition}

\begin{proposition}[Information localization]
\label{information localization}
For $A\in\mathcal{P}(\n)$, $W_{A}$ satisfies conditions \eqref{condition 1} and \eqref{condition 2}:
\[
W_{A}\cap\ker M_{A} = \{0\} \qquad\text{and}\qquad W_{A}\subset\bigcap_{\substack{B\in\mathcal{P}(\n) \\ B\not\supset A}}\ker M_{B}.
\]
\end{proposition}

Proposition \ref{information localization} is the major result of this subsection. Not only does it show that spaces $W_{A}$ satisfy the good information localization properties, but it is also one of the key results to prove our multiresolution decomposition of $L(\Sn)$. Its proof relies on the combinatorial properties of operator $\phi_{n}$ and requires some additional definitions.

\begin{definition}[Concatenation product]
The concatenation product of two injective words $\pi = a_{1}\dots a_{r}$ and $\pi^{\prime} = b_{1}\dots b_{s}$ such that $c(\pi)\cap c(\pi^{\prime}) = \emptyset$ is the word $\pi\pi^{\prime} = a_{1}\dots a_{r}b_{1}\dots b_{s}$. It is extended as the bilinear operator $L(\Gamma_{n})\times L(\Gamma_{n})\rightarrow L(\Gamma_{n})$ defined on Dirac functions by
\begin{equation*}
(\pi,\pi^{\prime}) \mapsto \left\{
\begin{aligned}
 \pi\pi^{\prime} \qquad &\text{if }c(\pi)\cap c(\pi^{\prime}) = \emptyset,\\
 0\qquad &\text{otherwise}.
\end{aligned}
\right.
\end{equation*}
\end{definition}

Starting from a word $\pi\in\Gamma_{n}$, the words of $\Rank{B}[\pi]$ for $B\in\mathcal{P}(\n)$ with $c(\pi)\subset B$ are obtained by inserting the elements of $B\setminus c(\pi)$ in all possible ways, either before or after $\pi$, but not inside. Thus it is clear that
\[
\Rank{B}[\pi] = \{ \omega_{1}\pi\omega_{2} \;\vert\; (\omega_{1}, \omega_{2})\in\Gamma(B)^{2},\ c(\omega_{1})\sqcup c(\omega_{2}) = B\setminus c(\pi) \}
\]
and $\vert \Rank{B}[\pi]\vert = (\vert B\vert - \vert \pi\vert + 1)!$.

\begin{example}
\[
\Sym{5}[143] = \{ 25143, 52143, 21435, 51432, 14325, 14352 \}.
\]
\end{example}

The concatenation product for chains allows us to give an even simpler formula for the indicator function of the set $\Rank{B}[\pi]$: $\1{\Rank{B}[\pi]} = \{ \omega_{1}\pi\omega_{2} \;\vert\; (\omega_{1}, \omega_{2})\in\Gamma(B)^{2} \}$. For $\omega\in\Gamma_{n}$, let $\mathfrak{i}_{\omega}$ and $\mathfrak{j}_{\omega}$ be the two operators on $L(\Gamma_{n})$ defined on the Dirac functions by
\begin{equation}
\mathfrak{i}_{\omega} : \pi \mapsto \omega\pi \qquad\text{and}\qquad\mathfrak{j}_{\omega} : \pi \mapsto \pi\omega.
\end{equation}
Operator $\mathfrak{i}_{\omega}$ is simply the insertion of the word $\omega$ at the beginning and $\mathfrak{j}_{\omega}$ at the end. It is clear that they commute and that for all $\pi\in\Gamma_{n}$, $\1{\Rank{B}[\pi]} = \sum_{\omega_{1}, \omega_{2}\in \Gamma(B)^{2}}\mathfrak{i}_{\omega_{1}}\mathfrak{j}_{\omega_{2}}\pi$. This formulation shows that the embedding operator $\phi_{n}$ is simply the sum of operators $\mathfrak{i}_{\omega_{1}}\mathfrak{j}_{\omega_{2}}$ for all $(\omega_{1}, \omega_{2})\in (\Gamma_{n})^{2}$:
\begin{equation}
\label{characterization phi}
\phi_{n} = \sum_{\omega_{1}, \omega_{2}\in \Gamma_{n}}\mathfrak{i}_{\omega_{1}}\mathfrak{j}_{\omega_{2}}.
\end{equation}
Now, the proof of proposition \ref{information localization} relies on this simple but crucial lemma. The technical proof can be found in the Appendix section.

\begin{lemma}
\label{commutation}
For $\omega\in\Gamma_{n}$ and $a\in\n\setminus c(\omega)$,
\[
\varrho_{a}\mathfrak{i}_{\omega} = \mathfrak{i}_{\omega}\varrho_{a} \qquad\text{and}\qquad \varrho_{a}\mathfrak{j}_{\omega} = \mathfrak{j}_{\omega}\varrho_{a}.
\]
\end{lemma}

\begin{proof}[Proof of proposition \ref{information localization}]
Let $A\in\mathcal{P}(\n)$, $f\in W_{A}$ and $x\in H_{A}$ such that $f = \phi_{n} x$. We have
\[
M_{A}f = M_{A}\phi_{n} x = \sum_{\pi\in\Rank{A}}x(\pi)M_{A}\1{\Sn[\pi]} = (n-\vert A\vert +1)!\sum_{\pi\in\Rank{A}}x(\pi)\pi = (n-\vert A\vert +1)!\ x,
\]
because $\Sn[\pi]\subset\Sn(\pi)$ and $\vert\Sn[\pi]\vert = (n-\vert A\vert+1)!$ for all $\pi\in\Rank{A}$. Therefore if $f\in\ker M_{A}$, $x = 0$ and so $f = 0$. This proves that $W_{A}\cap\ker M_{A} = \{0\}$. To prove the second part, first observe that if $c(\omega)\cap A \neq \emptyset$, $\mathfrak{i}_{\omega}\pi = 0$ for all $\pi\in\Rank{A}$ and thus $\mathfrak{i}_{\omega}x = 0$ (equivalently, $\mathfrak{j}_{\omega}x = 0$). Hence, using equation \eqref{characterization phi}, we have
\[
\phi_{n} x = \sum_{\omega_{1}, \omega_{2}\in \Gamma_{n}}\mathfrak{i}_{\omega_{1}}\mathfrak{j}_{\omega_{2}} x = \sum_{\omega_{1}, \omega_{2}\in \Gamma(\n\setminus A)}\mathfrak{i}_{\omega_{1}}\mathfrak{j}_{\omega_{2}} x.
\]
Now, let $B\in\mathcal{P}(\n)$ such that $B\not\supset A$. We want to show that $M_{B}f = 0$, \textit{i.e.} that $\varrho_{\n\setminus B}\phi_{n} x = 0$. Since $B\not\supset A$, there exists $a\in A$ such that $a\not\in B$, and we can write$\varrho_{\n\setminus B} = \varrho_{B^{\prime}}\varrho_{a}$. Then using lemma \ref{commutation},
\[
\varrho_{\n\setminus B}\phi_{n} x = \varrho_{B^{\prime}}\sum_{\omega_{1}, \omega_{2}\in \Gamma(\n\setminus A)}\varrho_{a}\mathfrak{i}_{\omega_{1}}\mathfrak{j}_{\omega_{2}}x = \varrho_{B^{\prime}}\sum_{\omega_{1}, \omega_{2}\in \Gamma(\n\setminus A)}\mathfrak{i}_{\omega_{1}}\mathfrak{j}_{\omega_{2}}\varrho_{a}x = 0,
\]
because $\varrho_{a}x = 0$ by definition of $H_{A}$.
\end{proof}

\subsection{The decomposition of $L(\Sn)$}
\label{subsec:global-decomposition}

Now that we have constructed the subspaces of $L(\Sn)$ that localize the information specific to each marginal, we show that they constitute a decomposition of the space $L(\Sn)$. Recall that $V^{0}$ is the subspace of $L(\Sn)$ of constant functions. So defining $L_{0}(\Sn) = \{ f\in L(\Sn) \;\vert\; \sum_{\sigma\in\Sn}f(\sigma) = 0\}$, we have
\begin{equation}
\label{decomposition constant functions}
L(\Sn) = V^{0}\oplus L_{0}(\Sn).
\end{equation}

\begin{proposition}
\label{direct sum}
The spaces $(W_{A})_{A\in\mathcal{P}(\n)}$ are in direct sum in $L_{0}(\Sn)$.
\end{proposition}

\begin{proof}
First, observe that for $A\in\mathcal{P}(\n)$ and $x\in H_{A}$,
\[
\sum_{\sigma\in\Sn}(\phi_{n}x)(\sigma) = \sum_{\sigma\in\Sn}\sum_{\pi\in\Rank{A}}x(\pi)\1{\Sn[\pi]}(\sigma) = (n - \vert A\vert + 1)!\sum_{\pi\in\Rank{A}}x(\pi) = 0,
\]
because as $x\in H_{A}$, $0 = \varrho_{A}x = \left[\sum_{\pi\in\Rank{A}}x(\pi)\right]\overline{0}$. Hence, $W_{A}\subset L_{0}(\Sn)$. To prove that the spaces $W_{A}$ are in direct sum, let $(f_{A})_{A\in\mathcal{P}(\n)}$ be a family of functions with $f_{A}\in W_{A}$ for each $A\in\mathcal{P}(\n)$, such that
\begin{equation}
\label{sum equal 0}
\sum_{A\in\mathcal{P}(\n)}f_{A} = 0.
\end{equation}
We need to show that $f_{A} = 0$ for all $A\in\mathcal{P}(\n)$. We proceed by induction on the cardinality of $A$. Let $A\in\binom{\n}{2}$. For all $B\in\mathcal{P}(\n)$ different from $A$, we have $A\not\supset B$. Thus, using the second part of proposition \ref{information localization}, $M_{A}f_{B} = 0$ for all $B\in\mathcal{P}(\n)\setminus\{A\}$. Applying $M_{A}$ in equation \eqref{sum equal 0} then gives $M_{A}f_{A} = 0$. This means that $f_{A}\in W_{A}\cap\ker M_{A}$ and so that $f_{A} = 0$, using the first part of proposition \ref{information localization}. Now assume that $f_{A} = 0$ for all $A\in\mathcal{P}(\n)$ such that $\vert A\vert \leq k-1$, with $k\in\{3, \dots, n\}$. Equation \eqref{sum equal 0} then becomes
\begin{equation}
\label{sum equal 0  - 2}
\sum_{\vert A\vert \geq k}f_{A} = 0.
\end{equation}
Let $A\in\binom{\n}{k}$. For all $B\subset\n$ such that $\vert B\vert \geq k$ and different from $A$, we have $A\not\supset B$. Thus, using again proposition \ref{information localization}, $M_{A}f_{B} = 0$, and applying this to equation \eqref{sum equal 0  - 2} gives $M_{A}f_{A}$. We conclude using proposition \ref{information localization} one more time.
\end{proof}

The second step in the proof of our decomposition is a dimensional argument. Notice that for $A\in\binom{\n}{k}$ with $k\in\{ 2, \dots, n\}$, $H_{A}$ is isomorphic to the space
\[
H_{k} = \{ x\in L(\Gamma(\{1, \dots, k\})) \;\vert\; \varrho_{i}x = 0 \text{ for all } i\in\{1, \dots, k\} \}.
\]
Now, it happens that this space is actually closely related to another well-studied space in the algebraic topology literature, namely the top homology space of the complex of injective words (see \cite{Farmer78}, \cite{BW83}, \cite{Reiner04}, \cite{Hanlon04}). The link is made in \cite{Reiner13} (the space $H_{k}$ is denoted by $\ker\pi_{\nk}$), and leads in particular to the following result (see proposition 6.8 and corollary 6.15).

\begin{theorem}[Dimension of $H_{k}$]
\label{topology}
For $k\in\{ 2, \dots, n\}$,
\[
\dim H_{k} = d_{k},
\]
where $d_{k}$ is the number of fixed-point free permutations (also called derangements) on the set $\{ 1, \dots, k\}$.
\end{theorem}

As simple as it may seem, this result is far from being trivial. Its proof relies on the topological nature of the partial order of subword inclusion on the complex of injective words and the use of the Hopf trace formula for virtual characters. It is a cornerstone in the proof of our multiresolution decomposition.

\begin{theorem}[Multiresolution decomposition]
\label{item-scale decomposition}
The following decomposition of $L(\Sn)$ holds:
\[
L(\Sn) = V^{0}\oplus\bigoplus_{A\in\mathcal{P}_{2}(\n)} W_{A}.
\]
In addition, $\dim W_{A} = d_{\vert A\vert}$ and $\phi_{n}(H_{A}) = W_{A}$ for all $A\in\mathcal{P}(\n)$.
\end{theorem}

\begin{proof}
For $A\in\mathcal{P}(\n)$ and $x\in H_{A}$, $\phi_{n} x = \sum_{\pi\in\Rank{A}}x(\pi)\1{\Sn[\pi]}$ by definition. Since for $(\pi,\pi^{\prime})\in(\Rank{A})^{2}$ such that $\pi\neq\pi^{\prime}$, the sets $\Sn[\pi]$ and $\Sn[\pi^{\prime}]$ are disjoint, it is clear that $\phi_{n} x = 0 \Rightarrow x(\pi) = 0$ for all $\pi\in\Rank{A}$, \textit{i.e.} $x = 0$. This proves that the restriction of $\phi_{n}$ to $H_{A}$ is injective, and thus that $\dim W_{A} \geq \dim H_{A}$, \textit{i.e.} $\dim W_{A} \geq d_{\vert A\vert}$, using theorem \ref{topology}. Now, using proposition \ref{direct sum} and equation \eqref{decomposition constant functions}, we obtain
\[
\dim \left[V^{0}\oplus\bigoplus_{A\in\mathcal{P}_{2}(\n)} W_{A}\right] \geq 1 + \sum_{k=2}^{n}\binom{n}{k}d_{k} = \sum_{k=0}^{n}\binom{n}{k}d_{n-k} = n!,
\]
where the last equality results from the observation that the number of permutations with $k$ fixed points is equal to $\textstyle{\binom{n}{k}}d_{n-k}$. Since $\dim L(\Sn) = n!$, this concludes both the proof of the decomposition of $L(\Sn)$ and the dimension of $W_{A}$, and the fact that $\phi_{n}(H_{A}) = W_{A}$ follows.
\end{proof}

This decomposition appears implicitly in \cite{Reiner13}, in the combination of theorem 6.20 and formula (22). It is however defined modulo isomorphism, and not easily usable for applications. Our explicit construction permits a practical use of this decomposition. In particular, it allows to localize the information related to any observation design $\mathcal{A}\subset\mathcal{P}(\n)$, as declared in the introduction.

\begin{corollary}
\label{information decomposition}
For any subset $A\in\mathcal{P}(\n)$,
\begin{equation*}
L(\Sn) = \ker M_{A} \oplus \left[V^{0}\oplus\bigoplus_{B\in\mathcal{P}(A)} W_{B}\right],
\end{equation*}
and for any observation design $\mathcal{A}\subset\mathcal{P}(\n)$,
\begin{equation*}
L(\Sn) = \ker M_{\mathcal{A}} \oplus \left[V^{0}\oplus\bigoplus_{B\in\bigcup_{A\in\mathcal{A}}\mathcal{P}(A)} W_{B}\right].
\end{equation*}
\end{corollary}

\begin{proof}
Let $A\in\mathcal{P}(\n)$. By theorem \ref{item-scale decomposition}, we have
\[
\ker M_{A} = \left(\ker M_{A}\cap V^{0}\right)\oplus\bigoplus_{B\in\mathcal{P}(\n)}\left(\ker M_{A}\cap W_{B}\right).
\]
It is clear that $\ker M_{A}\cap V^{0} = \{0\}$ ($M_{A}$ maps constant functions on $\Sn$ to constant functions on $\Rank{A}$). Moreover, for $B\in\mathcal{P}(A)$, $\ker M_{A}\cap W_{B}\subset\ker M_{B}\cap W_{B} = \{0\}$, using proposition \ref{refinement} and the first part of proposition \ref{information localization}. At last, for all $B\in\mathcal{P}(\n)\setminus\mathcal{P}(A)$, $A\not\supset B$, and thus $W_{B}\subset \ker M_{A}$, using the second part of proposition \ref{information localization}. This means that $\ker M_{A} = \bigoplus_{B\in\mathcal{P}(\n)\setminus\mathcal{P}(A)}W_{B}$, and the first part of corollary \ref{information decomposition} follows. The second part results from the calculation
\[
\ker M_{\mathcal{A}} = \bigcap_{A\in\mathcal{A}}\ker M_{A} = \bigcap_{A\in\mathcal{A}}\bigoplus_{B\in\mathcal{P}(\n)\setminus\mathcal{P}(A)}W_{B} = \bigoplus_{B\in\mathcal{P}(\n)\setminus\bigcup_{A\in\mathcal{A}}\mathcal{P}(A)}W_{B}.
\]
\end{proof}

\begin{example}
Let's consider an example with $n=4$. The multiresolution decomposition of $L(\Sym{4})$ is given by the following representation.
\begin{center}
$W_{\{1,2,3,4\}}$\\
\vspace{0.1cm}
$W_{\{ 1,2,3\}} \oplus W_{\{ 1,2,4\}} \oplus W_{\{ 1,3,4\}} \oplus W_{\{ 2,3,4\}}$\\
\vspace{0.1cm}
$W_{\{ 1,2\}} \oplus W_{\{ 1,3\}} \oplus W_{\{ 1,4\}} \oplus W_{\{ 2,3\}} \oplus W_{\{ 2,4\}} \oplus W_{\{ 3,4\}}$\\
\vspace{0.1cm}
$V^{0}$
\end{center}
The spaces in bold contain the information related to the observation of marginals on $\{1,2,3\}$ in this representation,
\begin{center}
$W_{\{1,2,3,4\}}$\\
\vspace{0.1cm}
$\mathbf{W_{\{ 1,2,3\}}} \oplus W_{\{ 1,2,4\}} \oplus W_{\{ 1,3,4\}} \oplus W_{\{ 2,3,4\}}$\\
\vspace{0.1cm}
$\mathbf{W_{\{ 1,2\}}} \oplus \mathbf{W_{\{ 1,3\}}} \oplus W_{\{ 1,4\}} \oplus \mathbf{W_{\{ 2,3\}}} \oplus W_{\{ 2,4\}} \oplus W_{\{ 3,4\}}$\\
\vspace{0.1cm}
$\mathbf{V^{0}}$
\end{center}
to the observation of marginals on $\{1,3,4\}$ in this one,
\begin{center}
$W_{\{1,2,3,4\}}$\\
\vspace{0.1cm}
$W_{\{ 1,2,3\}} \oplus W_{\{ 1,2,4\}} \oplus \mathbf{W_{\{ 1,3,4\}}} \oplus W_{\{ 2,3,4\}}$\\
\vspace{0.1cm}
$W_{\{ 1,2\}} \oplus \mathbf{W_{\{ 1,3\}}} \oplus \mathbf{W_{\{ 1,4\}}} \oplus W_{\{ 2,3\}} \oplus W_{\{ 2,4\}} \oplus \mathbf{W_{\{ 3,4\}}}$\\
\vspace{0.1cm}
$\mathbf{V^{0}}$
\end{center}
and to the observation of marginals of the observation design $\{\{1,2,3\}, \{1,3,4\}\}$ in this final representation.
\begin{center}
$W_{\{1,2,3,4\}}$\\
\vspace{0.1cm}
$\mathbf{W_{\{ 1,2,3\}}} \oplus W_{\{ 1,2,4\}} \oplus \mathbf{W_{\{ 1,3,4\}}} \oplus W_{\{ 2,3,4\}}$\\
\vspace{0.1cm}
$\mathbf{W_{\{ 1,2\}}} \oplus \mathbf{W_{\{ 1,3\}}} \oplus \mathbf{W_{\{ 1,4\}}} \oplus \mathbf{W_{\{ 2,3\}}} \oplus W_{\{ 2,4\}} \oplus \mathbf{W_{\{ 3,4\}}}$\\
\vspace{0.1cm}
$\mathbf{V^{0}}$
\end{center}
\end{example}

From a practical point of view, if we observe $(f_{A})_{A\in\mathcal{A}}\in\mathbb{M}_{\mathcal{A}}$ then by corollary \ref{information decomposition}, there exists a unique $f\in V^{0}\oplus\bigoplus_{B\in\bigcup_{A\in\mathcal{A}}\mathcal{P}(A)} W_{B}$ such that $M_{\mathcal{A}}f = (f_{A})_{A\in\mathcal{A}}$. Furthermore, if
\[
f = \tilde{f}_{0} + \sum_{B\in\bigcup_{A\in\mathcal{A}}\mathcal{P}(A)}\tilde{f}_{B}
\]
is the decomposition of $f$ corresponding to $\bigoplus_{B\in\bigcup_{A\in\mathcal{A}}\mathcal{P}(A)} W_{B}$, we obtain the wanted relation \eqref{objective}: for any $A\in\mathcal{A}$,
\[
f_{A} = M_{A}\left[\tilde{f}_{0} + \sum_{B\in\mathcal{P}(A)}\tilde{f}_{B}\right].
\]

\subsection{Multiresolution analysis}
\label{subsec:multiresolution-analysis}

Until now, we have only used the expression ``multiresolution decomposition'', not ``multiresolution analysis''. The latter has indeed a specific mathematical definitions, first formalized in \cite{Meyer1992} and \cite{MallatMulti} (it is called ``multiresolution approximation'' in the latter) for the space $L^{2}(\mathbb{R})$. A multiresolution analysis of $L^{2}(\mathbb{R})$ is a sequence $(\tilde{V}^{j})_{j\in\mathbb{Z}}$ of closed subspaces of $L^{2}(\mathbb{R})$ such that:
\begin{enumerate}
	\item $\tilde{V}^{j}\subset \tilde{V}^{j+1}$ for all $j\in\mathbb{Z}$
	\item $\bigcup_{j\in\mathbb{Z}}\tilde{V}^{j} = L^{2}(\mathbb{R})$ and $\bigcap_{j\in\mathbb{Z}}\tilde{V}^{j} = \{0\}$
	\item $f(x)\in \tilde{V}^{j} \Leftrightarrow f(2x)\in \tilde{V}^{j+1}$ for all $j\in\mathbb{Z}$
	\item $f(x)\in \tilde{V}^{j} \Leftrightarrow f(x-2^{-j}k)\in \tilde{V}^{j}$ for all $k\in\mathbb{Z}$
	\item There exists $g\in \tilde{V}^{0}$ such that $(g(x-k))_{k\in\mathbb{Z}}$ is a Riesz basis of $\tilde{V}^{0}$.
\end{enumerate}
In order to define an analogous definition for $L(\Sn)$, we get back to the general principles behind it. The idea is that the index $j$ represents a scale, and each space $\tilde{V}^{j}$ contains the information of all scales lower than $j$, thus $\tilde{V}^{j}\subset \tilde{V}^{j+1}$. In finite dimension, the number of scales is necessarily finite, and we request that the space of largest scale be equal to the full space (we can request that the space of lower scale be $\{0\}$ but it is useless). The principle of multiresolution analysis is not only to have a nested sequence of subspaces corresponding to different scales, but also to define the operators that leave a space $\tilde{V}^{j}$ invariant and the ones that send from a space $\tilde{V}^{j}$ to $\tilde{V}^{j+1}$ and vice versa. In the case of $L^{2}(\mathbb{R})$, these operators are respectively the scaled translation $f(x) \mapsto f(x-2^{-j}k)$ and the dilation $f(x) \mapsto f(2x)$, defined in conditions 3. and 4. If we see a function $f$ as an image, the dilation corresponds to a zoom, and a scaled translation corresponds to a displacement.

To define a multiresolution analysis in our case, we first need a notion of scale. In our construction, the natural notion of scale for the spaces $W_{A}$'s appears clearly on the precedent representations of the multiresolution decomposition of $L(\Sn)$: the cardinality of the indexing subsets. We say that the marginal $p_{A}$ of a probability distribution $p$ on $\Sn$ on a subset $A\in\mathcal{P}(\n)$ is of scale $k$ if $\vert A\vert = k$. This means that $p_{A}$ is a probability distribution over rankings involving $k$ items. The information contained in $p_{A}$ can be decomposed in components of scales $\leq k$, and the projection of $p$ on $W_{A}$ contains the information of scale $k$. For $k\in\{ 2, \dots, n\}$, we define the space $W^{k}$ that contains all the information of scale $k$ by
\begin{equation}
\label{definition Wk}
W^{k} = \bigoplus_{\vert A\vert = k}W_{A}
\end{equation}
and the space $V^{k}$ that contains all the information of scales $\leq k$ by
\begin{equation}
\label{definition Vk}
V^{k} = V^{0}\oplus\bigoplus_{j=2}^{k}\bigoplus_{\vert A\vert = j}W_{A}.
\end{equation}
We thus have
\begin{equation}
\label{nested sequence}
V^{0} \subset V^{2} \subset V^{3} \subset \dots \subset V^{n} = L(\Sn) \qquad\text{and}\qquad L(\Sn) = V^{0}\oplus\bigoplus_{k=2}^{n}W^{k}.
\end{equation}
The space $W^{k}$ represent the information gained at scale $k$, and for that we can call it a detail space by analogy with multiresolution analysis on $L^{2}(\mathbb{R})$. In the present case however, the decomposition of $L(\Sn)$ into detail spaces is not orthogonal, as we shall see in section \ref{sec:wavelet}.

The second step in the definition of a multiresolution analysis is the construction of operators of ``zoom'' and ``displacement''. Translations on $\mathbb{R}$ are given by the additive action of $\mathbb{R}$ on itself, \textit{i.e.} are of the form $x \mapsto x + a$, with $a\in\mathbb{R}$. The associated translations on $L^{2}(\mathbb{R})$ are then of the form $f \mapsto f(. - a)$, so that the indicator function of a singleton $\{x\}$ is sent to the indicator function of the singleton $\{x+a\}$. In the case of injective words, we consider the canonical action of $\Sn$ on $\Gamma_{n}$, defined by $\pi \mapsto\sigma_{0}(\pi)$, where for $\sigma_{0}\in\Sn$ and $\pi = \pi_{1}\dots\pi_{k}\in\Gamma_{n}$, $\sigma_{0}(\pi)$ is the injective word $\sigma_{0}(\pi_{1})\dots\sigma_{0}(\pi_{k})$. We then define the associated translations on $L(\Gamma_{n})$ as the linear operators $T_{\sigma_{0}}$ defined on Dirac functions by
\begin{equation}
T_{\sigma_{0}}\pi = \sigma_{0}(\pi),
\end{equation}
for $\sigma_{0}\in\Sn$. We could still denote the translation operator by $\sigma_{0}$ but we choose the notation $T_{\sigma_{0}}$ for clarity's sake. It is easy to see that the orbits of the action are the $\Gamma^{k}$, for $k\in\{0, \dots, n\}$. Translation operators thus stabilize each space $L(\Gamma^{k})$, and in particular $L(\Gamma^{n}) = L(\Sn)$. We still denote by $T_{\sigma_{0}}$ the induced operator. By construction, the operator (and its induced operators) $T_{\sigma_{0}}$ is invertible with inverse $T_{\sigma_{0}}^{-1} = T_{\sigma_{0}^{-1}}$, for any $\sigma_{0}\in\Sn$.

\begin{remark}
If $\pi = \pi_{1}\dots\pi_{n}\in\Gamma^{n}$ is seen as a permutation, then $\pi_{i} = \pi^{-1}(i)$ for $i\in\n$ and $\sigma_{0}(\pi)$ is the injective word associated to the permutation $\pi\sigma_{0}^{-1}$. Translation $T_{\sigma_{0}}$ on $L(\Gamma^{n}) = L(\Sn)$ can thus also be defined by $T_{\sigma_{0}}f(\sigma) = f(\sigma\sigma_{0})$. The mapping $\sigma_{0} \mapsto T_{\sigma_{0}}$ is called the right regular representation in group representation theory.
\end{remark}

\begin{lemma}
\label{properties translations}
Let $\sigma_{0}\in\Sn$, $\omega\in\Gamma_{n}$ and $a\in\n$.
\begin{enumerate}
	\item $T_{\sigma_{0}}\varrho_{a} = \varrho_{\sigma_{0}(a)}T_{\sigma_{0}}$.
	\item $T_{\sigma_{0}}\mathfrak{i}_{\omega} = \mathfrak{i}_{\sigma_{0}(\omega)}T_{\sigma_{0}} \quad\text{and}\quad T_{\sigma_{0}}\mathfrak{j}_{\omega} = \mathfrak{j}_{\sigma_{0}(\omega)}T_{\sigma_{0}}$.
	\item $T_{\sigma_{0}}\phi_{n} = \phi_{n} T_{\sigma_{0}},\quad$ \textit{i.e.} $\quad T_{\sigma_{0}}\1{\Sn[\pi]} = \1{\Sn[\sigma_{0}(\pi)]}\ \text{ for all }\ \pi\in\Gamma_{n}$.
\end{enumerate}
\end{lemma}

\begin{proof}
\textit Properties {1.} and \textit{2.} are trivially verified. To prove \textit{3.}, observe that $\omega\mapsto \sigma_{0}(\omega)$ being a group action, it is bijective, and thus using equation \eqref{characterization phi} and \textit{2.}, we obtain
\[
T_{\sigma_{0}}\phi_{n} =  \sum_{\omega_{1}, \omega_{2}\in \Gamma_{n}}T_{\sigma_{0}}\mathfrak{i}_{\omega_{1}}\mathfrak{j}_{\omega_{2}} = \sum_{\omega_{1}, \omega_{2}\in \Gamma_{n}}\mathfrak{i}_{\sigma_{0}(\omega_{1})}\mathfrak{j}_{\sigma_{0}(\omega_{2})}T_{\sigma_{0}} = \left[\sum_{\omega_{1}^{\prime}, \omega_{2}^{\prime}\in \Gamma_{n}}\mathfrak{i}_{\omega_{1}^{\prime}}\mathfrak{j}_{\omega_{2}^{\prime}}\right]T_{\sigma_{0}} = \phi_{n} T_{\sigma_{0}}.
\]
\end{proof}

The following proposition shows that translation operators $T_{\sigma_{0}}$ can be seen as ``displacement'' operators adapted to our multiresolution decomposition.

\begin{proposition}[Displacement operator]
\label{displacement operator}
Let $k\in\{2, \dots, \n\}$, $(A,B)\in\binom{\n}{k}^{2}$ and $\sigma_{0}\in\Sn$ such that $\sigma_{0}(A) = B$. Then
\[
T_{\sigma_{0}}(W_{A}) = W_{B}.
\]
\end{proposition}

\begin{proof}
Since $T_{\sigma_{0}}$ is invertible and $\dim W_{A} = \dim W_{B} = d_{k}$ by theorem \ref{item-scale decomposition}, we only need to prove that $T_{\sigma_{0}}(W_{A}) \subset W_{B}$. Let $x\in H_{A}$. Property \textit{3.} in lemma \ref{properties translations} gives $T_{\sigma_{0}}\phi_{n} x = \phi_{n} T_{\sigma_{0}}x$. Thus we just have to show that $T_{\sigma_{0}}x\in H_{B}$. Since $\sigma_{0}$ is a permutation such that $\sigma_{0}(A) = B$, it is clear that $\{ \sigma_{0}(\pi) \;\vert\; \pi\in\Rank{A} \} = \Rank{B}$, and $T_{\sigma_{0}}x\in L(\Rank{B})$. Now, using property \textit{1.} in lemma \ref{properties translations}, we have for any $b\in B$, $\varrho_{b}T_{\sigma_{0}}x = T_{\sigma_{0}}\varrho_{\sigma_{0}^{-1}(b)}x = 0$ because $\sigma_{0}^{-1}(b)\in A$.
\end{proof}

Looking at definitions \eqref{definition Wk} and \eqref{definition Vk}, proposition \ref{displacement operator} immediately gives the following result.

\begin{proposition}[Translation invariance]
\label{translation invariance}
For $k\in\{2, \dots, n\}$, the spaces $W^{k}$ and $V^{k}$ are invariant under all the translations $T_{\sigma_{0}}$, for $\sigma_{0}\in\Sn$.
\end{proposition}

Observe that the space $V^{j}$ is invariant under all translations $T_{\sigma_{0}}$ whereas in the case of the multiresolution analysis on $L^{2}(\mathbb{R})$, the space $\tilde{V}^{j}$ is only invariant under scaled translations $f \mapsto f(. - 2^{-j}k)$. The latter property means that the size of translations is limited by the resolution level. The same interpretation is actually also true in our context: though $V^{j}$ is invariant under all translations $T_{\sigma_{0}}$, they only involve the action of $\Sn$ on the sets $\Gamma^{i}$ for $i\leq j$. The ``size'' of translations on $V^{j}$ is thus inherently limited by the resolution level.

While the construction of our displacement operator is based on the same algebraic objects as for $L^{2}(\mathbb{R})$, namely translations associated to a group action, it is not possible to base the construction of a zooming operator on dilation. This is the bottleneck of any construction of a multiresolution analysis on a discrete space such as $\Sn$, as observed in \cite{Kondor2012}. Hence, there is no simple way to define an operator that allows to change scales, such as $f(x) \mapsto f(2x)$. We can however construct a family of ``dezooming'' operators $\Phi_{k}$ that each project onto the corresponding space $V^{k}$. For $k\in\{ 2, \dots, n\}$, we denote by $M_{k}$ the operator associated to all the marginals of scale $k$, \textit{i.e.} $M_{k} := M_{\binom{\n}{k}}$,
\begin{align*}
M_{k} : L(\Sn) &\rightarrow \bigoplus_{\vert A\vert = k} L(\Rank{A})\\
f &\mapsto (f_{A})_{\vert A\vert = k}.
\end{align*}
Using corollary \ref{information decomposition} for $\mathcal{A} = \binom{\n}{k}$, we have
\[
L(\Sn) = \ker M_{k} \oplus \left[V^{0}\oplus\bigoplus_{B\in\bigcup_{\vert A\vert = k}\mathcal{P}(A)}W_{B}\right] = \ker M_{k} \oplus V^{k}.
\]
Therefore, for any $F\in M_{k}(L(\Sn))$, there exists a unique $f\in V^{k}$ such that $M_{k}f = F$. We denote by $M_{k}^{+}$ the operator from $M_{k}(L(\Sn))$ to $V^{k}$ that sends $F$ to $f$. This is a pseudoinverse of $M_{k}$, but not the Moore-Penrose pseudoinverse because $V^{k}$ is not the orthogonal supplementary of $\ker M_{k}$.

\begin{definition}[Dezooming operator]
Let $\Phi_{0} : f \mapsto (\left\langle f, \1{\Sn}\right\rangle/n!) \1{\Sn}$ be the orthogonal projection on $V^{0}$ and for $k\in\{2, \dots, n\}$,
\[
\Phi_{k} = M_{k}^{+}M_{k}.
\]
\end{definition}

\subsection{Decomposition of the space $W^{k}$ into irreducible components}
\label{subsec:decomposition-space-W}

By proposition \ref{translation invariance}, the space $W^{k}$ with $k\in\{2, \dots, n\}$ is invariant under all the translations $T_{\sigma_{0}}$ for all $\sigma_{0}\in\Sn$. In other words, it is a representation of the symmetric group $\Sn$. It can thus be decomposed as a sum of irreducible representations $S^{\lambda}$. The multiplicity of each irreducible is nonetheless not obvious to compute. This is one of the major results established in \cite{Reiner13}. Its statement requires some definitions.

A Young diagram (or a Ferrer's diagram) of size $n$ is a collection of boxes of the form
\begin{center}
\includegraphics[scale=0.3]{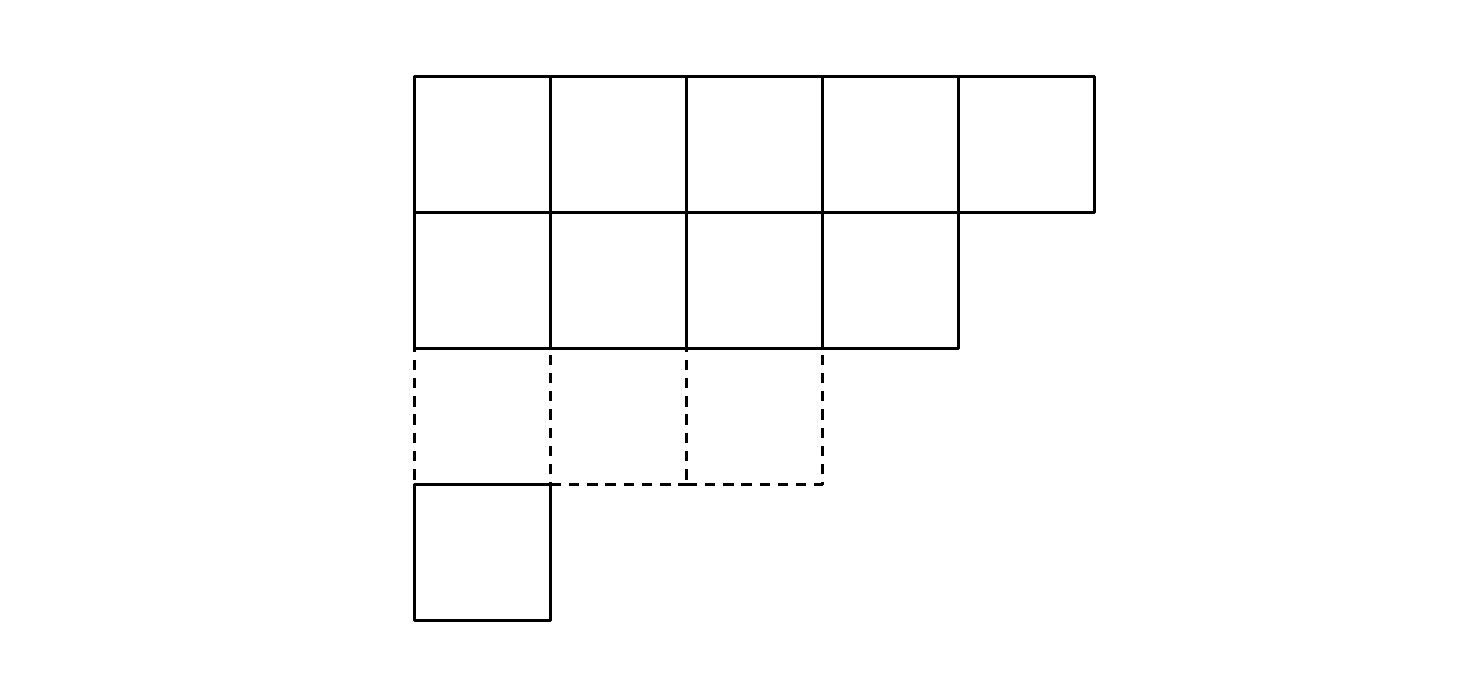}
\end{center}
where if $\lambda_{i}$ denotes the number of boxes in row $i$, then $\lambda = (\lambda_{1}, \dots, \lambda_{r})$, called the shape of the Young diagram, must be a partition of $n$. The total number of boxes of a Young diagram is therefore equal to $n$, and each row contains at most as many boxes as the row above it. A Young tableau is a Young diagram filled with all the integers $1, \dots, n$, one in each boxes. The shape of a Young tableau $Q$, denoted by $\text{shape}(Q)$, is the shape of the associated Young Diagram, it is thus a partition of $n$. There are clearly $n!$ Young tableaux of a given shape $\lambda \vdash n$. A Young tableau is said to be \textit{standard} if the numbers increase along the rows and down the columns.

\begin{example}
In the following figure, the first tableau is standard whereas the second is not.
\begin{center}
\includegraphics[scale=0.3]{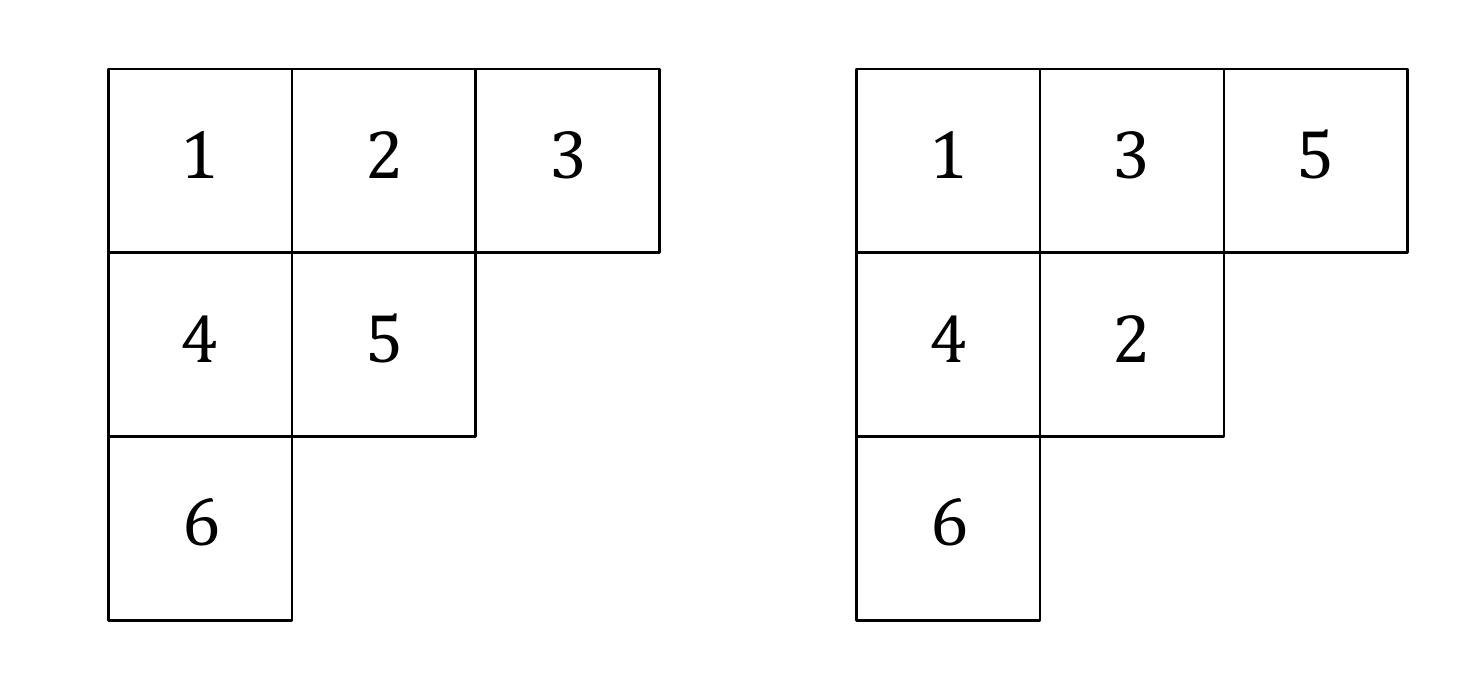}
\end{center}
\end{example}

Notice that a standard Young tableau always have $1$ in its top-left box, and that the box that contains $n$ is necessarily at the end of a row and a column. We denote by $\text{SYT}_{n}$ the set of all standard Young tableaux of size $n$ and by $\text{SYT}_{n}(\lambda) = \{Q\in \text{SYT}_{n} \;\vert\; \text{shape}(Q) = \lambda\}$ the set of standard Young tableaux of shape $\lambda$, for $\lambda\vdash n$. By construction, $\text{SYT}_{n} = \bigsqcup_{\lambda\,\vdash\, n}\text{SYT}_{n}(\lambda)$.  A classic result in the representation theory of the symmetric group is that for all $\lambda\vdash n$, the dimension $d_{\lambda}$ of $S^{\lambda}$, which is also its multiplicity in the decomposition of $L(\Sn)$, is actually equal to the number of standard Young tableaux of shape $\lambda$. Thus the decomposition of $L(\Sn)$ into irreducible representations is given by:
\[
L(\Sn) = \bigoplus_{Q\in\text{SYT}_{n}}S^{\,\text{shape}(Q)}.
\]
Figure \ref{fig:tableaux-decomposition} represents all the standard Young tableaux of size $n = 4$, gathered by shape.

\begin{figure}[h!]
\centering
\includegraphics[scale=0.6]{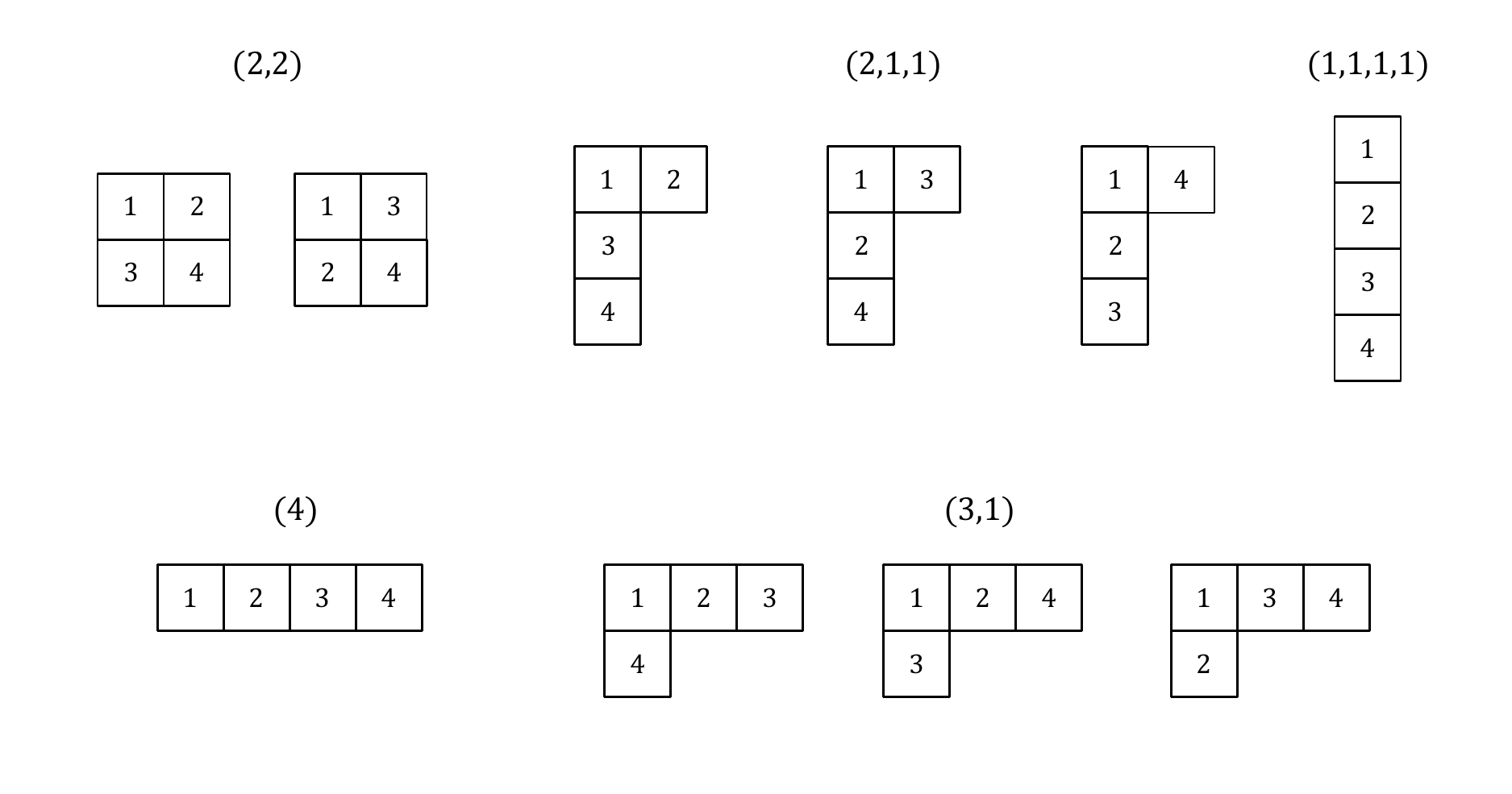}
\caption{Standard Young tableaux of size $n = 4$}
\label{fig:tableaux-decomposition}
\end{figure}

By construction, $L(\Sn) = V^{0}\oplus\bigoplus_{k=2}^{n}W^{k}$, where $V^{0}$ is isomorphic to the Specht module $S^{(n)} = S^{\,\text{shape}(Q_{0})}$, $Q_{0}$ being the unique standard Young tableau of shape $(n)$. So for each $k\in\{2, \dots, n\}$ the decomposition of the space $W^{k}$ must involve a certain subset $T_{k}$ of $\text{SYT}_{n}$, such that $\text{SYT}_{n} = \{Q_{0}\}\sqcup\bigsqcup_{k=2}^{n}T_{k}$. The construction of these subsets is done in \cite{Reiner13}. We reproduce it here. Let $Q$ be a standard Young tableau. Then it contains a unique maximal subtableau of the form
\begin{center}
\includegraphics[scale=0.5]{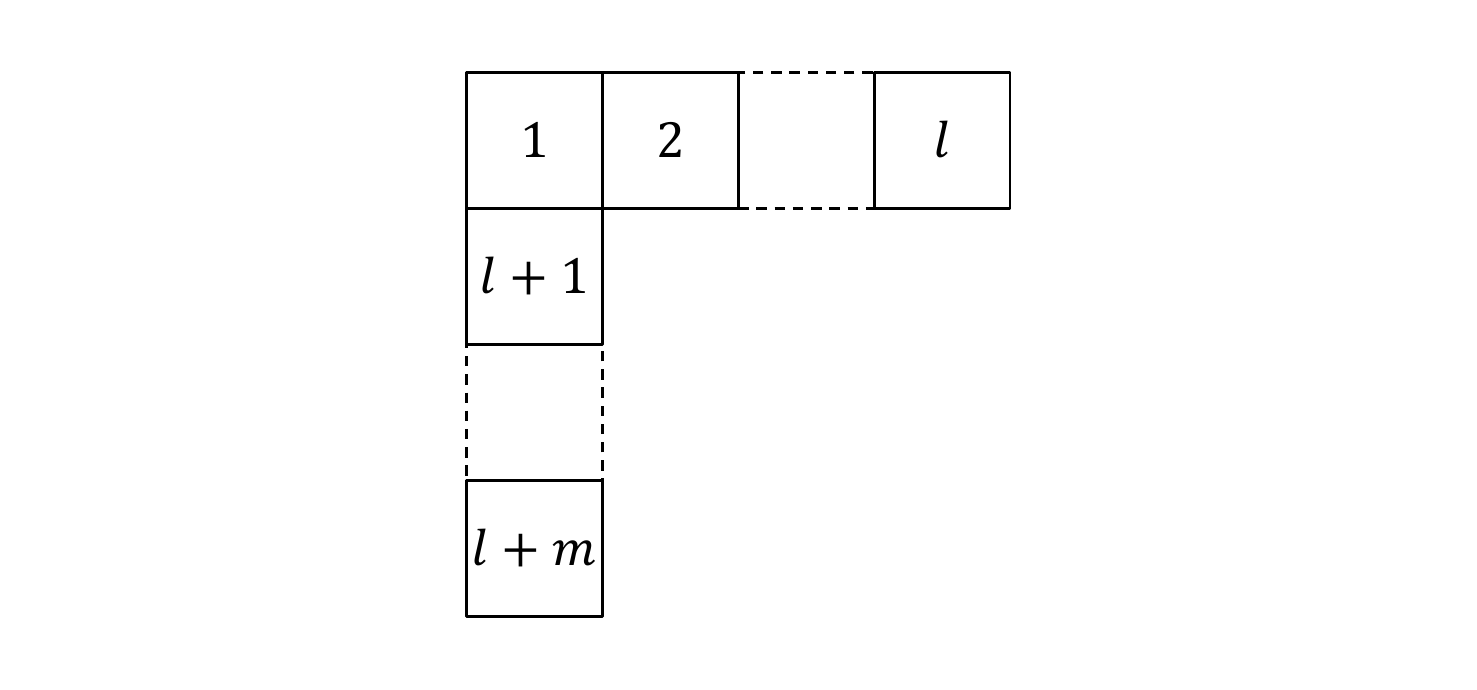}
\end{center}
with $1 \leq l \leq n$ and $0 \leq m \leq n-l$. Define
\begin{equation}
\text{eig}(Q) = \left\{
\begin{aligned}
l &\quad\text{if } m \text{ is even},\\
l-1 &\quad\text{if } m \text{ is odd}.\\
\end{aligned}\right.
\end{equation}
(This definition is given in \cite{Reiner13}, in the proof of Proposition 6.23).

\begin{theorem}[Decomposition of $W^{k}$ into irreducible representations]
\label{decomposition-W}
For $k\in\{2,\dots,n\}$, the following decomposition holds
\[
W^{k} \cong \bigoplus_{\substack{Q\in\text{SYT}_{n} \\ \text{eig}(Q) = n-k}}S^{\,\text{shape}(Q)}.
\]
\end{theorem}

\begin{proof}
For $k\in\{2,\dots,n\}$,
\[
W^{k} = \bigoplus_{\vert A\vert = k}W_{A} = \bigoplus_{\vert A\vert = k}\phi_{n}(H_{A}) = \phi_{n}\left(\bigoplus_{\vert A\vert = k}H_{A}\right) \cong \bigoplus_{\vert A\vert = k}H_{A},
\]
where the two last linear spaces are isomorphic because $\phi_{n}$ is an isomorphism between $H_{A}$ and $W_{A}$ for any $A\in\mathcal{P}(\n)$ by theorem \ref{item-scale decomposition}. Furthermore, point $3.$ of lemma \ref{properties translations} shows that $W^{k}$ and $\bigoplus_{\vert A\vert = k}H_{A}$ are isomorphic as representations of $\Sn$. In the notations of \cite{Reiner13}, $H_{A} = \ker \pi_{A}$, so by their theorem 6.20, $W^{k}\cong F_{n, n-k}$ as representations of $\Sn$. Theorem \ref{decomposition-W} is then just a reformulation in the present setting of theorem 6.26 from \cite{Reiner13}.
\end{proof}

For $k\in\{2,\dots,n\}$, the subset $T_{k}$ of standard Young tableaux involved in the decomposition of $W^{k}$ is thus defined by $T_{k} = \{Q\in\text{SYT}_{n} \;\vert\; \text{eig}(Q) = n-k \}$. Figure \ref{fig:decomposition-W} represents the different subsets $T_{k}$ with the associated decompositions for $n = 4$.

\begin{figure}[h!]
\centering
\begin{tabular}{cc}
\hline
\raisebox{1cm}{$W^{4} \cong S^{(3,1)}\oplus S^{(2,2)}\oplus S^{(2,1,1)}\oplus S^{(1,1,1,1)}$} & \includegraphics[scale=0.4,trim= 0 -1 0 -1]{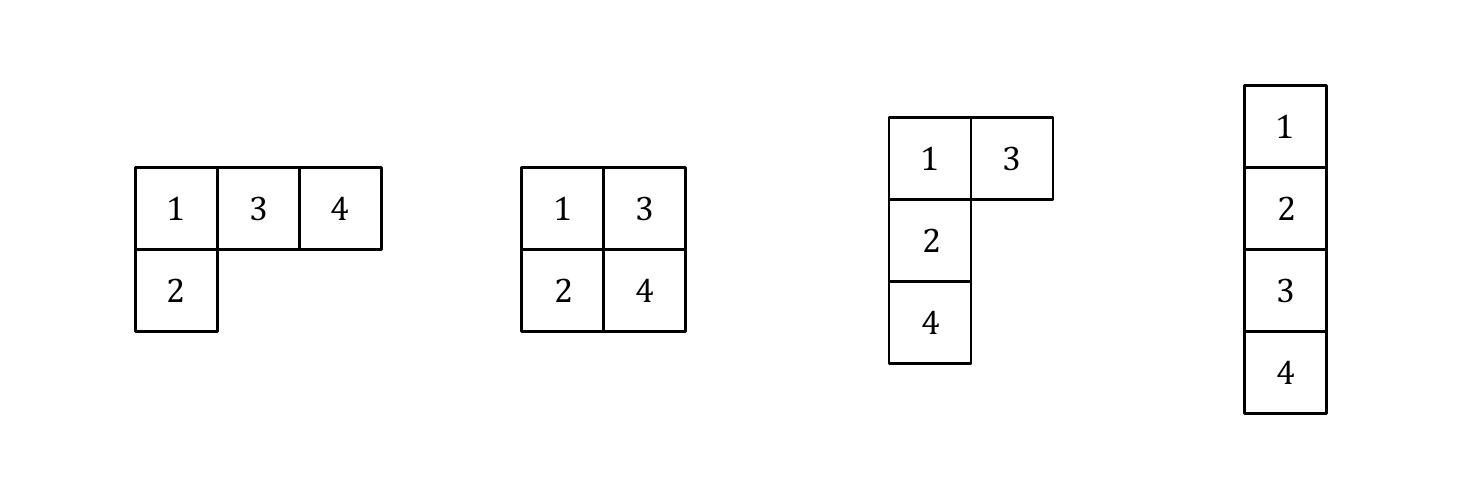}\\
\hline
\raisebox{1cm}{$W^{3}\cong S^{(3,1)}\oplus S^{(2,2)}\oplus S^{(2,1,1)}$} & \includegraphics[scale=0.4,trim= 0 -1 0 -1]{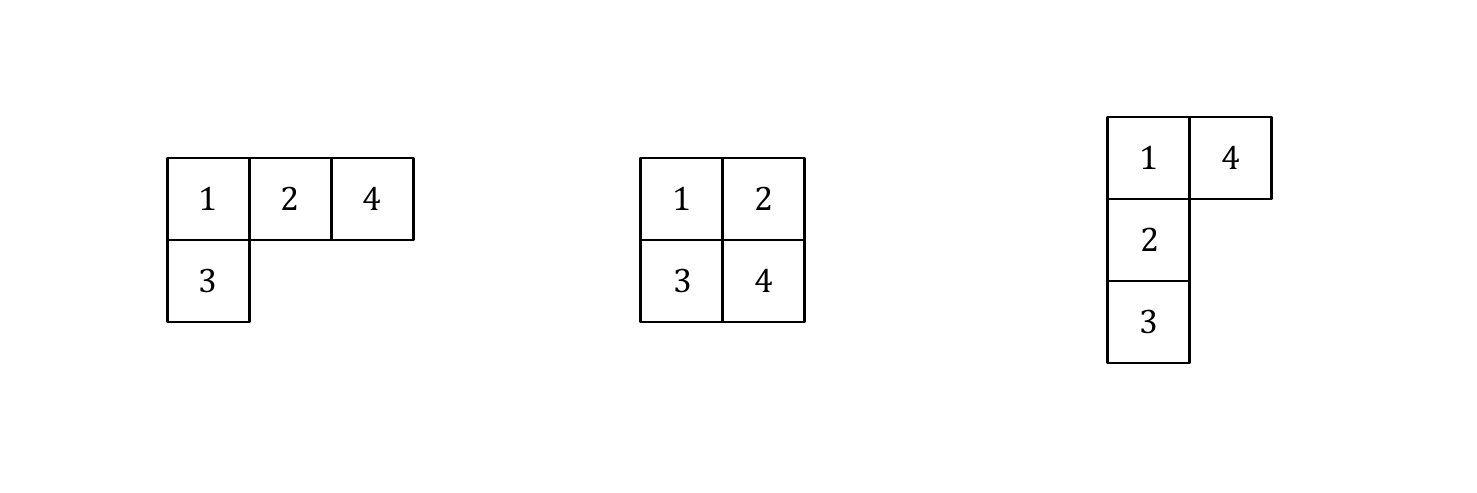}\\
\hline
\raisebox{1cm}{$W^{2}\cong S^{(3,1)}\oplus S^{(2,1,1)}$} & \includegraphics[scale=0.4,trim= 0 -1 0 -1]{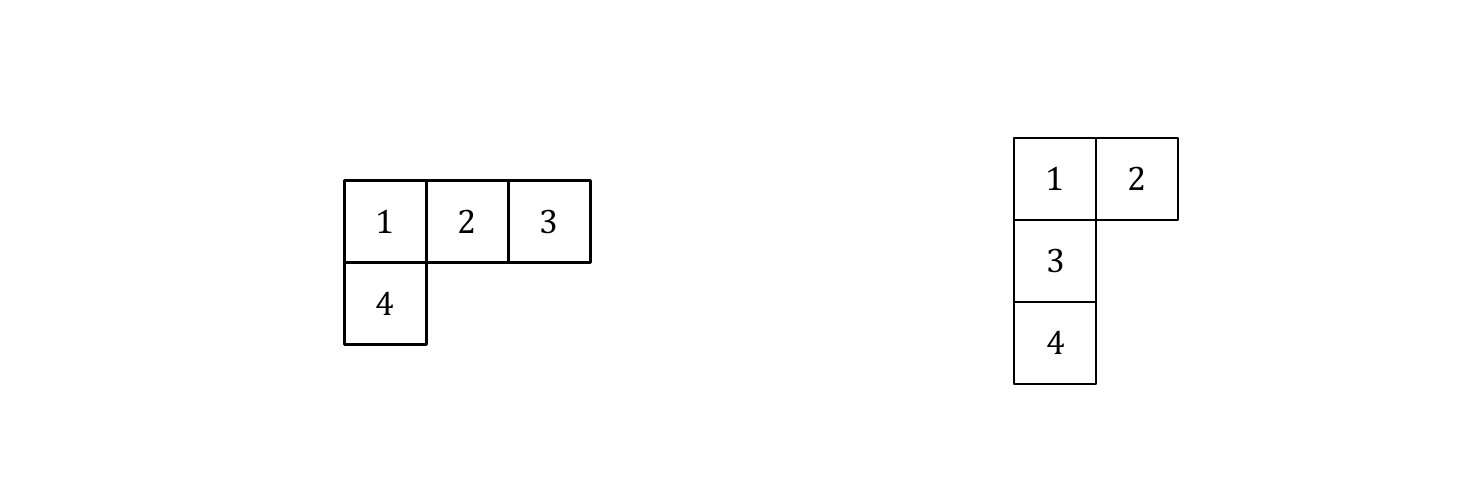}\\
\hline
\raisebox{1cm}{$V^{0}\cong S^{(4)}$} &\includegraphics[scale=0.4,trim= 0 -1 0 -1]{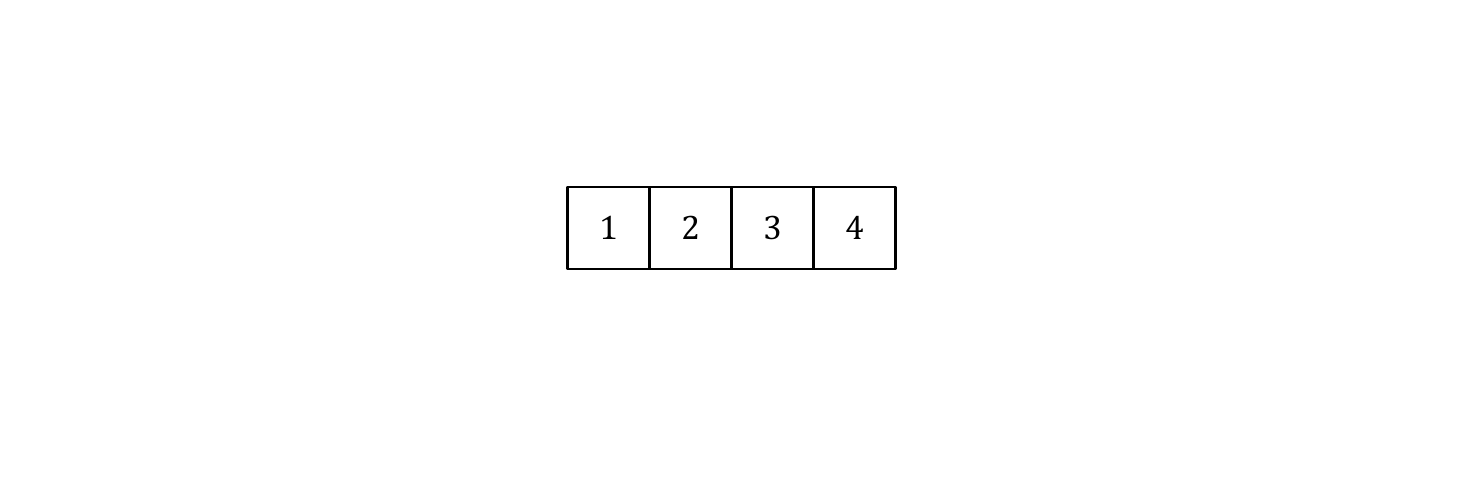}\\
\hline
\end{tabular}
\caption{Spaces $W^{k}$ and their decompositions into irreducibles, for $n = 4$}
\label{fig:decomposition-W}
\end{figure}

\begin{remark}
For $\lambda = (\lambda_{1}, \dots, \lambda_{r})\vdash n$, the usual ranking interpretation of the Specht module $S^{\lambda}$ is that it localizes information at ``scale'' $n-\lambda_{1}$, in the sense that it localizes the absolute rank information of $n-\lambda_{1}$ items. The Specht module $S^{(n-1,1)}$ localizes absolute rank information about $1$ item, $S^{(n-2,2)}$ and $S^{(n-2,1,1)}$ both localize absolute rank information about $2$ items, and so on. It is interesting to notice that this interpretation does not hold when dealing with relative rank information. The space $W^{k}$ can indeed be seen as localizing the relative rank information at scale $k$ \textit{i.e.} the relative rank information related to incomplete rankings involving $k$ items. However, theorem \ref{decomposition-W} shows that the decomposition of $W^{k}$ can involve Specht modules $S^{\lambda}$ with $n-\lambda_{1} \neq k$. Figure \ref{fig:decomposition-W} shows for example that for $n = 4$, $W^{3}$ involves absolute rank information of scale $1$ and $2$.
\end{remark}

\section{The wavelet basis}
\label{sec:wavelet}

We now construct an explicit basis $\Psi$ adapted to the multiresolution decomposition of $L(\Sn)$, in the sense that $\Psi = \{\psi_{0}\}\cup\bigcup_{A\in\mathcal{P}(\n)}\Psi_{A}$ where $\Psi_{A}$ is a basis of $W_{A}$ for all $A\in\mathcal{P}(\n)$, and establish its main properties.

\subsection{Generative algorithm}
\label{subsec:generative-algorithm}

The basis is defined by an algorithm adapted from \cite{RT11}, which requires some definitions about cycles and permutations. A cycle on $\n$ is a permutation $\gamma\in\Sn$ for which there exist $m$ distinct elements $a_{1}, \dots, a_{m}\in\n$, with $m\geq 2$, such that $\gamma(a_{i}) = a_{i+1}$ for $i = 1, \dots, m-1$, $\gamma(a_{m}) = a_{1}$, and $\gamma(a^{\prime}) = a^{\prime}$ for all $a^{\prime}\in \n\setminus\lbrace a_{1}, \dots, a_{m}\rbrace$. The cycle $\gamma$ is then denoted by $(a_{1}\dots a_{m})$, its support is the set $\lbrace a_{1}, \dots, a_{m}\rbrace$ and its length is $l(\gamma) = m$. For $A\in\mathcal{P}(\n)$, we denote by $\Cycle(A)$ the set of all cycles with support $A$. It is well known that a permutation $\tau\in\Sn$ admits a unique decomposition as a product of cycles with distinct supports $\tau = \gamma_{1}\dots\gamma_{r}$ (fixed-points are not represented). This decomposition can though be written in several ways, depending on the order of the cycles and the first element of each cycle.

\begin{definition}[Standard cycle form]
A permutation is written in standard cycle form if it is written as a product of disjoint cycles so that the minimum element of a cycle appears at the leftmost letter in that cycle, and the cycles are arranged from left to right in increasing values of minimum letters.
\end{definition}

\begin{example}
The permutation $(134)(25)$ is written in standard cycle form, while the alternative representations $(413)(25)$ or $(25)(134)$ are not.
\end{example}

For a permutation $\tau\in\Sn$, we denote by $\cyc(\tau)$ the number of its cycles, define its support by $\supp(\tau) = \{ i\in\n \;\vert\; \tau(i) \neq i\}$ and its length by $l(\tau) = \vert\supp(\tau)\vert$. These definitions extend the definitions of the support and the length for a cycle, and if $\gamma_{1}\dots\gamma_{\cyc(\tau)}$ is the cycle decomposition of $\tau$, $l(\tau) = l_{1} + \dots + l_{\cyc(\tau)}$. For $A\in\mathcal{P}(\n)$, we define
\[
\mathcal{D}_{A} = \{\sigma\in\Sn \;\vert\; \supp(\sigma) = A \},
\]
and we set by convention $\mathcal{D}_{\emptyset} = \{id\}$, where $id\in\Sn$ is the identity permutation on $\n$. By definition, a permutation $\sigma\in\mathcal{D}_{A}$ induces a fixed-point free permutation, also called a derangement, on $A$. The set $\mathcal{D}_{A}$ is thus the natural embedding of the set of derangements on $A$ in $\Sn$. The algorithm of \cite{RT11} computes a basis for the top homology space of the complex of injective words over the field $\mathbb{F}_{2} = \mathbb{Z}/2\mathbb{Z}$ of two elements. It uses the operation on $\mathbb{F}_{2}$-valued chains ``$x\diamond y = xy + yx$''. In the present setting, we use the following definition.

\begin{definition}[Diamond operator]
\label{operator diamond}
For $x, y\in L(\Gamma_{n})$, we define
\begin{equation*}
x\diamond y = xy - yx.
\end{equation*}
\end{definition}

The algorithm of \cite{RT11} takes a derangement on $\{1, \dots, k\}$ as input and outputs an element of the top homology space of the complex of injective words. It happens that the same algorithm with the diamond operator of definition \ref{operator diamond} maps a derangement on $\{1, \dots, k\}$ to an element of $H_{k}$. Moreover, the algorithm is naturally extended to take a permutation $\tau\in\mathcal{D}_{A}$ as input and output an element $x_{\tau}$ of the space $H_{A}$, for any $A\in\mathcal{P}(\n)$.

\begin{algorithm}
\label{basis generation}
Let $A\in\mathcal{P}(\n)$. The input is a permutation $\tau\in\mathcal{D}_{A}$ written in standard cycle form, and the output is a chain $x_{\tau}\in H_{A}$.
\begin{center}
\begin{tabular}{p{35pt}p{280pt}}
Step 1. & Between each consecutive pair of letters in each cycle of $\tau$, insert the symbol $\star$.\\
Step 2. & If there are no $\star$ symbols in the string, then HALT. Otherwise, determine which symbol $\star$ has the largest right-hand neighbor.\\
Step 2. & Suppose that the symbol located in Step 2 is between quantities $Q$ and $R$; that is, it appears as $Q\star R$. Then replace $Q\star R$ by $(Q\diamond R)$.\\
Step 4. & GOTO Step 2.
\end{tabular}
\end{center}
\end{algorithm}

\begin{example}
Let $A = \{1,2,3,4,5\}$ and $\tau = (134)(25)$. Algorithm \ref{basis generation} gives the following sequence of steps.
\begin{center}
$(1\star 3\star 4)(2\star 5)$\\
$(1\star 3\star 4)(2\diamond 5)$\\
$(1\star (3\diamond 4))(2\diamond 5)$\\
$(1\diamond (3\diamond 4))(2\diamond 5)$\\
\end{center}
Expanding the concatenation and the $\diamond$ operations, we obtain:
\begin{align*}
x_{(134)(25)} 	&= (1\diamond (3\diamond 4))(2\diamond 5)\\
				&= (1\diamond (34-43))(25-52)\\
				&= (134-143-341+431)(25-52)\\
				&= 13425-13452-14325+14352-34125+34152+43125-43152.
\end{align*}
\end{example}

\subsection{The basis of $L(\Sn)$}
\label{subsec:global-basis}

We now construct the wavelet basis of $L(\Sn)$. We first show that the outputs of algorithm \ref{basis generation} belong to the claimed space.

\begin{proposition}
\label{elements of H}
Let $A\in\mathcal{P}(\n)$. For all $\tau\in\mathcal{D}_{A}$, $x_{\tau}\in H_{A}$.
\end{proposition}

As in \cite{RT11}, the proof relies on the simple following lemma, of which proof is straightforward and is thus omitted.

\begin{lemma}
\label{deletion}
Let $x,y\in L(\Gamma_{n})$ with $c(x)\cap c(y) = \emptyset$, and $a\in c(x)$. Then
\[
\varrho_{a}(xy) = \varrho_{a}(x)y \qquad\text{and}\qquad \varrho_{a}(x\diamond y) = \varrho_{a}(x)\diamond y.
\]
\end{lemma}

\begin{proof}[Proof of proposition \ref{elements of H}]
Let $A\in\mathcal{P}(\n)$ and $\tau\in\mathcal{D}_{A}$. We need to show that for all $a\in A$, $\varrho_{a}x_{\tau} = 0$. Let $a\in A$ and $\tau = \gamma_{1}\dots\gamma_{r}$ be the standard cycle form of $\tau$. By definition of $\mathcal{D}_{A}$, $\{\supp(\gamma_{1}), \dots, \supp(\gamma_{r}) \}$ is a partition of $A$. Let $\gamma_{i}$ be the cycle which support contains $a$. By definition of the algorithm, $x_{\tau} = x_{\gamma_{1}}\dots x_{\gamma_{r}}$, and using lemma $\ref{deletion}$, we have $\varrho_{a}(x_{\tau}) = x_{\gamma_{1}}\dots\varrho_{a}(x_{\gamma_{i}})\dots x_{\gamma_{r}}$. Since $\gamma_{i}$ is a cycle, its support contains at least two elements, and thus $x_{\gamma_{i}}$ contains a product $a\diamond u$ or $u \diamond a$. Now, $\varrho_{a}(a\diamond u) = \varrho_{a}(au - ua) = u - u = 0$. Using lemma \ref{deletion}, this implies that $\varrho_{a}(x_{\gamma_{i}}) = 0$ and then that $\varrho_{a}x_{\tau} = 0$, which concludes the proof.
\end{proof}

\begin{example}
Using the precedent example, we can see that
\[
\varrho_{4}x_{(134)(25)} = 1325-1352-1325+1352-3125+3152+3125-3152 = 0.
\]
\end{example}

\begin{remark}
The proof of proposition \ref{elements of H} does not use the fact that the cycle decomposition is in standard form. This condition is indeed only necessary to prove that the outputs of the algorithm for all $\tau\in\mathcal{D}_{A}$ constitute a basis of $H_{A}$.
\end{remark}

We now get to the central result in the construction of our wavelet basis: $(x_{\tau})_{\tau\in\mathcal{D}_{A}}$ is a basis of $H_{A}$ for all $A\in\mathcal{P}(\n)$. In \cite{RT11}, they prove that their algorithm generates a basis for the top homology space of the complex of injective words. This result cannot be directly transposed in our context because $H_{A}$ is not the top homology space of the complex of injective words on $A$. It happens however that the proof is exactly the same as the proof of theorem 5.2 in \cite{RT11} and relies on concepts introduced specifically for that purpose (namely ``graph derangements'' and the ``collapsing map''). It is thus left to the reader.

\begin{theorem}
\label{basis H}
For all $A\in\mathcal{P}(\n)$, $(x_{\tau})_{\mathcal{D}_{A}}$ is a basis of $H_{A}$.
\end{theorem}

We are now able to construct the wavelet basis of $L(\Sn)$, using the embedding operator $\phi_{n}$. Notice that for all $\tau\in\Sn\setminus\{id\}$, $\supp(\tau)\in\mathcal{P}(\n)$. We thus set $\psi_{id} = \psi_{0} = \1{\Sn}$ and for $\tau\in\Sn\setminus\{id\}$, we define
\begin{equation}
\label{definition psi}
\psi_{\tau} = \phi_{n}(x_{\tau}) = \sum_{\pi\in\supp(\tau)}x_{\tau}(\pi)\1{\Sn[\pi]}.
\end{equation}
By theorem \ref{item-scale decomposition}, $\phi_{n}$ is an isomorphism between $H_{A}$ and $W_{A}$ for all $A\in\mathcal{P}(\n)$. Combined with theorem \ref{basis H} we immediately obtain the following theorem.

\begin{theorem}
\label{wavelet basis general}
For all $A\in\mathcal{P}(\n)$, $(\psi_{\tau})_{\mathcal{D}_{A}}$ is a basis of $W_{A}$, and
\[
(\psi_{\tau})_{\tau\in\Sn} \text{ is a basis of } L(\Sn).
\]
\end{theorem}

\begin{example}
\label{ex:full-basis}
For $n = 4$, figure \ref{fig:full-wavelet-basis} gives the full wavelet basis of $L(\Sym{4})$ ($[\pi]$ is a shortcut for $\1{\Sym{4}[\pi]}$).

\begin{figure}[h!]
\centering
\begin{tabular*}{\textwidth}{@{\extracolsep{\fill}}cccccccccccc}
\multicolumn{12}{c}{$\psi_{(1234)}$}\\
\multicolumn{12}{c}{\tiny{$1234 - 1243 - 1342 + 1432 - 2341 + 2431 + 3421 - 4321$}}\\
\multicolumn{12}{c}{$\psi_{(1243)}$} \\
\multicolumn{12}{c}{\tiny{$1243 - 1324 + 1342 + 1423 - 2431 + 3241 - 3421 + 4231$}}\\
\multicolumn{12}{c}{$\psi_{(1324)}$} \\
\multicolumn{12}{c}{\tiny{$1324 - 1342 - 2413 + 2431 - 3124 + 3142 + 4213 - 4231$}}\\
\multicolumn{12}{c}{$\psi_{(1342)}$} \\
\multicolumn{12}{c}{\tiny{$1342 - 1432 - 2134 + 2143 + 2341 - 2431 + 3412 - 4312$}}\\
\multicolumn{12}{c}{$\psi_{(1423)}$} \\
\multicolumn{12}{c}{\tiny{$1423 - 1432 - 2314 + 2341 + 3214 - 3241 - 4123 + 4132$}}\\
\multicolumn{12}{c}{$\psi_{(1432)}$} \\
\multicolumn{12}{c}{\tiny{$1432 - 2143 + 2314 - 2341 + 2413 - 3142 + 3412 - 4132$}}\\
\multicolumn{12}{c}{$\psi_{(12)(34)}$} \\
\multicolumn{12}{c}{\tiny{$1234 - 1243 - 2134 + 2143$}}\\
\multicolumn{12}{c}{$\psi_{(13)(24)}$} \\
\multicolumn{12}{c}{\tiny{$1324 - 1342 - 3124 + 3142$}}\\
\multicolumn{12}{c}{$\psi_{(14)(23)}$} \\
\multicolumn{12}{c}{\tiny{$1423 - 1432 - 4123 + 4132$}}\\
\vspace{10pt}\\
\multicolumn{3}{c}{$\psi_{(123)}$} & \multicolumn{3}{c}{$\psi_{(124)}$} & \multicolumn{3}{c}{$\psi_{(134)}$} & \multicolumn{3}{c}{$\psi_{(234)}$}\\
\multicolumn{3}{c}{\tiny{$[123]  - [132] - [231] + [321]$}} & \multicolumn{3}{c}{\tiny{$[124]  - [142] - [241] + [421]$}} & \multicolumn{3}{c}{\tiny{$[134]  - [143] - [341] + [431]$}} & \multicolumn{3}{c}{\tiny{$[234]  - [243] - [342] + [432]$}}\\
\multicolumn{3}{c}{$\psi_{(132)}$} & \multicolumn{3}{c}{$\psi_{(142)}$} & \multicolumn{3}{c}{$\psi_{(143)}$} & \multicolumn{3}{c}{$\psi_{(243)}$}\\
\multicolumn{3}{c}{\tiny{$[132]  - [213] + [231] - [312]$}} & \multicolumn{3}{c}{\tiny{$[142]  - [214] + [241] - [412]$}} & \multicolumn{3}{c}{\tiny{$[143]  - [314] + [341] - [413]$}} & \multicolumn{3}{c}{\tiny{$[243]  - [324] + [342] - [423]$}}\\
\vspace{10pt}\\
\multicolumn{2}{c}{$\psi_{(12)}$} & \multicolumn{2}{r}{$\psi_{(13)}$} & \multicolumn{2}{c}{$\psi_{(14)}$} & \multicolumn{2}{c}{$\psi_{(23)}$} & \multicolumn{2}{r}{$\psi_{(24)}$} & \multicolumn{2}{c}{$\psi_{(34)}$} \\
\multicolumn{2}{c}{\tiny{$[12] - [21]$}} & \multicolumn{2}{r}{\tiny{$[13] - [31]$}} & \multicolumn{2}{c}{\tiny{$[14] - [41]$}} & \multicolumn{2}{c}{\tiny{$[23] - [32]$}} & \multicolumn{2}{r}{\tiny{$[24] - [42]$}} & \multicolumn{2}{c}{\tiny{$[34] - [43]$}}\\
\vspace{10pt}\\
\multicolumn{12}{c}{$\psi_{id}$}
\end{tabular*}
\caption{Wavelet basis of $L(\Sym{4})$}
\label{fig:full-wavelet-basis}
\end{figure}

\end{example}

\begin{remark}
The wavelet basis and the multiresolution decomposition are not orthogonal, example \ref{ex:full-basis} provides many couples $\tau,\tau^{\prime}\in \Sym{4}$ such that $\left\langle \psi_{\tau}, \psi_{\tau^{\prime}}\right\rangle \neq 0$.
\end{remark}

\subsection{General properties of the wavelet basis}
\label{subsec:general-properties}

For a chain $x\in L(\Gamma_{n})$ (in particular a function in $L(\Sn)$), we define its support by $\supp(x) = \{ \pi\in\Gamma_{n} \;\vert\; x(\pi) \neq 0 \}$. See the Appendix section for the proof of the following proposition.

\begin{proposition}
\label{general wavelets}
Let $\tau\in\Sn\setminus\{id\}$, $k = \vert\tau\vert$ and $r = \cyc(\tau)$.
\begin{enumerate}
\item $\psi_{\tau}(\sigma)\in\{-1, 0, 1\}$ for all $\sigma\in\Sn$.
\item $\vert\supp(\psi_{\tau})\vert = 2^{k-r}(n-k+1)!$.
\end{enumerate}
\end{proposition}

This first proposition provides some general intuition about the wavelet basis. In particular, property \textit{1.} is interesting because it means that all the properties of a wavelet function simply depend on the sign of its values and on the combinatorial structure of its support. The following proposition shows the interaction between wavelets and translations. It appears clearly in the representation of the full wavelet basis of $L(\Sym{4})$ in example \ref{ex:full-basis} that at scale $k$, wavelet functions in $W_{A}$ with $A\in\binom{\n}{k}$ are the translated of wavelet functions in $W_{\{1, \dots, k\}}$. And indeed, as $(\psi_{\tau})_{\tau\in\mathcal{D}_{\{1, \dots, k\}}}$ is a basis of $W_{\{1, \dots, k\}}$ (by theorem \ref{wavelet basis general}), $(T_{\sigma_{0}}\psi_{\tau})_{\tau\in\mathcal{D}_{\{1, \dots, k\}}}$ is a basis of $W_{A}$ for any $\sigma_{0}\in\Sn$ such that $\sigma_{0}(\{1, \dots, k\}) = A$, by proposition \ref{displacement operator}. The following proposition refines this result.

\begin{proposition}
\label{translation wavelets}
Let $\tau\in \Sn$ and $\sigma_{0}\in\Sn$ a permutation that preserves the order of the elements of $\supp(\tau)$, \textit{i.e.} if $\supp(\tau) = \{a_{1}, \dots, a_{k}\}$ with $a_{1} < \dots < a_{k}$, then $\sigma_{0}(a_{1}) < \dots < \sigma_{0}(a_{k})$. Then we have
\[
T_{\sigma_{0}}\psi_{\tau} = \psi_{\sigma_{0}\tau\sigma_{0}^{-1}}.
\]
\end{proposition}

\begin{proof}
If $\tau = id$, $\psi_{id}$ is invariant under translations and the equality is trivially verified. We assume $\tau\neq id$, thus $\psi_{\tau} = \phi_{n}\, x_{\tau}$. By lemma \ref{properties translations}, $T_{\sigma_{0}}\phi_{n}\, x_{\tau} = \phi_{n}\, T_{\sigma_{0}}x_{\tau}$. Let $\gamma_{1}\dots\gamma_{r}$ be the standard cycle form of $\tau$ with $\gamma_{i} = (a_{i, 1}\dots a_{i,k_{i}})$. Then it is easy to see that $T_{\sigma_{0}}x_{\tau}$ is the output of algorithm \ref{basis generation} when taking as input the permutation with cycle form $\gamma_{1}^{\prime}\dots\gamma_{r}^{\prime}$ where $\gamma_{i}^{\prime} = (\sigma_{0}(a_{i, 1})\dots\sigma_{0}(a_{i,k_{i}}))$. The order-preserving condition on $\sigma_{0}$ assures that this is a standard cycle form. The proof is concluded by a classic result (or a simple verification) that says that this is the cycle form of the permutation $\sigma_{0}\tau\sigma_{0}^{-1}$.
\end{proof}

The third general property concerns the marginals of the wavelet functions. It actually only relies on the embedding operator $\phi_{n}$, and not on algorithm \ref{basis generation}. For $A\in\mathcal{P}(\n)$, we define the embedding operator $\phi_{A} : \bigoplus_{B\subset A}L(\Gamma(B)) \rightarrow L(\Rank{A})$ on the Dirac functions by
\begin{equation}
\phi_{A} : \pi \mapsto \1{\Rank{A}[\pi]} = \sum_{\sigma\in\Rank{A}[\pi]}\sigma.
\end{equation}

\begin{proposition}
\label{marginal-embedding}
Let $A\in\mathcal{P}(\n)$ and $\pi\in\Gamma_{n}$ such that $c(\pi)\subset A$. Then
\[
M_{A}\phi_{n}\pi = \frac{(n-\vert\pi\vert+1)!}{(\vert A\vert-\vert\pi\vert+1)!}\ \phi_{A}\pi.
\]
\end{proposition}

Proposition \ref{marginal-embedding} is a direct consequence of the following lemma, of which proof is given in the Appendix section.

\begin{lemma}
\label{cardinality-intersection}
Let $(\pi,\pi^{\prime})\in(\Gamma_{n})^{2}$, and $\pi^{0}$ be the subword of $\pi$ with content $c(\pi)\cap c(\pi^{\prime})$. Denoting by $\vert\pi\vert = k$, $\vert\pi^{\prime}\vert = l$ and $\vert c(\pi)\cap c(\pi^{\prime})\vert = m$, we have
\begin{equation*}
\vert\Sn[\pi]\cap\Sn(\pi^{\prime})\vert =
\left\{ \begin{aligned}
 \frac{(n-k+1)!}{(l-m+1)!} &\qquad\text{if $\pi^{0}$ is a contiguous subword of $\pi^{\prime}$,}\\
 0\qquad &\qquad\text{otherwise.}
 \end{aligned}\right.
\end{equation*}
\end{lemma}

The combination of proposition \ref{information localization} and \ref{marginal-embedding} give an explicit formula for the marginals of any elements of a space $W_{B}$, in particular for the marginals of wavelet functions.

\begin{proposition}[Marginals of the wavelet functions]
\label{marginals wavelets}
Let $A\in\mathcal{P}(\n)$. $M_{A}\psi_{id}$ is the constant function on $\Rank{A}$ equal to $n!/\vert A\vert!$, and for $\tau\in\Sn\setminus\{id\}$,
\begin{equation*}
M_{A}\psi_{\tau} = \left\{
\begin{aligned}
\frac{(n-\vert\tau\vert+1)!}{(\vert A\vert - \vert\tau\vert + 1)!}\ \phi_{A}(x_{\tau}) &\qquad\qquad\text{if } \supp(\tau)\subset A,\\
0 \qquad\qquad\qquad&\qquad\qquad\text{otherwise.}
\end{aligned}\right.
\end{equation*}
\end{proposition}

This last proposition provides the explicit wavelet basis for the space $\mathbb{M}_{\mathcal{A}}$ for any observation design $\mathcal{A}\subset\mathcal{P}(\n)$.

\begin{example}
We come back to the same example as in subsection \ref{subsec:injective-words}: $n=4$ and $\mathcal{A} = \{ \{1,3\}, \{2,4\},\\ \{3,4\}, \{1,2,3\}, \{1,3,4\} \}$. The space $\mathbb{M}_{\mathcal{A}}$ has dimension $11$ and its basis is represented by figure \ref{fig:wavelet-basis} (only the marginals on subsets $A\in\mathcal{A}$ are represented, all with the same scale).

\begin{figure}[h!]
\centering

\begin{tabular}{@{}c@{}c@{}c@{}}
$(1/2)\psi_{id}$ & $\psi_{(12)}$ & $\psi_{(13)}$ \\
\includegraphics[scale=0.9]{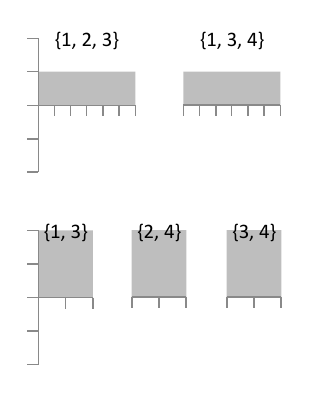} & \includegraphics[scale=0.9]{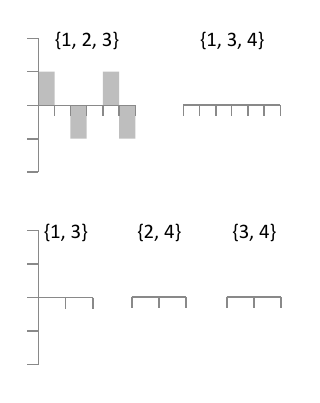} & \includegraphics[scale=0.9]{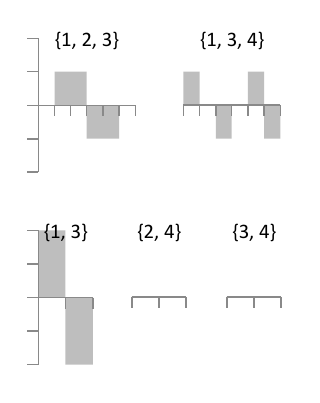}
\end{tabular}

\begin{tabular}{@{}c@{}c@{}c@{}c@{}}
$\psi_{(14)}$ & $\psi_{(23)}$ & $\psi_{(24)}$ & $\psi_{(34)}$ \\
\includegraphics[scale=0.9]{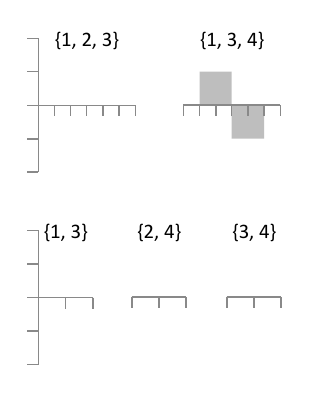} & \includegraphics[scale=0.9]{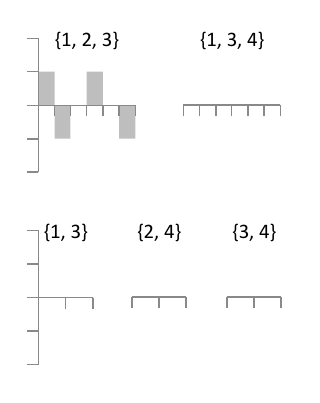} & \includegraphics[scale=0.9]{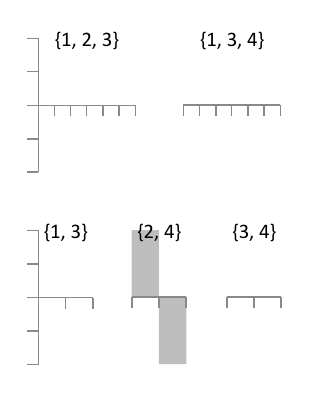} & \includegraphics[scale=0.9]{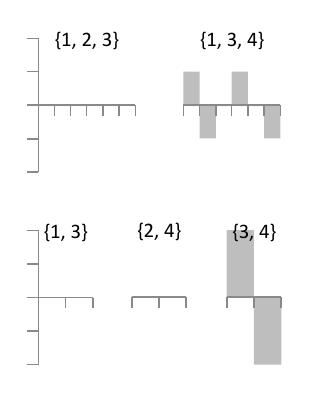}
\end{tabular}

\begin{tabular}{@{}c@{}c@{}c@{}c@{}}
 $\psi_{(123)}$ & $\psi_{(132)}$ & $\psi_{(134)}$ & $\psi_{(143)}$ \\
\includegraphics[scale=0.9]{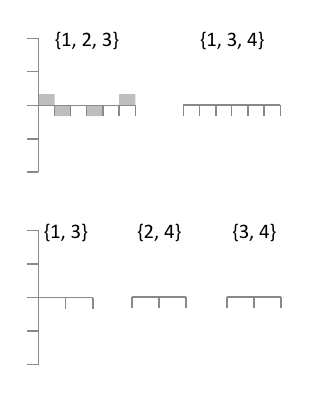} & \includegraphics[scale=0.9]{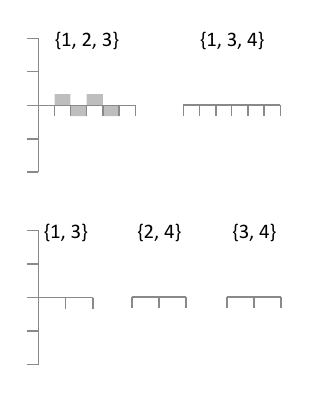} & \includegraphics[scale=0.9]{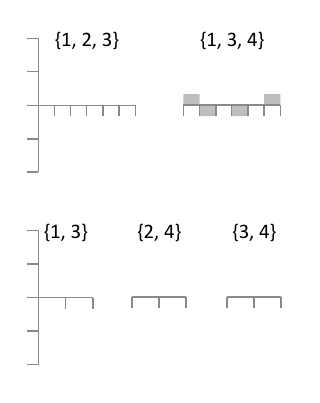} & \includegraphics[scale=0.9]{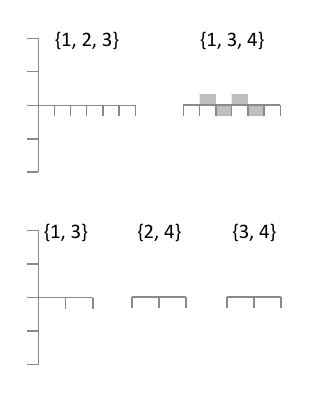}
\end{tabular}

\caption{Wavelet basis of $\mathbb{M}_{\mathcal{A}}$, for $\mathcal{A} = \{ \{1,3\}, \{2,4\}, \{3,4\}, \{1,2,3\}, \{1,3,4\} \}$}
\label{fig:wavelet-basis}
\end{figure}

\end{example}

\subsection{Structure of the wavelet basis}
\label{subsec:structure-wavelet-basis}

The properties of a multiresolution analysis $(\tilde{V}^{j})_{j\in\mathbb{Z}}$ on $L^{2}(\mathbb{R})$ directly lead to the definition of an adapted wavelet basis  $(\tilde{\psi}_{j,n})_{(j,n)\in\mathbb{Z}^{2}}$: take $\tilde{\psi}\in \tilde{V}^{0}$ and define $\tilde{\psi}_{j,n}(x) = 2^{j/2}\tilde{\psi}(2^{j}x-n)$. Then $(\tilde{\psi}_{j,n})_{(j,n)\in\mathbb{Z}^{2}}$ is a basis of $L^{2}(\mathbb{R})$ adapted to $(\tilde{V}^{j})_{j\in\mathbb{Z}}$ (see \cite{Mallat1989}) and has a simple interpretation, $\tilde{\psi}$ is the ``mother'' wavelet and all the wavelet functions are obtained from $\tilde{\psi}$ by dilation and translation. More specifically, at scale $j$, the wavelet function $\tilde{\psi}_{j,0}$ is obtained by dilation of $\tilde{\psi}_{j-1,0}$, $\tilde{\psi}_{j,0}(x) = \sqrt{2}\tilde{\psi}_{j-1,0}(2x)$, and all the $\tilde{\psi}_{j,n}$'s by translation of $\tilde{\psi}_{j,0}$, $\tilde{\psi}_{j,n}(x) = \tilde{\psi}_{j,0}(x-2^{-j}n)$. These relations encode the structure of the basis and are at the core of many applications.

In the present setting, while the translation operators are adapted to the multiresolution decomposition, the latter can only be equipped with a dezooming operator, as explained in subsection \ref{subsec:multiresolution-analysis}. Hence, there is no natural operation that, in conjunction with translations, would fully encode the structure of any wavelet basis associated to it. Our wavelet basis possesses however a particular structure that stems from its generative algorithm. It is encoded in two relations that show how to obtain a wavelet chain $x_{\tau}$ from the wavelet chains $x_{\tau^{\prime}}$ of lower scales. The first relation encode the links between wavelet chains indexed by one cycle. It uses a recursive structure on cycles, given by the following lemma, the proof of which is only technical and left in appendix.

\begin{lemma}
\label{cycles-combinatorics}
Let $A = \{a_{1}, \dots, a_{k}\}\subset\n$ with $k\in\{1, \dots, n-1\}$ and $b\not\in A$. Then
\begin{enumerate}
	\item $(a_{1}\, \dots\, a_{k})(a_{j}\, b) = (a_{1}\,\dots\,a_{j}\,b\,a_{j+1}\,\dots\,a_{k}) \quad$ for $j\in\{1, \dots, k\}$,
	\item $\Cycle(A\cup\{b\}) = \{ \gamma\cdot(a_{j}\, b) \;\vert\; j\in\{ 1, \dots, k\}, \gamma\in\Cycle(A)\}$.
\end{enumerate}
\end{lemma}

This lemma means that the set of cycles with support $A\cup\{b\}$ can be obtained recursively from the cycles with support $A$ by inserting $b$ in each cycle $\gamma\in\Cycle(A)$ to the right of an element $a_{j}$ of this cycle. This can be represented by a tree.

\begin{example}
Cycles with support $\{1,2,3,4\}$ are obtained via the following tree.
\begin{center}
\includegraphics[scale=0.5]{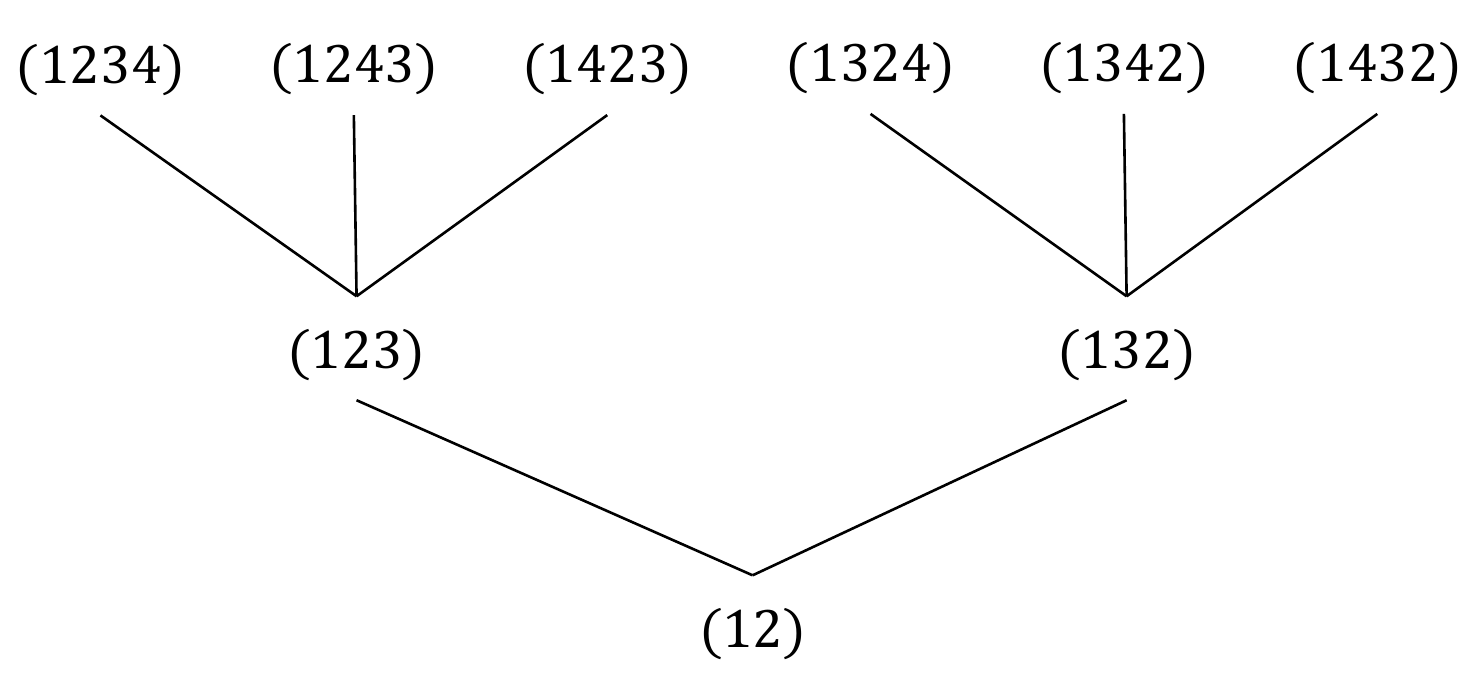}
\end{center}
\end{example}

For $a,b\in\n$, we define the elementary chain $\varepsilon_{b,a}\in L(\Gamma_{n})$ by
\begin{equation*}
\varepsilon_{b,a}(\pi) = \left\{
\begin{aligned}
1 &\qquad\text{if } a,b\in c(\pi) \text{ and } \pi(b) - \pi(a) = 1,\\
-1 &\qquad\text{if } a,b\in c(\pi) \text{ and } \pi(b) - \pi(a) = -1,\\
0 &\qquad\text{otherwise}.
\end{aligned}
\right.
\end{equation*}

\begin{theorem}
\label{cycle-recurrence}
Let $\gamma = (a_{1}\, \dots\, a_{k})$, $A = \supp(\gamma)$, $b > \max A$ and $j\in\{1, \dots, k\}$. Then for all $\pi\in\Rank{A\cup\{b\}}$,
\[
x_{\gamma\cdot (a_{j}\, b)} = \varepsilon_{b, a_{j}}(\pi)\, x_{\gamma}(\pi\setminus\{b\}).
\]
\end{theorem}

\begin{proof}
By lemma \ref{cycles-combinatorics}, $\gamma\cdot (a_{j}\, b) = (a_{1}\,...\,a_{j}\,b\,a_{j+1}\,...\,a_{k})$. Since $b > \max A$, applying algorithm \ref{basis generation} to $\gamma\cdot (a_{j}\, b)$ gives
\begin{align*}
x_{\gamma\cdot (a_{j}\, b)} &= x_{(a_{1}\,...\,a_{j}\,b\,a_{j+1}\,...\,a_{k})} \\
							&= ...\diamond...(a_{j}\diamond b)...\diamond... \\
							&= ...\diamond...(a_{j}b -b a_{j})...\diamond... \\
							&= ...\diamond...(a_{j}b)...\diamond... - ...\diamond...(b a_{j})...\diamond...
\end{align*}
Thus for $\pi\in\Rank{A\cup\{b\}}$,
\begin{align*}
x_{\gamma\cdot (a_{j}\, b)}(\pi) &=
\left\{
\begin{aligned}
x_{\gamma}(\pi\setminus\{b\}) &\qquad\text{if } \pi(b) - \pi(a_{j}) = 1 \\
- x_{\gamma}(\pi\setminus\{b\}) &\qquad\text{if } \pi(b) - \pi(a_{j}) = -1 \\
0 &\qquad\text{otherwise}
\end{aligned}\right. \\
&= \varepsilon_{b, a_{j}}(\pi)\, x_{\gamma}(\pi\setminus\{b\}).
\end{align*}
\end{proof}

\begin{example}
For $A = \{1,2,3,4\}$, for all $\pi\in\Rank{A}$,
\[
x_{(1342)}(\pi) = \varepsilon_{4,3}(\pi)\, x_{(132)}(\pi\setminus\{4\}).
\]
\end{example}

Theorem \ref{cycle-recurrence} leads to an explicit formula for $x_{\gamma}(\pi)$ by a simple induction. It just requires some more notations. For $A\in\mathcal{P}(\n)$, we define a sequence of subsets by $A^{(0)} = A$ and
\[
A^{(j)} = A^{(j-1)}\setminus \{\max A^{(j-1 )}\} \qquad\text{for } j\in\{ 1, \dots, \vert A\vert - 1\}.
\]
If $A = \{a_{1}, \dots, a_{k}\}$ with $a_{1} < \dots < a_{k}$, then $A^{(j)} = \{a_{1}, \dots, a_{k-j}\}$. It is easy to see that for any $\gamma\in\Cycle(A)$, there exists a unique $(u_{1}, \dots, u_{k-1})\in A^{(k-1)}\times\dots\times A^{(1)}$, denoted by $\mathbf{u}(\gamma)$,  such that
\[
\gamma = (u_{1}\,a_{2})(u_{2}\,a_{3})\dots(u_{k-1}\,a_{k})
\]
(It is given by $u_{k-1} = \gamma^{-1}(a_{k})$, and $u_{i} = \left[\gamma\,(u_{k-1}\,a_{k})\dots(u_{i+1}\,a_{i+2})\right]^{-1}(a_{i+1})$, for $i\in\{ 1, \dots, k-2\}$).

\begin{corollary}
\label{explicit-recurrence}
Let  $A = \lbrace a_{1}, \dots, a_{k}\rbrace\subset\n$ with $a_{1} <  \dots < a_{k}$, and $\gamma\in\Cycle(A)$. We set $\mathbf{u}(\gamma) = (u_{1}, \dots, u_{k-1})$. Then for all $\pi\in\Rank{A}$,
\[
x_{\gamma}(\pi) = \prod_{j=0}^{k-2}\varepsilon_{a_{k-j},u_{k-j-1}}(\pi_{\vert A^{(j)}}).
\]
\end{corollary}

\begin{example}
For $A = \{1,2,3,4\}$, for all $\pi\in\Rank{A}$,
\[
x_{(1342)}(\pi) = \varepsilon_{4,3}(\pi)\varepsilon_{3,1}(\pi_{\vert\{1,2,3\}})\varepsilon_{2,1}(\pi_{\vert\{1,2\}}).
\]
\end{example}

The second relation that encodes the structure of the wavelet chains gives the link between a wavelet chain indexed by a product of cycles and the wavelet chains indexed by these cycles. It stems directly from the definition of algorithm \ref{basis generation}.

\begin{theorem}
\label{cycle-decomposition}
Let $\tau = \gamma_{1}\dots\gamma_{r} \in\Sn$ written in standard cycle form. Then $x_{\gamma_{1}\dots\gamma_{r}}$ is the concatenation of $x_{\gamma_{1}}$, \dots, $x_{\gamma_{r}}$:
\[
x_{\gamma_{1}\dots\gamma_{r}} = x_{\gamma_{1}}\dots\,x_{\gamma_{r}}.
\]
\end{theorem}

For $\tau = \gamma_{1}\dots\gamma_{r}\in\Sn$ written in standard cycle form, we define the decomposition of a word $\pi\in\Rank{\supp(\tau)}$ associated to the cycle structure of $\tau$ by the tuple of contiguous subwords $(\pi^{1}, \dots, \pi^{r})$ such that $\pi = \pi^{1}...\,\pi^{r}$ and $\vert\pi^{i}\vert = l(\gamma_{i})$ for all $i\in\{1, \dots, r\}$. The explicit version of theorem \ref{cycle-decomposition} is given by the following corollary.

\begin{corollary}
\label{explicit-decomposition}
Let $\tau = \gamma_{1}\dots\gamma_{r}\in\Sn$ written in standard cycle form, and $\pi\in\Rank{\supp(\tau)}$. Let $\left(\pi^{1}, \dots, \pi^{r}\right)$ be the decomposition of $\pi$ associated to the cycle structure of $\tau$. Then
\begin{equation*}
x_{\gamma_{1}\dots\gamma_{r}}(\pi) = \left\{
\begin{aligned}
\prod_{i=1}^{r}x_{\gamma_{i}}\left(\pi^{i}\right) \qquad&\text{if } c(\pi^{i}) = \supp(\gamma_{i}) \text{ for all } i\in\{1, \dots, r\},\\
0 \qquad&\text{otherwise}.
\end{aligned}\right.
\end{equation*}
\end{corollary}

\begin{example}
Let $\tau = (134)(25) = \gamma_{1}\gamma_{2}$. We have $\supp(\gamma_{1}) = \{1,3,4\}$ and $\supp(\gamma_{2}) = \{2,5\}$. The decomposition of a word $\pi = \pi_{1}...\,\pi_{5}\in\Sym{5}$ associated to the cycle structure of $\tau$ is given by $\pi^{1} = \pi_{1}\pi_{2}\pi_{3}$ and $\pi^{2} = \pi_{4}\pi_{5}$.
\begin{itemize}
	\item For $\pi = 24351$, $(c(\pi^{1}), c(\pi^{2})) = (\{2,3,4\}, \{1,5\}) \neq (\supp(\gamma_{1}),\supp(\gamma_{2}))$, so  \[ x_{(134)(25)}(24351) = 0. \]
	\item For $\pi = 41352$, $(c(\pi^{1}), c(\pi^{2})) = (\{1,3,4\}, \{2,5\}) = (\supp(\gamma_{1}), \supp(\gamma_{2}))$, so  \[ x_{(134)(25)}(41352) = x_{(134)}(413)x_{(25)}(52). \]
\end{itemize}
\end{example}

The two relations given by theorems \ref{cycle-recurrence} and \ref{cycle-decomposition} encode the full structure of the wavelet basis, and allow to compute recursively any wavelet chain, from the wavelet chains of scale $2$. We do not have an analogous concept of the ``mother'' wavelet in our case because the operations involved in the computation of a wavelet chain vary at each stage, but these relations remain the base for many applications, such as the design of fast decomposition algorithms in the wavelet basis.

\section{Conclusion and perspectives}
\label{sec:perspectives}

Exploiting the powerful formalism of injective words, we developed the first general framework to perform data analysis on incomplete rankings in the present paper. Its cornerstone is the multiresolution decomposition of $L(\Sn)$ in function of the spaces $W_{A}$, that provides a decomposition of the space $\mathbb{M}_{\mathcal{A}}$ for any observation design $\mathcal{A}$. The explicit wavelet basis $\Psi$ adapted to this multiresolution decomposition is the key to use this framework in practice, allowing to perform linear or nonlinear approximation in any space $\mathbb{M}_{\mathcal{A}}$. It paves the way for many statistical applications, such as estimation of a ranking distribution or prediction of a ranking on a new subset of items, aggregation of many incomplete rankings into one full ranking or clustering of incomplete rankings. All these applications require the design of fast decomposition algorithms as well as the theoretical study of the properties of the wavelet basis regarding (nonlinear) approximation. This will be the subject of forthcoming articles. At last, another line of further research consists in trying to generalize the present framework to incomplete rankings which also allow ties.

\bibliographystyle{plain}

\bibliography{biblio}

\section{Appendix}

\subsection{Background on group theory}

A group is a set $G$ equipped with an associative operation $G^{2}\rightarrow G, (g,h) \mapsto gh$ and an element $e\in G$ such that for all $g\in G$, $ge = eg = g$ and there exists  $g^{-1}\in G$ such that $gg^{-1} = g^{-1}g = e$. The element $e$ is called the identity element, and $g^{-1}$, necessarily unique, is called the inverse of $g\in G$. The operation is not necessarily commutative. A subgroup of $G$ is a subset $H\subset G$ such that $e\in H$ and for all $(h,h^{\prime})\in H^{2}$, $hh^{\prime}\in H$. A left coset of a subgroup $H$ of $G$ is a subset (usually not a subgroup) of $G$ of the form $\{gh \;\vert\; h\in H\}$ with $g\in G$. A simple result states that for any subgroup $H$ of a finite group $G$, all the left cosets of $H$ have same cardinality $\vert H\vert$ and they constitute a partition of $G$.

An action of a group $G$ over a set $E$ is an operation $G\times E \rightarrow E, (g,x) \mapsto g\cdot x$ such that for all $(g,g^{\prime})\in G^{2}$ and $x\in E$, $e\cdot x = x$ and $g^{\prime}\cdot (g\cdot x) = (g^{\prime}g)\cdot x$. For $x\in E$, its orbit under the action of $G$ is the set $O_{x} = \{ g\cdot x \;\vert\; g\in G\}$, and its stabilizer is the subgroup of $G$ $\{ g\in G \;\vert\; g\cdot x = x\}$. A subset of $E$ is an orbit of $G$ if it is equal to a $O_{x}$. The collection of all the orbits of $G$ is a partition of $E$. The action of $G$ on $E$ is called transitive if it has only one orbit ($E$), \textit{i.e.} if for all $x\in E$, $O_{x} = E$.

A representation of a group $G$ is couple $(V,\rho)$ where $V$ is a linear space and $\rho$ a mapping $\rho : G \rightarrow GL(V)$, where $GL(V)$ is the group of invertible linear maps from $V$ to $V$, such that for all $(g,g^{\prime})\in G^{2}$, $\rho(gg^{\prime}) = \rho(g)\rho(g^{\prime})$. We speak indifferently of the representation $(V, \rho)$, the representation $\rho$ or the representation $V$. When $G$ acts transitively on a finite set $E$, there is a canonical representation of $G$ on $L(E)$, called the permutation representation, defined on the Dirac functions by $\rho(g)\delta_{x} = \delta_{g\cdot x}$, for $x\in E$. From an analytical point of view, the operators $\rho(g)$ are exactly the translations operators on $L(E)$ associated to the action of $G$, and besides, for all $f\in L(E)$, $g\in G$ and $x\in E$, $(\rho(g)f)(x) = f(g^{-1}\cdot x)$. When $E = G$, this representation is called the regular representation.

A representation $(V, \rho)$ of $G$ is called irreducible if $V\neq\{0\}$ and there is no subspace $W\subset V$ such that $\rho(g)(W)\subset W$ for all $g\in G$ other than $\{0\}$ and $V$. Two representations $(V_{1}, \rho_{1})$ and $(V_{2}, \rho_{2})$ of a group $G$ are isomorphic if there exists an isomorphism $\phi$ between $V_{1}$ and $V_{2}$ such that $\phi(\rho_{1}(g)v) = \rho_{2}(g)\phi(v)$ for all $g\in G$ and $v\in V$. Irreducible representations of a group are assimilated to their equivalence class of isomorphic representations.

A major result in the representation theory of finite groups is that the number of irreducible representations of a finite group $G$ is finite (actually equal to the number of conjugacy classes of $G$) and that any finite-dimensional representation $V$ of $G$ admits a decomposition as a direct sum of irreducible representations. The number of copies of one irreducible representation in this decomposition is called its multiplicity. The decomposition of the regular representation $L(G)$ involves all the irreducible representations of $G$, each appearing with multiplicity equal to its dimension. If $\text{Irr}(G)$ denotes the set of irreducible representations of $G$, then
\[
L(G) \cong \bigoplus_{W\in\text{Irr}(G)}d_{W}W,
\]
where for $W\in\text{Irr}(G)$, $d_{W} = \dim W$. See \cite{Diaconis1988} for more developments on group representation theory.

\subsection{Technical proofs}

\begin{proof}[Proof of lemma \ref{combinatorics}]
Let $(A,B)\in\mathcal{P}(\n)^{2}$ with $A\subset B$. The permutation group $\Sym{A}$ acts on $\Rank{A}$ and $\Rank{B}$. The mapping $r_{B,A}:\Rank{B}\rightarrow\Rank{A}, \sigma\mapsto\sigma_{\vert A}$, is equivariant for this action, \emph{i.e.}, for any $\tau\in\Sym{A}$ and $\sigma\in\Rank{B}$, $r_{B,A}(\tau\cdot\sigma)= \tau\cdot r_{B,A}(\sigma)$. The action being transitive on $\Rank{A}$, $r_{B,A}$ is surjective. Moreover, for $\pi\in\Rank{A}$, $\Rank{B}(\tau\cdot\pi) = r_{B,A}^{-1}(\{\tau\cdot\pi\}) = \tau\cdot r_{B,A}^{-1}(\{\pi\}) = \tau\cdot\Rank{B}(\pi)$. Consequently $|\Rank{B}(\pi)| = |\Rank{B}(\sigma\cdot\pi)|$, which, combined with $\Rank{B} = \sqcup_{\pi\in\Rank{A}}r_{B,A}^{-1}(\pi)$, gives the sought result.
\end{proof}

\begin{proof}[Proof of lemma \ref{commutation}]
Let $\omega\in\Gamma_{n}$, $a\in\n\setminus c(\omega)$ and $\pi\in\Gamma_{n}$. If $c(\pi)\cap c(\omega) \neq \emptyset$, then also $c(\varrho_{a}\pi)\cap c(\omega) \neq \emptyset$, and both $\mathfrak{i}_{\omega}\pi$ and  $\mathfrak{i}_{\omega}\varrho_{a}\pi$ are equal to $0$ by definition. If $c(\pi)\cap c(\omega) = \emptyset$, then $\mathfrak{i}_{\omega}\pi = \omega\pi$. Since $a\not\in c(\omega)$, it can only be deleted in the word $\omega\pi$ if it is deleted from $\pi$. This means that $\varrho_{a}\omega\pi = \omega\varrho_{a}\pi$, whether $a\in c(\pi)$ or not. We prove identically that $\varrho_{a}\mathfrak{j}_{\omega} = \mathfrak{j}_{\omega}\varrho_{a}$.
\end{proof}

\begin{proof}[Proof of proposition \ref{general wavelets}]
The proof of this proposition is a simple analysis of algorithm \ref{basis generation}. For a cycle $\gamma = (a_{1}\dots a_{k})$, the associated $x_{\gamma}$ is equal to an expression of the form $a_{1}\diamond \dots \diamond a_{k}$ with a particular way to put parentheses. When expanded, this expression gives $2^{k-1}$ terms with sign $+$ or $-$ between them. It could happen that some of the terms are the same and thus add or balance. But actually, for $x\in L(\Gamma(A))$ with $A\subset\n$, $1\leq \vert A\vert\leq n-1$ and $b\in\n\setminus A$, $\supp(x\diamond b) = \{\pi b \;\vert\; \pi\in\supp(x)\}\sqcup\{b\pi \;\vert\; \pi\in\supp(x)\}$. By recursion, we obtain that $\vert\supp(x_{\gamma})\vert = 2^{k-1}$, meaning also that all the terms in the expanded version of $a_{1}\diamond \dots \diamond a_{k}$ are different. Furthermore, for $x\in L(\Gamma(A))$ and $y\in L(\Gamma(B))$ with $A,B \subset\n$, $A,B\neq \emptyset$ and $A\cap B = \emptyset$, we have $\vert\supp(xy)\vert = \vert\supp(x)\vert \vert\supp(y)\vert$. Now, let $\tau = \gamma_{1}\dots\gamma_{r}$ be a permutation written in standard cycle form, with $\gamma_{i} = (a_{i,1}\dots a_{i,k_{i}})$. Then $x_{\tau} = (a_{1,1}\diamond \dots \diamond a_{1,k_{1}})\dots(a_{r,1}\diamond \dots \diamond a_{r,k_{r}})$, and this expression expands in $2^{k_{1}-1}\dots2^{k_{r}-1} = 2^{k-r}$ different terms. This shows both that $\vert\supp(x_{\tau})\vert = 2^{k-r}$ and that $x_{\tau}$ takes its values in $\{-1,0,1\}$. Applying $\phi_{n}$ concludes the proof.
\end{proof}

\begin{proof}[Proof of lemma \ref{cardinality-intersection}]
Let $(\pi,\pi^{\prime})\in(\Gamma_{n})^{2}$, and $\pi^{0}$ be the subword of $\pi$ with content $c(\pi)\cap c(\pi^{\prime})$. We denote by $\vert\pi\vert = k$, $\vert\pi^{\prime}\vert = l$ and $\vert c(\pi)\cap c(\pi^{\prime})\vert = m$. By definition, $\Sn[\pi]\cap\Sn(\pi^{\prime}) = \{\sigma\in\Sn \;\vert\; \sigma$ admits $\pi$ as a contiguous subword and $\pi^{\prime}$ as a subword$\}$. If $\pi^{0}$ is not a contiguous subword of $\pi^{\prime}$, then there exist a subword $\pi^{\ast}$ of $\pi^{0}$ which is a contiguous subword of $\pi$, $a\in c(\pi^{\prime})\setminus c(\pi^{0})$ and $i\in\{2, \dots, m\}$ such that $\pi^{\ast}\triangleleft_{i}a$ is a subword of $\pi^{\prime}$. So if $\sigma\in\Sn[\pi]\cap\Sn(\pi^{\prime})$, $\sigma$ admits \textit{a fortiori} $\pi^{\ast}$ as a contiguous subword and $\pi^{\ast}\triangleleft_{i}a$ as a subword, which is not possible. Hence,  $\vert\Sn[\pi]\cap\Sn(\pi^{\prime})\vert = 0$ in this case. We now assume that $\pi^{0}$ is a contiguous subword of $\pi^{\prime}$. Let $i\in\{1, \dots, l\}$ such that $\pi_{i}^{\prime}\dots\pi_{i+m-1}^{\prime} = \pi^{0}$. Then each element of $\Sn[\pi]\cap\Sn(\pi^{\prime})$ can be seen as a way of filling the blanks denoted by $\underline{\phantom{aaaaa}}$ with all the elements of $\n\setminus (c(\pi)\cup c(\pi^{\prime}))$, in the following figure.
\[
\underline{\phantom{aaaaa}}\ \pi_{1}^{\prime}\ \underline{\phantom{aaaaa}}\ \dots\ \underline{\phantom{aaaaa}}\ \pi_{i-1}^{\prime}\ \underline{\phantom{aaaaa}}\ \pi\ \underline{\phantom{aaaaa}}\ \pi_{i+m}^{\prime}\ \dots \underline{\phantom{aaaaa}}\ \pi_{l}^{\prime}\ \underline{\phantom{aaaaa}}.
\]
If we do not take the order of the elements into account, the number of such fillings is equal to the number of ways of putting $n-(k+l-m)$ indistinguishable balls (the elements of $\n\setminus (c(\pi)\cup c(\pi^{\prime}))$) into $l-m+2$ boxes (the blanks). From a classic result in combinatorics, this number is equal to
\[
\binom{(n-(k+l-m))+(l-m+2)-1}{(l-m+2)-1} = \binom{n-k+1}{l-m+1}.
\]
Now, to take the order into account, we have to multiply by the number of possible reorderings of the filling elements, equal to $(n-(k+l-m))!$. The final result is thus
\[
\binom{n-k+1}{l-m+1}(n-(k+l-m))! = \frac{(n-k+1)!}{(l-m+1)!}.
\]
\end{proof}

\begin{proof}[Proof of lemma \ref{cycles-combinatorics}]
Let $A = \{a_{1}, \dots, a_{k}\}\subset\n$ with $k\in\{1, \dots, n-1\}$ and $b\not\in A$. For $j\in\{1, \dots, k\}$, $\gamma = (a_{1}\, \dots\, a_{k})$  and $\tau = (a_{1}\, \dots\, a_{k})(a_{j}\, b)$,  we have
\begin{align*}
\tau(a_{i}) &= \gamma(a_{i}) = a_{i+1} \qquad\text{for } i\in\{1, \dots, k\}\setminus\{j\} \text{ with } a_{k+1} = a_{1} \text{ by convention,} \\
\tau(a_{j}) &= \gamma(b) = b,\\
\tau(b) &= \gamma(a_{j}) = a_{j+1},\\
\tau (a^{\prime}) &= \gamma(a^{\prime}) = a^{\prime} \ \ \ \qquad\text{ for all } a^{\prime}\not\in A\cup\{b\}.
\end{align*}
Hence $\tau = (a_{1}\,\dots\,a_{j}\,b\,a_{j+1}\,\dots\,a_{k})$. This proves $1.$ and at the same time that $\{ \gamma\cdot(a_{j}\, b) \;\vert\; j\in\{ 1, \dots, k\}, \gamma\in\Cycle(A)\} \subset \Cycle(A\cup\{b\})$. Now, let $\gamma\in\Cycle(A\cup\{b\})$, $a^{\ast} = \gamma^{-1}(b)$ and $\gamma^{\prime}\in\Cycle(A)$ be the cycle obtained when deleting $b$ in $\gamma$. Then by $1.$, $\gamma = \gamma^{\prime}\cdot(a^{\ast}\, b)$. This concludes the proof.
\end{proof}

\end{document}